\documentclass[bj,authoryear,noshowframe]{imsart}
\usepackage{comment}

\startlocaldefs
\numberwithin{equation}{section}
\theoremstyle{plain}

\newtheorem{theorem}{Theorem}[section]
\newtheorem{lemma}[theorem]{Lemma}
\newtheorem{proposition}{Proposition}

\theoremstyle{remark}

\newtheorem{fact}{Fact}
\newtheorem{assumption}{Assumption}

\newtheorem{condition}{Condition}
\newtheorem{remark}{Remark}

\makeatletter
\def\l@paragraph{\@tocline{4}{0pt}{1pc}{7pc}{}}
\def\l@subparagraph{\@tocline{5}{0pt}{1pc}{7pc}{}}
\makeatother
\endlocaldefs
\usepackage{mypackage,mymacros,mymacros_two_sample,mymacros2, mymacros_brackets}
\usepackage[overload]{textcase}
\begin{document}
\clearpage  

\pagestyle{plain}

\pagenumbering{arabic}

\begin{frontmatter}
\title{Revisiting the two-sample location shift model with a log-concavity assumption}
\runtitle{Two-sample location shift model}

\begin{aug}
\author[A]{\inits{F.}\fnms{Riddhiman}~\snm{Saha}
}
\author[B]{\inits{S.}\fnms{Priyam}~\snm{Das}
\orcid{https://orcid.org/0000-0003-2384-0486}}
\author[C]{\inits{T.}\fnms{Nilanjana}~\snm{Laha} \thanks{Corresponding author, email: \textcolor{blue}{nlaha@tamu.edu}.}
}
\address[A]{Department of Biostatistics, Harvard. T. H. Chan School of Public Health, Boston, MA, USA}

\address[B]{Department of Biostatistics,
Virginia Commonwealth University, Richmond, VA, USA}

\address[C]{Department of Statistics,
Texas A\&{M} University, College Station, TX, USA}

\end{aug}

\begin{abstract}
 In this paper, we consider the two-sample location shift model, a classic semiparametric model introduced by \cite{stein}. This model is known for its adaptive nature, enabling nonparametric estimation with full parametric efficiency. Existing nonparametric estimators of the location shift often depend on external tuning parameters, which restricts their practical applicability \citep{van2021stein}. We demonstrate that introducing an additional assumption of log-concavity on the underlying density can alleviate the need for tuning parameters. We propose a one step estimator for location shift estimation, utilizing log-concave density estimation techniques to facilitate tuning-free estimation of the efficient influence function. While we employ a truncated version of the one step estimator for theoretical adaptivity, our simulations indicate that the one step estimators perform best with zero truncation, eliminating the need for tuning during practical implementation.
\end{abstract}


\end{frontmatter}


\section{Introduction}
\label{sec:intro}
This paper focuses on the two-sample semiparametric location model, a widely recognized semiparametric model introduced by \cite{stein}. Suppose $X$ and $Y$ are two independent random variables. The two-sample semiparametric location can be described as follows:
\begin{equation}
\label{eq: main model}
    X=\Z_1+\mu,\quad Y=\Z_2+\mu+\Delta,
\end{equation}
where $\mu$ and $\Delta$ are real numbers and $\Z_1$ and $\Z_2$ are mean-zero independent random variables with density $g$ and corresponding distribution function $G$.  Let $\mP_0$  denote the class of  all densities  on the real line $\RR$. We assume that $g$ belongs to the class
\begin{align}
    \label{def: mP 0}
    \mP_0=\bigg\{g\in\mP\ :\ \int xg(x)dx=0,\ \mathcal{I}_g<\infty\bigg\}.
\end{align}
Here $\mathcal I_g$ is the Fisher information for location, which is finite if and only if $g$ is an absolutely continuous density satisfying
\begin{equation}
\label{eq: finitieness of FI}
    \edint\lb\dfrac{g'(x)}{g(x)}\rb^2 g(x)dx<\infty,
\end{equation}
where $g'$ is a weak derivative of $g$   \citep[cf. Theorem $3$ of][]{huber}. In this case $\mathcal{I}_g$ takes the form
\begin{equation*}
\mathcal I_g=\edint\lb\dfrac{g'(x)}{g(x)}\rb^2 g(x)dx.
\end{equation*} 
The constraint $\int xg(x)dx=0$  ensures the identifiability of $\mu$ and $g$ in model \eqref{eq: main model}. Suppose we  observe $m$ and $n$ many i.i.d. copies of $X$ and $Y$, respectively.  The goal is to estimate $\Delta$ in  the presence of the nuisance parameters $\mu$ and $g$.


One specialty of the two-sample location shift model is that it is an adaptive model in the sense of \cite{jonsemi}. To elaborate, even when  $g$ is unknown, it is possible to estimate $\Delta$  with the same asymptotic variance as  the scenario when $g$ is known. 
Thus, it is possible to obtain asymptotically efficient estimators of $\Delta$ that ``adapt" to the unknown infinite dimensional parameter $g$ \footnote{ Here and afterwards, an asymptotically efficient estimator is to be understood in the sense of \cite{vdv}, page 367. In particular, it attains the lowest possible asymptotic variance among the regular estimators; see Section \ref{sec: background} for more details.}. Ever since \cite{stein} introduced the idea of adaptation in this model,  it has spurred a series of subsequent investigations \citep{bhattacharya1967efficient,beran,eden}.  The location-scale version of the model has  been explored by \cite{park,hsieh1995empirical,weiss1970}, where, in addition to the location shift $\Delta$, $Y$ also undergoes a scale shift.

There remain two mainstream approaches for estimating $\Delta$ in the semiparametric model \eqref{eq: main model}. The first approach is based on  rank-based estimators such as the Hodges-Lehmann rank estimator or \citeauthor{kraft1970}'s rank estimator \citep{beran,eden,weiss1970,bhattacharya1967efficient}. The rank-based approach prevailed in early research on this problem until the advent of semiparametric efficiency theory. \citep{van2021stein}. The second approach, which is based on one step estimators, utilizes modern semiparametric efficiency theory \citep{begunhall,bickel1982,park}. 



%

Although the aforementioned approaches offer an adaptive estimator $\widehat\Delta$ of $\Delta$ in the model \eqref{eq: main model}, they share a common drawback: a heavy reliance on external tuning parameters \citep{van2021stein, hogg1974, potgieter2012}. The primary source of these external tuning parameters is the estimation procedure for $g$ and its functionals. Specifically, both the one step estimators and the rank-based estimators depend on estimating the score function $-g'/g$. The existing approaches using the one step estimators and the rank-based estimators typically involve smoothing parameters for estimating $g$ or its functionals. However, the estimation of $\Delta$ is reported to be sensitive to such smoothing parameters \citep{park, potgieter2012}. Some of the aforementioned approaches require additional tuning parameters for estimating $g$ or its functionals. For instance, \cite{beran}'s approach relies on a Fourier basis expansion to estimate the scores $-g'/g$, necessitating a tuning parameter for choosing the number of Fourier basis functions. The score-estimation approach in \cite{eden}, which requires grid-based discretization of functions, introduces tuning parameters for selecting the grid points.

 Moreover, the estimation of $g$ or its functionals is not the sole source of tuning parameters in the above line of work. Some of the  aforementioned approaches  involve   
strategies such as sample splitting \citep{eden} and truncation \citep{park} to improve the estimation performance. These strategies also introduce additional tuning parameters.
In an attempt to potentially alleviate the burden of tuning parameters, \cite{potgieter2012} and \cite{hsieh1995empirical} proposed an alternative approach based on \citeauthor{parzen1979nonparametric}'s (\citeyear{parzen1979nonparametric}, Section 10) formulation of the two-sample location-scale problem. This formulation transforms the problem into a generalized linear model using sample quantiles, eliminating the need to estimate the score $-g'/ g$. However, this approach does not completely eliminate the tuning parameter issue, as it still requires estimating $g$ for constructing confidence intervals, introducing tuning parameters. Additionally, the discretization of the quantile function introduces further tuning parameters.

 
The  existing literature lacks a solid methodology for optimally choosing these tuning parameters \citep{van2021stein}.  In their 2021 memoir of Stein's paper \citep{stein}, \citeauthor{van2021stein} remarks that the practical implementation of such semiparametric estimators remains ``as tricky as in the 1950s" due to the intricacies associated with these tuning parameters. Among the previously mentioned works, only  \cite{potgieter2012} and \cite{park}  discuss data-dependent procedures for tuning parameter selection. \cite{potgieter2012} requires tuning parameters for estimating $g$ and choosing grid points to discretize a quantile function. As stated in their paper, while their tuning procedure prioritizes simplicity and computational feasibility, it does not target the optimal tuning parameters and it is likely to incur a loss of efficiency.

On the other hand, \cite{park} aims to identify the optimal smoothing parameter for estimating $g$ by minimizing the estimated Mean Squared Error (MSE) of $\widehat\Delta$ through a grid search. \cite{park} uses bootstrapping to estimate the MSE of $\widehat\Delta$ for each value of the smoothing parameter. Currently, there are no consistency guarantees for this bootstrap estimator.  
However, in the context of nonparametric regression and density estimation, it has been established that  bootstrap estimators  of the MSE of smoothed estimators can have non-negligible bias \citep{hall1990using, hall2013simple, hardle1988bootstrapping}. Furthermore, there is no guarantee that tuning parameters minimizing the MSE would be optimal in terms of maximizing the asymptotic efficiency. Our numerical experiments in Section \ref{sec: simulations} reveal that the MSE-based procedure may result in suboptimal choices of smoothing parameters. In fact, in our simulation study, the asymptotic efficiency and the MSE of the resulting estimators were comparable to our baseline estimator, i.e., the difference of sample means, which is not asymptotically efficient for the two-sample location shift model.

Some of the tuning-related difficulties may be unavoidable since  external tuning parameters are essential for estimating $g$ in $\mP_0$. This necessity stems from the fact that the search space for $g$, namely the set $\mP_0$, is too large,  lacking enough structure. 
If one aims to estimate $g$ or its functionals without relying on tuning parameters, additional structural  constraints on $\mathcal{P}_0$ might be necessary. This naturally raises the question: \textit{what restriction  on  $\mathcal{P}_0$ would provide sufficient structure for achieving tuning-free estimation of $g$?} In this paper, we impose the shape restriction of log-concavity on  $g$ to facilitate tuning-free  estimation of this density and its functionals. \vspace{.1cm}\\
\textbf{Log-concavity:} A density is called log-concave if its logarithm is a concave function. We propose this restriction for four main reasons, which are described below.\vspace{.1cm}\\
{\textit{Tuning-free density estimation}:} The class of log-concave densities is structurally rich enough to admit a maximum likelihood estimator (MLE) \citep{2009rufi,exist}. The MLE among the class of all log-concave densities  can be efficiently computed  using existing R packages in a tuning-free way \citep{logcondens}. 
 The log-concavity restriction on $\mathcal{P}_0$ enables us to leverage this MLE for constructing tuning-free estimators of $g$ \citep{walther2009inference,smoothed,dumbreg,xuhigh,adaptive2016,barber2020}.\vspace{.1cm}\\
{\textit{Versatility of log-concave density class}: The class of log-concave densities is an important subclass of unimodal densities, containing a list of well-known and well-studied commonly used unimodal  sub-exponential densities, such as, normal, Laplace,  logistic,  gamma (for shape parameter greater than one), and beta (for parameters greater than one). A more exhaustive list can be found in \cite{bagnoli2005}. Additionally, this class is closed under convolution. \vspace{.1cm}\\
  \textit{Availabilty of tests to check log-concavity}: While visually verifying log-concavity from the data can be challenging, several tests  are available in the literature to check log concavity \citep{gangrade2023sequential, asmussen2017distinguishing, dunn2021universal}. If there are grounds to believe that the underlying density is unimodal but not heavy-tailed, then log-concavity is a very plausible assumption \citep{laha2022improved,walther2009inference}. 
    \vspace{.1cm}\\
  \textit{Application in semiparametric models:} Shape constraints related to convexity have already been considered in other semiparametric models such as the one-sample location model \citep[log-concavity,][]{laha2021adaptive}, the single index model \citep[convexity,][]{kuchibhotla}, and the mixture model \citep[log-concavity, ][]{walther2002detecting,walther2009inference}.
  The assumption of log-concavity has also been applied to our two-sample location shift model. \cite{eden}, who developed an adaptive estimator of $\Delta$ in this model, assumed the function $u \mapsto -g'(G^{-1}(u))/g(G^{-1}(u))$ to be non-increasing, which is equivalent to $g$ being log-concave \citep[][Proposition A.1]{bobkov1996}. However, \cite{eden} did not use the term "log-concave" because the literature on log-concave density estimation had not yet developed by their  time. Therefore, \cite{eden}'s approach fundamentally differs from our  approach, which leverages modern tools for log-concave density estimation. 
  \subsection{Main contribution}
  In this paper, the proposed estimator  of $\Delta$  is built upon the one step estimator advocated by \cite{park} and \cite{bickel1982}. As mentioned earlier, a common critique of this estimator has been reliant on the estimation of $g'$ \citep{potgieter2012,hsieh1995empirical}. However, this concern is alleviated under the log-concavity assumption on $\mP_0$, since log-concave MLE-based estimators of $g$ readily provide an estimator of $g'$ \citep{laha2021adaptive}. We adopt the smoothed log-concave MLE technique proposed by \cite{2009rufi,chen2016}. This technique yields  smooth estimators of $g$ and $g'$ without requiring any external tuning parameter for smoothing. 
  
  
  In order to establish the asymptotic efficiency, we use a truncated version of the one step estimator. It is established that the truncated one step estimator is adaptive, provided the truncation level decays to zero (Theorem \ref{theorem: main theorem specific} in Section \ref{sec: asymptotic properties}). This implies that after centering and scaling, the aforementioned estimator is asymptotically normal with variance $\mathcal{I}^{-1}$. Much of the theoretical analysis in this paper is dedicated to proving the adaptivity of the truncated one step estimator. This  proof necessitates a sharp theoretical analysis of the performance of the one step estimators, with intricacies relying on the  log-concave projection theory  and empirical processes \citep{dumbreg,barber2020,review}. Although it seems that the truncation level of the truncated one step estimators is an external tuning parameter, we emphasize that the truncation is introduced primarily for technical reasons in the proof. Our simulations demonstrate that the one step estimators exhibit maximum efficiency and minimum MSE when the truncation level $\eta_{m,n}$ is set to zero. Therefore, for practical implementation of our proposed estimators, we recommend a truncation level of zero. Hence, \textit{no tuning is required} for practical implementation of our estimators. The proposed estimator can be easily implemented using the provided R package \texttt{TSL.logconcave}, made available on Github \citep{TSL}.

The rest of the paper is organized as follows: Section \ref{sec: Preliminaries} presents the mathematical formulation. In Section \ref{sec: method}, we delve into the derivation of the one step estimator for the proposed problem, and we provide an outline of the estimation procedure relying on the log-concave MLE. Section \ref{sec: asymptotic properties} establishes the asymptotic efficiency of the truncated versions of the one step estimator. Section \ref{sec: simulations} conducts extensive simulations to compare our one step estimators (both truncated and untruncated) with existing semiparametric estimators. The detailed proof of Theorem \ref{theorem: main theorem specific} is provided in the Appendix.

\section{Mathematical formulation}
\label{sec: Preliminaries}

We assume that $X_1,\ldots, X_m$ and $Y_1,\ldots, Y_n$ are independent and identically distributed (i.i.d.) draws from $f_0=g_0(\cdot-\mu_0)$ and $h_0=g_0(\cdot-\mu_0-\Delta_0)$, respectively; where $g_0$ is a member of $\mP_{0}.$ We assume $\min(m,n)\to\infty$ and set $N$ to be $m+n$.  Let us denote the distribution functions pertaining  to the densities $f_0$, $h_0$ and $g_0$ by $F_0$, $H_0$ and $G_0$ respectively. The empirical distribution functions corresponding to the $X_i$'s and the $Y_j$'s will be denoted by $\Fm$ and $\Hn$, respectively. Additionally, we  denote the respective sample averages of the $X$ and $Y$ samples by $\bX$ and $\bY$.

Now we formally define the class of log-concave densities, denoted by  $\mathcal{LC}$.  Let us define
 \[\mathcal C:=\lbs \phi:\RR\mapsto[-\infty,\infty) : \phi \text{ is concave, closed, and proper}\rbs.\]
 Here, the concepts of properness and closedness for a concave function adhere to the standard definitions, cf. pages 24 and 50 of \cite{rockafellar}.  We follow the convention that all concave functions  are
defined  across the entire real line $\RR$, and take the value $-\infty$ outside the effective domain $\dom(\phi)=\{x\in\RR:\phi(x)>-\infty\}$ \citep[page 40,][]{rockafellar}. Let $\mathcal{LC}$ be the collection of all densities on $\RR$ such that  $\log g \in \mathcal{C}$. This definition of log-concavity aligns with prior literature \citep{dosssymmetric,doss2019,laha2021adaptive}. Given $g_0$ is an element of the class $\mathcal{LC}$, it follows that $\ps_0 = \log g_0$ is concave, and both $f_0$ and $g_0$ are members of the set $\mathcal{LC}$ as well. The main theoretical result of this paper assumes that $g_0\in\mathcal{LC}\cap\mathcal{P}_0$. \vspace{1em}\\
\underline{\textit{Examples:}}  Examples of  $g_0\in\mathcal{LC}\cap\mathcal{P}_0$ include, but are not limited to, the standard Gaussian density, standard Laplace density,   standard logistic density, centered gamma density with shape parameter greater than two, centered beta density with parameters greater than two, centered Gumbel density, and centered  Weibull density with parameter greater than two. We can show that the above-mentioned densities are all log-concave with finite Fisher information for location.


\subsection{Notation}
\label{sec: notation}

For a sequence of random variables $\{X_n\}_{n\geq 1}$ and a fixed random variable $X$, $X_n\P X$ indicates that $X_n$ converges to $X$ in probability. For a sequence of distribution functions $\{F_n\}_{n\geq 1}$, we say that $F_n$ converges weakly to $F$ and denote it as $F_n \overset{d}{\rightarrow} F$ if, for all bounded continuous functions $h:\RR\mapsto\RR$, we have $\lim\limits_{n\to\infty}\int h dF_n=\int hdF$. 
  The Hellinger distance $\H(f_1,f_2)$ between two densities $f_1$ and $f_2$ is defined as
\[\H^2(f_1,f_2)=\frac{1}{2}\edint (\sqrt f_1(x)-\sqrt f_2(x))^2 dx.\]
Throughout this paper, the standard Gaussian density is denoted by $\varphi$.

Unless otherwise mentioned, for a real-valued function $h$, provided they exist,  $h'$ and $h'(\mathord{\cdot}-)$ refer to the right and left derivatives of $h$, respectively. The support of any density $f$ on $\RR$ is denoted by $\text{supp}(f) = \{x\in\RR: f(x)>0\}$. The interior of any set $A$ is denoted by $\text{int}(A)$. For any probability distribution $P$, $L_2(P)$ denotes the space of all function $h$ with $\int h^2dP<\infty$. We equip $L_2(P)$ with the canonical inner product  $\langle \cdot,\cdot\rangle_{L_2(P)}$. As usual, $\mathbb{N}$ is used to denote the set of natural numbers. Throughout this paper, $C$ denotes an arbitrary constant, whose value might vary 
from line to line. The expression $x\lesssim y$ implies that there exists an absolute constant $C>0$, such that $x\leq Cy$.

\section{Method}
\label{sec: method}

The one step estimator is derived from the efficient influence function, which is a key concept in semiparametric efficiency theory. This function characterizes the sensitivity of an estimator to small perturbations in the data. As mentioned in Section \ref{sec:intro}, the two-sample location shift model is a well-studied semiparametric model \citep{begunhall,bickel1982,park}. 
\cite{park,bickel1982,van2021stein} explore the efficient influence function in the balanced case, i.e., when $m=n$. 
 The balanced case automatically fits into the mold of the one-sample semiparametric theory since, in this case, the problem can be represented as a one-sample problem with bivariate data. We consider the potentially unbalanced case, which demands extra care because, unlike the balanced case,  this scenario does not automatically fit the mold of  one-sample semiparametric theory.

We derive The efficient influence function and the one step estimator are derived for the unbalanced case in Section \ref{sec: derivation}. The efficient influence function for estimating $\Delta$ can be derived working along the lines of \cite{begunhall}. Although our approach shares similarities with \cite{begunhall}, we introduce  essential minor deviations  to ensure its alignment with contemporary semiparametric literature such as \cite{park} and \cite{van2021stein}. While \cite{begunhall} used Hellinger differentiability to define scores, our definitions adhere to the quadratic-mean-differentiability-based convention \citep[][page 362]{vdv}, ensuring consistency with modern semiparametric approaches. Also, the underlying Hilbert space considered in our analysis is $L_2(\mathbb{\PP})$, which differs from \cite{begunhall}, where $L_2$-spaces of Lebesgue-like dominating measures were considered.  If preferred, readers may choose to skip the derivation in Section \ref{sec: derivation} and proceed directly to Section \ref{sec: shape constrained}. Section \ref{sec: shape constrained} uses the derived formula of the one step estimator (as described in Section \ref{sec: derivation}) to construct an estimator of $\Delta$ utilizing the additional log-concavity assumption.

\subsection{Derivation of the one step estimator}
\label{sec: derivation}
 



This section is organized as follows. Section \ref{sec: background} provides a review of essential semiparametric concepts and introduces relevant terminology. Section \ref{sec: the model} is dedicated to the derivation of the efficient influence function for estimating $\Delta$. Following that, Section \ref{sec: one step estimator} derives the desired one step estimator using the efficient influence function. Throughout Section \ref{sec: derivation}, our focus is mainly concerned on the main model \eqref{eq: main model}, assuming $g\in \mathcal{P}_0$. Later in Remark \ref{remark: LC and mP}, we illustrate that the efficient influence function for estimating $\Delta$ remains unchanged in submodels, where $g$ belongs to $\mathcal{P}_0\cap\mathcal{LC}$ or any other subclass of $\mathcal{P}_0$.

    \subsubsection{Preliminary background}
    \label{sec: background}

In this section, we provide a concise overview of some key concepts in semiparametric efficiency theory, which play a crucial role in the subsequent derivations. Readers who are already familiar with this topic may choose to skip this section and proceed to Section \ref{sec: the model}. 

Let $\mathcal V$ be an infinite dimensional set and let $k$ be a positive integer. 
Consider the problem of estimating a finite-dimensional parameter $\theta \in \RR$, in the presence of a finite-dimensional nuisance parameter $v \in \RR^k$ and an infinite-dimensional nuisance parameter $\zeta \in \mathcal{V}$. We assume that the underlying probability distribution $\mathbb{\PP}_{\theta,v,\zeta}$ is entirely determined by $\theta$, $v$, and $\zeta$. Furthermore, suppose $\mathbb{\PP}_{\theta,v,\zeta}$ is a probability measure defined on a measurable space $(\mathcal{U}, \mathcal{B})$, where $\mathcal{U}$ is an Euclidean space, and $\mathcal{B}$ is the corresponding Borel sigma field. To formalize, we let $\mS$  be the   semiparametric model parameterized by  $(\theta,v,\zeta)\mapsto \mathbb P_{\theta,v,\zeta}$, i.e., $\mS$ represents the class of distributions of the form $\mathbb{\PP}_{\theta,v,\zeta}$. Finally, let $\{U_i\}_{i=1}^n$ denote a sequence of $n$ i.i.d. random realizations taking value in $\mathcal U$  from the distribution $\PP_{\theta,v,\zeta} \in \mS$.

The score functions corresponding to $\theta$, $v$, and $\zeta$ are defined following \cite{vdv}, page 372, which we elaborate below. The score functions for the finite-dimensional parameters $\theta$ and $v$ are represented by the usual partial derivatives of the log-likelihood function and are denoted as $l_\theta$ and $l_v$, respectively.  However, to define the nuisance score corresponding to $\zeta$,  we require the proper concept of perturbation for the infinite dimensional parameter $\zeta$.   For this purpose, we introduce the tangent space of $\mathcal V$ at $\zeta$, denoted as $\dot{\mathcal V}(\zeta)$. The tangent space $\dot{\mathcal V}(\zeta)$ is the closed linear span of the tangent set of $\mathcal V$ at $\zeta$. This tangent set is the set of perturbations $b$ by which $\zeta$ can be approximated within $\mathcal V$.  Each perturbation $b$ represents a finite-dimensional submodel $t\mapsto \zeta_t$ of $\mathcal V$, taking the form $\zeta_t=(1+tb)\zeta$ for small values of $t$. These submodels pass through $\zeta$ at $t=0$.  See page 362 of \cite{vdv} for more details on these finite dimensional submodels. 

Let  $\nf(x;\theta,v,\zeta)$ be the density of $\PP_{\theta,v,\zeta}$ at $x$.  Then for any  $b\in\dot{\mathcal V}(\zeta)$, 
the score  corresponding to the perturbation $t\mapsto \zeta_t$ at $\PP_{\theta,v,\zeta}$ is the function $B_b:\mathcal U\mapsto\RR$ given by 
\begin{align}
 \label{def: score space}   
  B_b(u;\theta,v,\zeta)=\pdv{\log \nf(x;\theta,v,\zeta_t)}{t}\bl_{t=0}\quad\text{ for all }u\in\mathcal U.
\end{align}
By $\dot{\mathcal S}(\PP_{\theta,v,\zeta})$, we denote the closed linear span of all such scores $B_b$'s when $b$ ranges across the tangent space $\dot{\mathcal V}(\zeta)$. The space $\dot{\mathcal S}(\PP_{\theta,v,\zeta})$ represents the score space associated with the infinite-dimensional nuisance parameter $\zeta$. Importantly, this space becomes a Hilbert space when equipped with the $L_2(\PP_{\theta,v,\zeta})$ inner product \citep[cf. pp. 372 of][]{vdv}.

The nuisance score space is the closed linear span of all nuisance scores.
The efficient score function of $\theta$ in the model $\mS$ is the orthogonal projection of its score function $l_\theta$ onto the orthocomplement of the nuisance score space. Thus, for fixed values of $\theta$, $v$, and $\zeta$, the efficient score function $l_{\theta,\text{eff}}(\cdot\mid v,\zeta)$ is a function from $\mathcal U$ to $\RR$. The efficient score function is an important quantity for us because it is closely related to the efficient influence function.
 \cite{begunhall} gives a general recipe for deriving the efficient score function of $\theta$ \citep[see also pp. 74 of][]{jonsemi}.    We first project $l_\theta$, the score function of $\theta$, onto the orthocomplement of the closed linear span of the score function of $v$, denoted as $l_v$.  We will refer to the resulting  projection as $l_{\theta\mid v}$, and it would represent the efficient score of estimating $\theta$, had $\zeta$ been known. 
 To obtain the efficient score function for estimating $\theta$ in $\mS$, we   need to find the orthogonal projection of $l_{\theta\mid v}$ onto the orthocomplement of $\dot{\mathcal S}(\PP_{\theta,v,\zeta})$, the  nuisance score space for $\zeta$. For all projection-related operations, we take the underlying Hilbert space to be  $L_2(\PP_{\theta,v,\zeta})$. 



 The Fisher information $\III$ for estimating $\theta$ in $\mS$ is formally defined as the squared expectation of the efficient score function, which can be expressed as $\III = \E[l_{\theta,\text{eff}}(U\mid v,\zeta)^2]$, with the expectation taken under the distribution $\PP_{\theta,v,\zeta}$. In semiparametric efficiency theory, the Fisher information plays a crucial role because its inverse serves as the lower bound on the asymptotic variance of any regular estimator of $\theta$  \citep[Theorem 25.21,][]{vdv}. An estimator $\widehat\theta_n$ of $\theta$ in $\mS$ will be called  asymptotically efficient if it is regular and $\sqn(\widehat\theta_n-\th_0)$ is asymptotically normal with asymptotic variance $\III^{-1}$ \citep[pp. 367][]{vdv}.

In a special case where the score $l_{\theta\mid v}$ is orthogonal to $\dot{\mathcal S}(\PP_{\theta,v,\zeta})$, the efficient score $l_{\theta,\text{eff}}(\cdot\mid v,\zeta)$ becomes equal to $l_{\theta\mid v}$. In this scenario,    the Fisher information for estimating $\theta$ in  $\mS$ equals $\E[ l_{\theta|v}(U)^2]$, which is the Fisher information for estimating $\theta$ in the parametric submodel of $\mS$ where $\zeta$ is known. This phenomenon is referred to as adaptation \citep{begunhall,vdv,bickel1982}. 
Consequently, in adaptive models,  asymptotically efficient estimators of $\Delta$ attain the parametric Cramer Rao lower bound, which is the lowest possible asymptotic variance among the regular estimators in the corresponding parametric model. This implies that in adaptive semiparametric models, not having the knowledge of $\zeta$ does not result in any loss of information when estimating $\theta$.

The efficient influence function is the scaled version of the efficient score and is   defined as follows:   
\begin{align}
    \label{def: efficient influence function}
    \tilde l_{\theta}(u\mid v,\zeta)=\frac{l_{\theta,\text{eff}}(u\mid v,\zeta)}{\III} \quad \text{for all }u\in\mathcal U.
\end{align}
The efficient influence function plays a key role in the theory of asymptotaically efficient estimators because  any  estimator $ \widehat\theta_n$ of $\theta$
 is asymptotically efficient if and only if it admits the following asymptotic expansion \citep[see page 367,][]{vdv}:
\begin{equation}
   \label{expansin: influence function}
  \widehat\theta_n=  \theta+\frac{1}{n}\sum_{i=1}^n \tilde l_{\theta}(U_i\mid v,\zeta)+o_p(n^{-1/2}).
\end{equation}


\subsubsection{Estimation of $\Delta$ in the two-sample location shift model}
\label{sec: the model}
As mentioned earlier, when the sample sizes are unequal,  two-sample models do not directly fit into the framework of the conventional one-sample-based semiparametric framework elucidated above. Therefore, to fit the model in \eqref{eq: main model} into a one-sample mold, following \cite{begunhall},  we introduce an indicator variable $Z$, which indicates  the source sample for each observation. We  assume that we observe $m+n$ independent  replicates of  $(\U,Z)$, where $Z\sim \text{Bernoulli}(\lambda)$. The random variable $\U$ is distributed as $X$ when $Z=1$, and it is distributed as $Y$ when $Z=0$. Consequently, conditional on $Z=z$, the random variable $\U$ has the density
\begin{align*}
   \nf(u\mid z)=\begin{cases}
       g(u-\mu) &\ \text{if } z=0\\
       g(u-\mu-\Delta) & \text{ if } z=1,\quad \text{for all }u\in\RR,
   \end{cases}
\end{align*}
where $g\in\mP_0$. Letting $N=m+n$, we can present our observed data as
\begin{align*}
    (Z_i,U_i)=\begin{cases}
        (1,X_i) & \text{if }i=1,\ldots,m\\
        (0,Y_{i-m}) & \text{if }i=m+1,\ldots,N.
    \end{cases}
\end{align*}

 The joint density  of $(Z,\U)$ writes as 
\[\nf(z,u;\Delta,\mu,\lambda,g)\equiv \nf(z,u)=\slb\lambda g(u-\mu)\srb^z\slb (1-\lambda)g(u-\mu-\Delta)\srb^{1-z}\]
for all $z\in\{0,1\}$ and $u\in\RR$.
Note that, now this model fits into the framework elucidated in Section \ref{sec: background}. In this context, the finite-dimensional parameter of interest, i.e., $\theta$ in Section \ref{sec: background}, corresponds to $\Delta$. The finite-dimensional nuisance parameters, denoted as $v$ in Section \ref{sec: background}, correspond to $(\mu, \lambda)$, while the infinite-dimensional nuisance parameter $\zeta$ corresponds to $g$.

Letting $\PP_{\Delta,\mu,\lambda,g}$ denote 
the probability measure associated with the density $\nf(\cdot;\Delta,\mu,\lambda,g)$, 
we formally  define our model $\mS$ by 
\[\mS:=\lbs \PP_{\Delta,\mu,\lambda,g}: \Delta,\mu\in\RR,\quad  \lambda\in(0,1),\quad g\in\mP_0\rbs.\]
Observe that $\PP_{\Delta,\mu,\lambda,g}$ corresponds to the probability distribution $\PP_{\theta,v,g}$ in Section \ref{sec: background}. 
Thereafter, the score functions corresponding to $\Delta$, $\mu$, and $\lambda$ are denoted by $l_\Delta$, $\l_\mu$, and $l_\lambda$, respectively.\vspace{1em}\\
\textit{\underline{Calculation of the scores}:}
To compute the efficient influence function for estimating $\Delta$, the first step is to calculate the score functions. Let $g'$ be a weak derivative of $g$, which exists because $g \in \mP_0$ is absolutely continuous (see equation \ref{eq: finitieness of FI}). 
Denoting $\psi=g'/g$, we calculate the score functions for $\Delta$, $\mu$, and $\lambda$ as follows:
\begin{gather*}
    l_\Delta(z,u;\mu,\lambda,g):=\pdv{\log \nf(z,u;\mu,\lambda,g)}{\Delta}=-(1-z)\psi'(u-\mu-\Delta),\\
    l_\mu(z,u;\Delta,\lambda,g)=\pdv{\log \nf(z,u;\Delta,\mu,\lambda,g)}{\mu}=-z\psi'(u-\mu)-(1-z)\psi'(u-\mu-\Delta),\\
    l_\lambda(z,u;\Delta,\mu,g)=\pdv{\log \nf(z,u;\Delta,\mu,\lambda,g)}{z}=\frac{z}{\lambda}-\frac{1-z}{1-\lambda}.
\end{gather*}
The Fisher information matrix corresponding to the parameters $(\Delta,\mu,\lambda)$ \citep[cf. page 13 of][]{bickel1982} takes the form 
\begin{align}
\label{def: Fisher information matrix}
    I\equiv I(\Delta,\mu,\lambda,f)=\begin{bmatrix}
  (1-\lambda)\Igg & (1-\lambda)\Igg & 0\\
  (1-\lambda)\Igg & \Igg & 0\\
  0 & 0 & \lambda^{-1}+(1-\lambda)^{-1}\\
    \end{bmatrix}.
\end{align}
As mentioned in Section \ref{sec: background}, to compute the score corresponding to the infinite-dimensional nuisance parameter $g\in\mP_0$, we first need to obtain the tangent space of $\mP_0$ at $g$.  From \cite{begunhall}, it follows that for any $g\in\mP_0$, this tangent space has the form 
\begin{align}
    \label{def: tangent set}
    \dot{\mP}_0(g)=\lbs b\in L_2(G): \edint b(x)g(x)dx=0,\ \edint xb(x)g(x)dx=0\rbs.
\end{align}

Equation \ref{def: score space} shows that for any $b\in\dot{\mP}_0(g)$,    the score corresponding to the perturbation $g_t:t\mapsto (1+tb)g$ at $\PP_{\Delta,\mu,\lambda,g}$ is the function $B_b:\{0,1\}\times\RR\to\RR$. In our case, $B_b$ takes the form
\begin{align}
   \label{def: nisance score operator}
   B_b(z,u;\Delta,\mu,\lambda,g)= \pdv{\log \nf(z,u;\Delta,\mu,\lambda,g_t)}{t}\bl_{t=0}=zb(u-\mu)+(1-z)b(u-\mu-\Delta).
\end{align}
Hence, the score-space corresponding to the nuisance parameter $g$ at $\mathbb P_{\Delta,\mu,\lambda,g}$ is the closed linear span of all functions of the form $B_b$, where $b$ ranges over $\dot{\mP}_0(g)$. We denote this space as $\dot{\mS}(\PP_{\Delta,\mu,\lambda,g})$.\vspace{1em}\\
\textit{\underline{The efficient influence function}:}
Now we proceed to calculate the efficient score function for estimating $\Delta$ in the model $\mS$. 
Let us define
\[I_2=\begin{bmatrix}
    I_{22} & I_{23}\\
    I_{32} & I_{33}
\end{bmatrix}\quad \text{and}\quad I_{1.2}= \begin{bmatrix}
    I_{12} & I_{13}
\end{bmatrix},
    \]
    where $I$ is as defined in \eqref{def: Fisher information matrix}. 
  The $L_2(\PP_{\Delta,\mu,\lambda,g})$ projection of $l_\Delta$ onto the orthocomplement of the closed  linear span of  $l_\mu$ and $l_\lambda$ is given by   \citep[see Proposition 1 of][page 30]{bickel1982} 
  \begin{align*}
l_{\Delta|\mu,\lambda}(z,u; g) &= l_\Delta(z,u) - I_2^{-1}I_{1.2}\begin{bmatrix} l_\mu(z,u) \\ l_\lambda(z,u) \end{bmatrix} \\
&= z(1-\lambda)\psi'(u-\mu) - \lambda(1-z)\psi'(u-\mu-\Delta).
\end{align*}
  The score  $l_{\Delta| \mu,\lambda}$ is the efficient score in the submodel of $\mS$ where $g$ is known.   The latter submodel is a parametric model because $\mu$ and $\lambda$ are the only nuisance parameters in this model. 
 Our discussion in Section \ref{sec: background} implies that the efficient score for estimating $\Delta$ in $\mathcal S$ is the orthogonal projection of  $l_{\Delta| \mu,\lambda}$ onto the orthocomplement of $\dot{\mS}(\PP_{\Delta,\mu,\lambda,g})$, i.e.,  the score space of $g$. 
However, for all  scores $B_b\in \dot{\mS}(\PP_{\Delta,\mu,\lambda,g})$, we have:
\begin{align*}
& \langle B_b, l_{\Delta| \mu,\lambda}\rangle_{L_2(\PP_{\Delta,\mu,\lambda,g})}= E_{\PP}\slbt \slb Zb(U-\mu)+(1-Z)b(U-\mu-\Delta)\srb l_{\Delta|\mu,\lambda}(Z,U)\srbt=0,
\end{align*}
where for any  probability measure $P$, $E_P$ denotes  the expectation with respect to  $P$.
Thus $l_{\Delta| \mu,\lambda}$ is orthogonal to the score-space $\dot{\mS}(\PP_{\Delta,\mu,\lambda,g})$ corresponding to $g$. Therefore, our discussion in Section \ref{sec: background} implies that the model $\mS$ is adaptive and the efficient score function $l_{\Delta,\text{eff}}$ equals   $l_{\Delta| \mu,\lambda}$.
This result indicates that the Fisher information of estimating $\Delta$ in  $\mS$ equals $I(\Delta\mid \mu,\lambda,g)=\E[ l_{\Delta|\mu,\lambda}(Z,U;g)^2]$. Straightforward algebra shows that this expectation equals $\lambda(1-\lambda)\Igg$. 
Then \eqref{def: efficient influence function} implies that the efficient influence function $ \tilde l_{\Delta}(z,u\mid \mu,\lambda,g)$ for estimating $\Delta$ in $\mS$ takes the form
\begin{align}
    \label{def: efficient influence function main}
    \tilde l_{\Delta}(z,u\mid \mu,\lambda,g)=\frac{l_{\Delta| \mu,\lambda}(z,u;g)}{\lambda(1-\lambda)\Igg}=\frac{z\psi'(u-\mu)}{\lambda\Igg}-\frac{(1-z)\psi'(u-\mu-\Delta)}{(1-\lambda)\Igg}.
\end{align}
Since the Fisher information for estimating $\Delta$ in $\mS$ equals  $\lambda(1-\lambda)\Igg$,  a regular estimator $\widehat\Delta$ of $\Delta$ is asymptotically efficient if 
 \[\sqrt{N}(\widehat\Delta-\Delta)\to_d\N\left(0,\frac{\Igg^{-1}}{\lambda(1-\lambda)}\right).\]
Note that $m/N=\sum_{i=1}^N{Z_i}/N\to_p\lambda$ and $n/N=1-m/N\to_p 1-\lambda$. Therefore,  by Slutskey's theorem, asymptotic efficiency of $\widehat\Delta$ is equivalent to 
 \begin{equation}
     \label{def: asymptotic eff}
     \sqrt{\frac{mn}{N}}(\widehat\Delta-\Delta)\to_d\N\left(0,\Igg^{-1}\right).
 \end{equation}

\begin{remark}
\label{remark: LC and mP}
 Suppose, instead of $g\in\mP_0$, we consider $g\in\mP'\subset\mP_0$ in \eqref{eq: main model}. This leads to the model $\mS'$, which can be represented as $\{\PP_{\Delta,\mu,\lambda,g}: \Delta,\mu\in\RR,\lambda\in(0,1),g\in\mP'\}$. In an extreme case where $\mP'=\{g_0\}$ is a singleton set, $\mS'$ becomes a parametric model. The Fisher information for estimating $\Delta$ in this parametric model is easily checked to be $\lambda(1-\lambda)\I$. Generally, since $\mP'\subset \mP_0$, the finite-dimensional submodels of $\mP'$ passing through $g$ will also lie in $\mP_0$. Thus, if $b$ is a perturbation at $g$ in the submodel $\mP'$, it is also a perturbation with respect to the bigger model $\mP_0$. Therefore, the tangent space resulting from such perturbations will be a subset of $\dot{\mP_0}$, and the corresponding score space $\dot{S}'(\mathbb P_{\Delta,\mu,\lambda,g})\subset \dot{S}(\mathbb P_{\Delta,\mu,\lambda,g})$. Consequently, if $l_{\Delta\mid \mu,\lambda}$ is orthogonal to the larger score space $\dot{S}(\mathbb P_{\Delta,\mu,\lambda,g})$, it is also orthogonal to the subspace $\dot{S}'(\mathbb P_{\Delta,\mu,\lambda,g})$. Hence, adaptation in the larger model $\mS$ implies adaptation in all the submodels of the form $\mS'$. Therefore, the Fisher information for estimating $\Delta$ in any submodel of the form $\mS'$ is also $\lambda(1-\lambda)\Igg$ at $\PP_{\Delta,\mu,\lambda,g}$. The efficient influence function for estimating $\Delta$ also remains the same across all subsets of $\mS'$, given by \eqref{def: efficient influence function main}. In particular, in the special case when $\mP'=\mathcal{LC}\cap\mP_0$, the efficient Fisher information for estimating $\Delta$ in the corresponding submodel $\mS'$ remains $\lambda(1-\lambda)\Igg$, and $\tilde l_{\Delta}$ remains the efficient influence function. Therefore, having prior knowledge of log-concavity does not improve the Fisher information for estimating $\Delta$.
 
\end{remark}

\subsubsection{One step estimator}
\label{sec: one step estimator}

 In this section, we derive the one step estimator for estimating $\Delta$ in $\mS$ when $\Delta=\Delta_0$, $\mu=\mu_0$, and $g=g_0$. 
Since the efficient influence function for estimating $\Delta_0$ has the form in \eqref{def: efficient influence function main}, by equation \ref{expansin: influence function}, an estimator of $\Delta_0$ is asymptotically efficient if and only if it admits  the asymptotic expansion
\begin{equation}
\label{def: one step estimator: general}
\Delta_0+\frac{1}{N}\sum_{i=1}^N\lb \frac{Z_i\psi_0'(U_i-\mu_0)}{\lambda\Io}-\frac{(1-Z_i)\psi_0'(U_i-\mu_0-\Delta_0)}{(1-\lambda)\Io}\rb+o_p(N^{-1/2}).    
\end{equation}

Suppose we have at our disposal $\sqrt{N}$-consistent preliminary estimators  $\overline{\Delta}$ and $\overline{\mu}$ of $\Delta_0$ and $\mu_0$, respectively. Further, suppose we have reasonable preliminary estimators of $g_0$ and $\psi_0'$, denoted by $\hn$ and $\hln$, respectively.
Then an estimator of $\Io$ is readily given by:
\begin{equation}\label{definition: untruncated Fisher information estimate}
 \hi=\edint\hln'(x)^2\hn(x)dx.
 \end{equation} 
Finally, we estimate $\lambda$ by its maximum likelihood estimator $m/N$. Then using the asymptotic expansion of \eqref{def: one step estimator: general},  we can define  an  estimator  of $\Delta_0$  as follows:
\begin{align}
\label{def: untruncated one step estimator step 1}
    \uDelta=\bDelta+\sum_{i=1}^N\lb \frac{Z_i\hln'(U_i-\bmu)}{m\hi}-\frac{(1-Z_i)\hln'(U_i-\bmu-\bDelta)}{n\hi}\rb.
\end{align}
The estimator in \eqref{def: untruncated one step estimator step 1} is known as the one step estimator (see page 43 of \citeauthor{bickel1982}, \citeyear{bickel1982} or page 72 of \citeauthor{vdv}, \citeyear{vdv}). 
Replacing the $U_i$'s in the definition of $\uDelta$ by $X_i$'s and $Y_i$'s, we rewrite \eqref{def: untruncated one step estimator step 1}  as
\begin{align}
    \label{def: untruncated one step estimator}
  \uDelta= &\  \bDelta + \sum_{i=1}^m \frac{\hln'(X_i-\bmu)}{m\hi} - \sum_{i=1}^n\frac{\hln'(Y_i-\bmu-\bDelta)}{n\hi}\\
= &\ \bDelta +  \edint \frac{\hln'(x-\bmu)}{\hi}d\FF_m(x) -\edint\frac{\hln'(y-\bmu-\bDelta)}{\hi}d\mathbb H_n(y). \nn
\end{align}

The estimator in \eqref{def: untruncated one step estimator} is motivated by the asymptotic expansion in \eqref{def: one step estimator: general} and the semiparametric theory behind it. However, we want to highlight a subtle  difference. The one step estimator $\uDelta$ in \eqref{def: untruncated one step estimator} can be well-defined even when the asymptotic expansion in \eqref{def: one step estimator: general}  is not applicable. The semiparametric model $\mS$, and consequently the expansion in \eqref{def: one step estimator: general}, requires $\lambda\in(0,1)$.  Note that $\lambda$ is interpreted as the  limit of $m/N$ under the model $\mS$. However, $\uDelta$ remains a valid estimator even when $m/N\to 0$ or one. In fact, the definition of $\uDelta$ in \eqref{def: one step estimator: general} does not require $m/N$ to have a limit at all.   Section \ref{sec: asymptotic properties} will show that we  only require $\min(m,n)\to\infty$ to establish the adaptivity of the proposed one step estimators.   This aligns with the asymptotic behavior observed in   \cite{eden} and \cite{beran}'s estimators.
\subsection{Shape-constrained estimation of $\Delta_0$}
\label{sec: shape constrained}
In Section \ref{sec: derivation}, we derived the traditional one step estimator of $\Delta_0$, which is given by  \eqref{def: untruncated one step estimator}. However, for the sake of theoretical analysis, we use a truncated variant of this one step estimator.    Section \ref{sec: truncated one step} discusses the construction of this one step estimator. All one step estimators require preliminary estimators of $\Delta_0$, $\mu_0$, $g_0$, and its log-derivative $\psi_0'$. 
 Section \ref{sec: preliminary estimators} elaborates on the utilization of log-concave density estimators to  compute the one step estimators.

\subsubsection{Truncated one step estimator}
\label{sec: truncated one step}


The asymptotic analysis of the one step estimator in \eqref{def: untruncated one step estimator} poses challenges. To elaborate,  during the performance analysis of the one step estimator, 
 some form of uniform bound is required on the growth rate of $\hln'$ to apply empirical process tools and convergence results. However, since we impose minimal restriction on $\psi_0$ apart from concavity, we have  limited control on the asymptotic behaviour of $\hln'$ at the tails, irrespective of the chosen estimator. Thus, obtaining uniform rate results on $\hln'$ can be challenging.
In the literature on one step estimators, a commonly adopted strategy to tackle such challenges is the trimming of extreme observations during the construction of the one step estimator. This practice results in what is commonly known as the truncated one step estimator \citep{laha2021adaptive,stone,park}.  Hence, for the theoretical analysis presented in this paper, we confine our focus to the truncated one step estimator. However, we include the estimator in \eqref{def: untruncated one step estimator}  in our numerical experiments in Section \ref{sec: simulations}, and henceforth, we refer to it as the untruncated one step estimator.

We trim $\etan$-proportion of the data, where we let $\etan$ tend to  $0$. 
For $\etan\in(0,1/2)$,  we use the notations $\xi_1$ and $\xi_2$, respectively, to represent the $\etan$-th and $1-\etan$-th quantiles of the distribution $G_0$.
Suppose $\hG$ is the distribution function of $\hn$. Then   $\xia=\hG^{-1}(\etan)$ and  $\xib=\hG^{-1}(1-\etan)$ are  reasonable estimators of  $\xi_1$ and $\xi_2$, respectively. Although it would be ideal for the notation of $\xi_1$, $\xi_2$, $\xia$, and $\xib$ to explicitly depend on $m$ and $n$,  we omit this explicit dependency on sample sizes to steamline notation. Now we are ready to introduce our truncated one step estimator, which is given by 
\begin{align}\label{definition: truncated one step estimator}
\hDelta=&\ \bDelta+\dint_{\xia+\bmu}^{\xib+\bmu} \dfrac{\hln'(x-\bmu)}{\hi(\etan)}d\Fm(x)\nonumber\\
&\ -\dint_{\xia+\bmu+\bDelta}^{\xib+\bmu+\bDelta} \dfrac{\hln'(y-\bmu-\bDelta)}{\hi(\etan)}d\Hn(y),
\end{align}
where
\begin{equation}\label{definition: truncated Fisher information estimate}
 \hi(\etan)=\dint_{\xia}^{\xib}\hln'(x)^2\hn(x)dx
 \end{equation} 
  is the truncated  estimator of the  Fisher information $\I$.

We want to clarify a point here. Apparently truncation introduces the tuning parameter $\etan$, which represents the level of truncation.  However, our numerical experiments in Section \ref{sec: simulations} show that when we consider the shape-constrained version of $\hDelta$ (to be discussed in Section \ref{sec: shape constrained}), its variance monotonically decreases as $\eta_{m,n}$ approaches zero. In these experiments, the untruncated  one step estimator in \eqref{def: untruncated one step estimator} consistently demonstrates the lowest variance and the lowest MSE among the shape-constrained one step estimators. Consequently, we recommend the latter  estimator  for practical implementations. In  Section \ref{sec: shape constrained}, we will see that the shape-constrained version of  $\hDelta$ requires no additional tuning parameter other than $\etan$, and the shape-constrained version of $\uDelta$ is tuning free. Thus, although the theoretical analyses of this paper require truncation for technical reasons, the proposed estimator of this paper is tuning-free. This distinguishes our approach from   existing nonparametric methods, where optimal values of tuning parameters are typically unknown, and extensive tuning procedures are often required to identify optimal ranges for these parameters.

\subsubsection{Preliminary estimators of $\Delta_0$, $\mu_0$, and $g_0$}
\label{sec: preliminary estimators}
The current section discusses our choices of $\bmu$, $\bDelta$, and $\hn$. Our focus is primarily on the construction of $\hn$, which heavily relies on the maximum likelihood estimator of log-concave densities.\vspace{1em}\\
\textit{\underline{Preliminary estimators of $\Delta_0$ and $\mu_0$}:}
In general, one step estimators require the preliminary estimators of the finite dimensional parameters to be $\sqrt{n}$-consistent for achieving asymptotic efficiency (see page 72 of \cite{vdv}; see also \cite{park,beran}). 
Our theoretical analysis also  suggests that $\bmu$ and $\bDelta$ have to be $\sqrt{N}$-consistent for $\hDelta$ to be asymptotically  efficient. Under the model defined in equation \eqref{eq: main model}, we have $\mu_0=\E[X]$ and $\Delta_0=\E[Y]-\E[X]$. Therefore, if $g_0$ possesses a finite second moment, $\bmu=\bX$ and $\bDelta=\bY-\bX$ are $\sqrt{n}$-consistent estimators for $\mu_0$ and $\Delta_0$, respectively. It is worth noting that we can substitute the sample means in the above formulas  with other $\sqrt{n}$-consistent estimators of location, such as the sample median \citep[particularly in the log-concave case,][]{laha2021adaptive}, or the Z-estimator for the shift in the logistic shift model \citep[see Example 5.40 and Theorem 5.23 of][]{vdv}. However,  simulations showed that the sample-mean-based preliminary estimators tend to exhibit the best overall performance in terms of efficiency. Therefore, for the purpose of this paper, we take $\bmu$ and $\bDelta$ to be $\bX$ and $\bY-\bX$, respectively.


\vspace{1em}
\noindent\textit{\underline{Preliminary estimator of $g_0$}:}
As mentioned in Section \ref{sec:intro}, we take our estimator of $g_0$ to be log-concave. A vast literature exists on the estimation of log-concave densities \citep{2009rufi,chen2016,dumbreg}. Notably, the MLE within the class of log-concave densities does exist and is computationally tractable  \citep{2009rufi,exist}. However, this MLE cannot be used directly to estimate $g_0$ since our observations, i.e., the $X_i$'s and the $Y_j$'s, are not sampled from $g_0$. Nevertheless, $X_i-\mu_0$ and $Y_j-\Delta_0-\mu_0$ have density $g_0$. Therefore, we first construct the pseudo-observations $X_i-\bmu$'s and $Y_j-\bDelta-\bmu$'s for $i=1,\ldots,m$ and $j=1,\ldots,n$. Then we calculate the log-concave MLE  based on these $m+n$ pseudo-observations. Given that this density estimator is formed by merging pseudo-observations from both the $X$ and $Y$ samples, we refer to it as the pooled estimator for $g_0$ and denote it by $\hl$.

However, $\hl$ is non-smooth due to the non-smooth form of the log-concave MLE.  The lack of smoothness is a common feature of the MLEs of shape-constrained classes \citep{grenander,2009rufi,Balabdaoui2007}. In particular, the logarithm of the log-concave MLE is piecewise affine, with knots belonging to the underlying sample \citep{2009rufi}. However, the lack of smoothness can  potentially  lead to suboptimal  performance in small samples \citep{2009rufi, smoothed, review}. To address this issue, smoothing techniques utilizing a Gaussian kernel can be employed to smooth the log-concave MLE  \citep{smoothed}. The smoothed version of $\hl$  takes  the form
\begin{equation}
    \label{def: smoothed pooled estimator}
    \hat g_{m, n}^{\text{pool,sm}}(x)=\frac{1}{\hatlambda}\edint\varphi\left(\frac{x-t}{\hatlambda}\right)\hl(t)dt,
\end{equation}
where  $\varphi$ is the standard Gaussian density and $ \hatlambda$ is the smoothing parameter. Although it might seem that $\hatlambda$ is an external tuning parameter, that is not the case, as \cite{smoothed} provides a data-dependent closed formula for computing $\hatlambda$.

For our problem, this formula yields 
\begin{align}
\label{def: smoothing parameter}
    \hatlambda^2=\widehat s_n^2-\widehat\sigma^2_n,\quad\text{where}\quad  \widehat\sigma^2_n=\edint z^2\hl(z)dz-\lb\edint z\hl(z)dz \rb^2
\end{align}
is the variance of $\hl$, and $\widehat s_n^2$ is the sample variance of the combined sample.
Given that the overall mean of the combined sample of $X_i-\bmu$'s and $Y_j-\bmu-\bDelta$'s is zero,  the sample variance $\widehat s_n^2$ is given by 
\[\widehat s_n^2=\frac{\sum_{i=1}^m(X_i-\bmu)^2+\sum_{j=1}^n(Y_j-\bDelta)^2}{N-1}.\]
The difference $\widehat s_n^2-\widehat\sigma^2_n$ in \eqref{def: smoothing parameter} is positive by (2.1) of \cite{smoothed}. Consequently, the $ \hatlambda$ in \eqref{def: smoothing parameter} is well-defined.

We will refer to the estimator $ \hat g_{m, n}^{\text{pool,sm}}$ 
as the ``pooled-smoothed estimator", which signifies that it is derived through a two-step process. This process involves data pooling, where pseudo-observations are combined from both the $X$-sample and the $Y$-sample, followed by a smoothing step. 
  The pooled-smoothed estimator $ \hat g_{m, n}^{\text{pool,sm}}$ will serve as our preliminary estimator $\hn$. This density estimator is smooth and supported  on the entire real line $\RR$, ensuring that $\hln=\log\hn$ and its derivative are well-defined at all points on $\RR$.
Additionally, it can be shown that $\int\hat x g_{m, n}^{\text{pool,sm}}(x)dx=0$, indicating that this density estimator is centered, similar to the density $g_0$.

\vspace{1em}
\noindent\textit{\underline{Other possible choices for $\hn$}:}
One might come up with other potential choices for $\hn$. For instance, instead of combining the $X$-sample and the $Y$-sample, one might opt to estimate $g_0$ separately from both samples. A weighted average of the resulting estimators could yield an improved  density estimator.
A smooth $\hn$ can be obtained by computing  smoothed log-concave MLEs \citep{smoothed} from both samples during the above construction. However, averaging destroys the log-concavity structure,  making the weighted estimators described above non-log-concave. The pooled estimator $\hl$, as discussed earlier, is another potential choice for $\hn$. However, while this estimator is log-concave, it is non-smooth.
While these estimators could theoretically lead to asymptotically efficient $\hDelta$, simulations have revealed that the $\hDelta$ derived from our pooled-smoothed estimator outperforms those obtained from these other  estimators. As a result, we have excluded the weighted estimators and the non-smooth pooled estimator from further consideration in this paper.

\begin{remark}
Unimodality is a more naturally occurring shape constraint compared to log-concavity, and the class of unimodal densities is a strict superset of the class of log-concave densities \citep{walther2009inference}. Therefore, unimodality may seem like an alternative choice for the shape restriction on $g$. However, the class of unimodal densities is too large to admit an MLE when the mode is unknown \citep{birge1997}. Therefore, the unimodality constraint does not offer the computational advantages of the log-concavity constraint \citep{laha2021adaptive}. Hence, we opt for log-concavity instead of unimodality as the shape restriction.
\end{remark}

%
%

%



\section{Asymptotic properties}
\label{sec: asymptotic properties}
In this section, we establish the asymptotic efficiency of the truncated one step estimator $\hDelta$ provided $\eta_{m,n}\to 0$ at a sufficiently slow rate. In what follows, we will assume that $g_0\in\mathcal{LC}\cap\mathcal{P}_0$. 
First we present a technical assumption on $\psi_0$, which requires $\psi_0'$ to be globally Lipschitz on $\iint(\dom(\psi_0))$. 
\begin{assumption}
 \label{assump: L}
 There exists $\kappa>0$ so that  
 \[|\psi_0'(x)-\psi'_0(y)|\leq \kappa|x-y|\quad\text{ for all }x,y\in\iint(\dom(\psi_0)),\]
 where  $\psi_0'(x)$ (or $\psi_0'(y)$) can be either  the left or the right derivative of $\psi_0$ at $x$ (or $y$).
 \end{assumption}
When $g_0$ is log-concave, $\psi_0$ is concave, and hence the left and right derivative of $\psi_0$ exist on $\iint(\dom(\psi_0))$ \citep[see page 163][]{hiriart2004}. When  $\psi_0$ is twice differentible, Assumption \ref{assump: L} reads as $|\psi_0''|\leq \kappa$. \cite{beran}, who also considered adaptive estimation in the two-sample location shift problem,  imposed  a similar restriction  on a related function $t\mapsto \psi_0(G_0^{-1}(t))$. Their Assumption B implies that the second derivative of the  latter function is continuous on $[0,1]$. Since $[0,1]$ is compact, this assumption  implies $\psi_0\circ G_0^{-1}$ has a bounded second derivative on $[0,1]$. 
Boundedness assumptions on the second derivative  is also quite common in the shape-constrained litertaure \citep{laha2021adaptive,kuchibhotla,ramu2018}. 
We also want to clarify that Assumption \ref{assump: L} is  required for technical reasons in the proof, and it may not be necessary. As we will see, the Gamma distribution in our simulation setting 4  in Section \ref{sec: simulations} violates  Assumption \ref{assump: L} because  $\psi_0''(x)=-3/x^2$ in this case. However, our simulations in Section \ref{sec: simulations} suggest that $\hDelta$ remains  asymptotically efficient under this setting.


Now we are ready to present the main theorem of this section, which establishes the asymptotic efficiency of $\hDelta$ as defined by \eqref{def: asymptotic eff}. Its proof is provided in the  Appendix.
\begin{theorem}
\label{theorem: main theorem specific}
Suppose $X$ and $Y$ are as in \eqref{eq: main model} and  $g_0\in\mathcal{LC}\cap\mP_0$. Let $\eta_{m,n}=C N^{-2\varrho/5}$ where $C>0$ is a constant and $\varrho\in(0,1/4]$. Further suppose $\min(m,n)\to\infty$. Then, under  Assumption \ref{assump: L}, the truncated one step  estimator $\hDelta$ with preliminary estimators $\bmu=\bX$, $\bDelta=\bY-\bX$, and $\hn=\hls$ satisfies 
    \[\sqrt{\dfrac{mn}{N}}(\hDelta-\Delta_0)\to_d \N(0,\Io^{-1}),\]
where $\Io$ is the Fisher information corresponding to $g_0$.
\end{theorem}


The rate of decay of the truncation parameter $\eta_{m,n}$ in Theorem \ref{theorem: main theorem specific} merits some discussion. Proposition \ref{prop: pooled estimator satisfies Conditions} in the Appendix shows that the Hellinger distance between  $\hls$ and $g_0$
is of the order $O_p(N^{-1/4})$.  Theorem \ref{theorem: main theorem specific} requires $\eta_{m,n}$ to decay to zero at a slower rate than the latter. Thus Theorem \ref{theorem: main theorem specific} calibrates   $\eta_{m,n}$'s decay rate by the Hellinger error rate of $\hls$. 
 To demonstrate that the Hellinger error of $\hls$ is $O_p(N^{-1/4})$, we utilize the Hellinger continuity theory of log-concave projection developed by \cite{barber2020}. However, given the Hellinger error of the log-concave MLE is of the order $O_p(N^{-2/5})$ \citep[cf.][]{dossglobal}, the   actual rate of Hellinger error decay for $\hls$ may be faster than $O_p(N^{-1/4})$.
Although a sharper rate of Hellinger error decay would allow for a smaller truncation level $\eta_{m,n}$ in Theorem \ref{theorem: main theorem specific}, its impact in this work  will be nominal. This is because the theoretical truncation levels will not be used  during practical implementations.

Although we  present Theorem \ref{theorem: main theorem specific} for our specific  preliminary estimators, the assertion actually holds for a broader class of preliminary estimators. In fact, we  only require $\bmu$ and $\bDelta$ to be $\sqrt{N}$-consistent estimators of $\mu_0$ and $\Delta_0$, respectively. Additionally, $\hn$  needs to satisfy some consistency conditions including those related to the Hellinger error. A more generalized version of Theorem \ref{theorem: main theorem specific} can be found in Appendix \ref{sec: pf architecture}.

As previously mentioned, our empirical study indicates that the untruncated one step estimator performs as well as, or even better than, the truncated one step estimators. While we currently lack theoretical results regarding the asymptotic distribution of the untruncated one step estimator, our simulations lead us to conjecture that Theorem \ref{theorem: main theorem specific} holds for this estimator as well.\vspace{1em}\\
\underline{\textit{Proof outline:}}
Although the detailed proof of Theorem \ref{theorem: main theorem specific} is provided in the Appendix, we offer a concise overview of the proof here to provide some intuition.
 A pivotal and challenging aspect of our proof involves demonstrating that   $\hln'(x)$ does not grow faster than a polynomial rate when  $x\in[\xia,\xib]$. 
 Achieving such bounds on $\hln'$ necessitates a degree of control over the Hellinger error of $\hls$.
 
To this end, we show that $\H(\hls,g_0)=O_p(N^{-1/4})$, as mentioned previously. Proving this Hellinger error bound  is a major part of the proof, which also yields  that $\hin(\etan)\to_p\Io$. Our proof technique is tailored specifically to the truncated one step estimator, as extending it to the untruncated version would require a uniform bound on $\hln'$—a challenge that currently remains unresolved.

The rest of the proof hinges on the following  decomposition of  $\hDelta$:
\begin{align}
\label{inoutline: initial}
 \hDelta - \bDelta=&\ \underbrace{\dint_{\xia+\bmu}^{\xib+\bmu} \dfrac{\hln'(x-\bmu)}{\hi(\etan)}d\Fm(x)}_{A}-\underbrace{\dint_{\xia+\bmu+\bDelta}^{\xib+\bmu+\bDelta} \dfrac{\hln'(y-\bmu-\bDelta)}{\hi(\etan)}d\Hn(y)}_B. 
\end{align}
Using $\hin(\etan)\to_p\Io$ and the abovementioned bound on $\hln'$ on   $[\xia,\xib]$,  $A$ can be decomposed as follows: 
\begin{align}
\label{inoutline:x}
  A\ \approx\ \parbox{5em}{empirical process\\ term}\ +\ \text{CLT term}_A\ +\ \parbox{5em}{bias due to \\ estimating $g_0$}+\underbrace{\mu_0-\bmu}_{\text{bias due to estimating }\E[X]} +o_p(N^{-1/2}).
\end{align}
The central limit theorem (CLT) term is an important term in the above expansion since it contributes to the desired  asymptotic normality.  It is the average of some independent random variables based on the $X$ sample, and hence amenable to the central limit theorem.  
 The empirical process term is asymptotically negligible, and we use Donsker's theorem to establish its asymptotic negligibility.  The goal of the asymptotic expansion in \eqref{inoutline:x} is to separate the CLT term and the bias terms from the other asymptotically negligible terms. 
 
Similarly,  $B$'s asymptotic expansion writes as
 \begin{align}
 \label{inoutline: y}
  B\ \approx\  \parbox{5em}{empirical process\\ term}\ +\ \text{CLT term}_B\ +\  \parbox{5em}{bias due to\\ estimating $g_0$}\ +\  \underbrace{\mu_0+\Delta_0-\bmu-\bDelta}_{\text{bias due to estimating }\E[Y]} \ +\ o_p(N^{-1/2}),
\end{align}
 where as before, the CLT term contributes to the final asymptotic normality and the empirical process term is asymptotically negligible. The CLT term from $B$, being   an average based on the $Y$ sample, is independent of the CLT term resulting from the expansion of $A$. 
 
 The key point here is that the bias due to the estimation of $g_0$ in \eqref{inoutline:x} and \eqref{inoutline: y} are identical, and they cancel each other in the expression of $A-B$ in \eqref{inoutline: initial}.  Subtracting the CLT term of $B$ from that of $A$ and rearranging  \eqref{inoutline: initial} lead to the desired asymptotic normality in Theorem \ref{theorem: main theorem specific}.\vspace{1em}\\
 \underline{\textit{Comparison with the one sample case:}}
\cite{laha2021adaptive} developed a one step estimator of the location in the one-sample symmetric location model using log-concavity assumptions similar to ours. They demonstrated that their truncated one step estimator is adaptive when the log-concavity assumption holds. The main distinction between \cite{laha2021adaptive}'s setting and ours is that the underlying density $g_0$ is symmetric in \cite{laha2021adaptive}'s one-sample location model.  Therefore,  \cite{laha2021adaptive}'s estimators of $g_0$ are  different from ours. While \cite{laha2021adaptive} included a smooth density estimator among their density estimators, this density estimator was non-log-concave and substantially different from our $\hls$. Consequently, the Hellinger rate analysis of our smooth density estimator, a crucial component in proving Theorem \ref{theorem: main theorem specific}, differs from \cite{laha2021adaptive}. Additionally, the overall calculation in our case is more technically involved than \cite{laha2021adaptive}'s one-sample case due to additional complexities arising from dealing with two samples.
\section{Simulation Study}
\label{sec: simulations}
In this section, we compare the performance of the proposed estimator with that of some other asymptotically efficient estimators available in the literature. We particularly choose the estimators of \cite{park} and \cite{beran} as comparators since they involve fewer tuning parameters than most existing semiparametric estimators of $\Delta_0$. Additionally, we compute the parametric Maximum Likelihood Estimator (MLE) to serve as a benchmark for evaluating the performance of the semiparametric estimators.

\subsection{Simulation design}
\label{sec: simu design}

First, we describe the common simulation setup. We consider the following four schemes for data generation:
\begin{enumerate}
    \item $g_0\sim \N(0,1)$, $\mu=0$, $\Delta=1$.
    \item $g_0$ is the standard logistic density, $\mu=0$, and $\Delta=1$.
    \item $g_0$ is the standard Laplace density, $\mu=0$ and $\Delta=1$.
    \item $g_0=f_0(\cdot+2)$ where $f_0\sim \text{Gamma}(4,0.5)$. Here,  for $a,b>0$, we take  $\text{Gamma}(a,b)$ to be the Gamma distribution with shape $a$ and scale $b$. Since the expectation of $\text{Gamma}(4,0.5)$ is $2$,  it follows that $g_0$ is centered. We generate  the $X$-observations  from $\text{Gamma}(4,0.5)$, which means $\mu=2$. As before, we take $\Delta=1$.
\end{enumerate}
In all the scenarios mentioned above, $g_0$ is log-concave. 
The Fisher information for settings 1, 2, 3, and 4 are 1, 1/3, 1, and 2, respectively. 
We take $m=n$, with $n$ taking on values from the set $\{40, 100, 200, 500\}$. This results in four configurations for each data generating scheme.  For each sample size and each data-generating scheme, we generated 4000 Monte Carlo samples.

For our proposed one step estimators, we choose the preliminary estimators as discussed in Section \ref{sec: preliminary estimators}. Specifically, we set $\bmu=\bX$, $\bDelta=\bY-\bX$, and take the preliminary estimator $\hn$ of $g_0$ to be the pooled-smoothed estimator $\hls$.  We consider four  choices for the truncation levels, given by, $\eta=0, 0.01, 0.001$, and  $0.0001$. Note that the level $\eta = 0$ corresponds to the untruncated  one step estimator in \eqref{def: untruncated one step estimator}.  Theorem \ref{theorem: main theorem specific} implies that when $m=n$,  a 95\% confidence interval of $\Delta_0$ is given by
\begin{align}
    \label{def: confidence interval}
    \left(\hDelta-z_{0.025}\sqrt{2n^{-1}/\hin(\etan)},\ \hDelta+z_{0.025}\sqrt{2n^{-1}/\hin(\etan)}\right),
\end{align}
where $z_{.025}$ is the $0.975$th quantile of the standard Gaussian distribution. For the untruncated one step estimator $\uDelta$, we consider the following confidence interval:
\[
    \left(\uDelta-z_{0.025}\sqrt{2n^{-1}/{\hin}},\ \uDelta+z_{0.025}\sqrt{2n^{-1}/{\hin}}\right),
\]
where $\hin$ is the untruncated Fisher information estimator in \eqref{definition: untruncated Fisher information estimate}. 
Finally, to distinguish our one step estimator from that proposed by \cite{park}, we refer to these one step estimators as shape-constrained one step estimators. 

\subsubsection{Comparators}
Both of the considered comparators require $\sqrt{n}$-consistent preliminary estimators of $\Delta_0$ and $\mu_0$. We choose $\bmu$ as $\bX$ and $\bDelta$ as $\bY-\bX$ to maintain consistency with the shape-constrained one step estimators. Further details on these estimators are provided below.
\vspace{1em}\\
\noindent\textit{\underline{\cite{park}'s estimator}:}
\cite{park} considers the joint estimation of location and scale shifts in the two-sample location-scale model and provides a truncated one step estimator for the location shift.
 Since our model does not involve scale-shift, we take the  scale-shift to be one while computing the above estimator. 
From \cite{park}, it follows that the resulting estimator of location shift is asymptotically efficient under the model in \eqref{eq: main model}.  The main differences between \citeauthor{park}'s one step estimator and our $\hDelta$ are as follows: (a) they take $\hn$   to be a kernel density estimator based on the logistic kernel and (b)  they employ a smooth truncation where the influence function estimators gradually decrease to zero outside a compact interval.   This procedure necessitates two tuning parameters:  an external smoothing parameter $b_{m,n}$ for selecting the kernel bandwidth  and  a truncation parameter $c_{m,n}$.
According to \cite{park}, the one step estimator is not  sensitive to the level of truncation. In line with \cite{park}, we set $c_{m,n} = 8$. However, \cite{park}  reports their one step estimator to be sensitive to the choice of the smoothing parameter $b_{m,n}$ and proposes a tuning scheme for it. 
Their tuning scheme searches for the $b_{m,n}$ that minimizes a bootstrap estimator of the MSE of the one step estimator in the interval $[0.0005, 1.0]$. We follow this tuning scheme to select $b_{m,n}$. 
We calculate a 95\% confidence interval for \citeauthor{park}'s estimator similarly to \eqref{def: confidence interval}.
In this case, we use \citeauthor{park}'s estimator of $\Io$. Their estimator is similar to our truncated Fisher information estimator  $\hin(\etan)$ except that the  score estimates are derived from the kernel density estimator of $g_0$ and the truncation follows the  scheme outlined in their paper.\vspace{1em}\\
\noindent\textit{\underline{\cite{beran}'s estimator}:}
\cite{beran} uses a rank-based estimator developed by \cite{kraft1970} for estimating $\Delta_0$. Their method relies on preliminary estimators and the estimators of the scores and the Fisher information. The scores are computed using a Fourier series expansion, and the Fisher information is estimated using the estimated Fourier coefficients. We calculate a 95\% confidence similarly to \eqref{def: confidence interval}, where we incorporate the aforementioned estimator of the Fisher information.
\citeauthor{beran}'s estimator involves two tuning parameters: the number of Fourier basis functions $M_b$ and a smoothing parameter $\vartheta$. The smoothing parameter is required for nonparametrically estimating the Fourier series coefficients.

\cite{beran} does not provide guidance on the selection of these tuning parameters. We considered 75 pairs of $(M_b, \vartheta)$  over the two-dimensional grid $\{10, 20, 30, 40, 50\}$ $\times$ $\{0.1, 0.2, \ldots, 1.5\}$. For each pair and each simulation scheme, we estimated the MSE and coverage of  \citeauthor{beran}'s estimator using 400  Monte Carlo samples. For each scheme, two pairs of $(M_b, \vartheta)$ were chosen: 
 one minimizing the estimated MSE and the other maximizing the estimated coverage of the confidence intervals. We refer to the corresponding estimators as the MSE-tuned \citeauthor{beran}'s estimator and the coverage-tuned \citeauthor{beran}'s estimator, respectively.
While our tuning procedure for \citeauthor{beran}'s estimator requires knowledge of the data generating mechanism, we adopt this tuning scheme to ensure that the resulting estimators offer a competitive benchmark for evaluating our method. 

To benchmark the performance of our the semiparametric estimators, we also use $\widehat\Delta_{\text{MLE}}$, the parametric MLE  of $\Delta_0$. It is an oracle estimator in the sense that it  assumes  the knowledge of the density $g_0$. The asymptotic variance of the parametric MLE is $2\Io^{-1}/n$, which leads to the following 95\% confidence interval for $\Delta_0$ based on $\widehat\Delta_{\text{MLE}}$:
\begin{equation}
    \label{def: CI of MLE}
    \left(\widehat\Delta_{\text{MLE}}-z_{0.025}\sqrt{2n^{-1}/\Io},\ \widehat\Delta_\text{{MLE}}+z_{0.025}\sqrt{2n^{-1}/\Io}\right).
\end{equation}
In settings 1 and 3, the parametric MLE corresponds to the difference between the means and the medians of the two samples, respectively. For the logistic and gamma densities in settings 3 and 4, the parametric MLE of $\Delta_0$ is numerically obtained by maximizing the respective log-likelihood functions. In addition to these estimators, the preliminary estimator $\bDelta=\bY-\bX$ is included as a baseline for comparison. To construct an asymptotic confidence interval for the preliminary estimator, we apply the central limit theorem and estimate the corresponding variance term using the method of moments estimator:
 \[\frac{\sum_{i=1}^n(Y_i-\bY)^2+\sum_{i=1}^m(X_i-\bx)^2}{m+n-2}.\]


\subsubsection{Performance measures}
The primary performance measure of interest is the efficiency.  We  define the  efficiency of an estimator $\widehat{\Delta}$ of $\Delta_0$  as follows:
\begin{equation}
\label{def: efficiency}
    \text{Efficiency}(\widehat{\Delta})=\frac{\text{Var}(\widehat\Delta_{\text{MLE}})}{\text{Var}(\widehat{\Delta})}, 
\end{equation}
where $\widehat\Delta_{\text{MLE}}$ is the parametric MLE. In large samples, it is expected that the efficiency of all  estimators would be smaller than or equal to one. Higher values of the efficiency are preferred. To compute the efficiency, we estimate the variance of each estimator, including the parametric MLE, using the 4000 Monte Carlo samples. Figure \ref{fig:efficiency}  illustrates the efficiency of various estimators.  Figure \ref{fig:MSE}  plots the scaled (by $mn/(m+n)$) MSE. 
In addition to efficiency, we also provide estimates of the coverage of the confidence intervals using the 4000 Monte Carlo samples. We also investigate the width of the confidence intervals, an important aspect of our  comparison study. Assuming that the coverage of an estimator is satisfactory, narrower confidence intervals are preferred. Figures \ref{fig:coverage} and \ref{fig:CI_width} display the coverage and widths of the confidence intervals, respectively.

\subsection{Results}

First, we analyze the performance of the shape-constrained estimators. Figures \ref{fig:1} and \ref{fig:2} reveal that the shape-constrained one step estimators with truncation level  $\eta\leq 0.001$ consistently exhibit high efficiency, high coverage, and low MSE across all experimental settings.  In terms of efficiency and MLE, these estimators either outperform or perform comparably to \citeauthor{park} and \citeauthor{beran}'s estimators, with a minor exception in the Gaussian case, where they exhibit slightly lower efficiency and higher MSE. 
However, even in the Gaussian case, their efficiency remains close to 0.90  in large sample sizes at smaller truncation levels.  The shape-constrained estimators achieve coverage rates that are highly competitive with the parametric MLE and noticeably higher than those of \citeauthor{park} and \citeauthor{beran}'s estimators.  While the confidence intervals of the shape-constrained estimators tend to be wider in small samples, their width decreases rapidly as the sample size grows, becoming comparable to other estimators in large samples.

A comparison of different truncation levels for the shape-constrained estimator reveals that the coverage and the width of the confidence intervals remain fairly consistent regardless of the truncation level $\eta$. In fact, when $\eta\leq 0.001$, these aspects are nearly identical. However, Figure \ref{fig:1} indicates that the efficiency and the MSE  of the shape-constrained estimator improve as the truncation level $\eta$ decreases, although this difference becomes almost negligible when $\eta\leq 10^{-4}$. The untruncated one step estimator exhibits the highest efficiency and lowest MSE among all shape-constrained estimators. Therefore, it is evident that the untruncated one step estimator offers the best overall performance among the shape-constrained estimators, making it our recommended choice for practical implementation.


Note that, as mentioned earlier, Assumption \ref{assump: L} is not satisfied in the case of the gamma distribution in setting 4. However, it is observed that the simulations do not refute the possibility that $\hDelta$ is asymptotically  efficient in this setting. A closer look at Figures \ref{fig:efficiency} and \ref{fig:coverage} reveals that in this scenario, the variance and the coverage of the truncated one step estimators are very close to that of the parametric MLE, particularly in case of large samples.  
This observation indicates that Assumption \ref{assump: L} may not be a necessary condition for Theorem \ref{theorem: main theorem specific}.

Now we turn our attention to the semiparametric comparators, i.e., \citeauthor{beran} and \citeauthor{park}'s estimators. 
The MSE-tuned \citeauthor{beran}'s estimator and \citeauthor{park}'s estimator show similar trends due to their similar tuning procedures, which is based on either minimizing the MSE (\citeauthor{beran}) or its bootstrap estimator (\citeauthor{park}).  They demonstrate the highest efficiency among the semiparametric estimators in the Gaussian case. However, in the Laplace and Gamma scenarios, their efficiency is noticeably lower than that of the semiparametric estimators. Figure \ref{fig:MSE} also indicates that in the last two cases, their MSE is higher than that of the shape-constrained estimators. Additionally, their coverage is the lowest among all estimators. The coverage-tuned \citeauthor{beran}'s estimator behaves differently from the aforementioned two estimators due to its distinct tuning procedure that prioritizes the coverage.  While it achieves moderate coverage, it has the lowest efficiency among all estimators. Furthermore, its exhibited coverage is still lower than that of the shape-constrained estimators, as shown in Figure  \ref{fig:coverage}.

\begin{figure}[h]
    \centering
    \begin{subfigure}{\textwidth}
    \centering
         \includegraphics[width=\linewidth, height=2.5in]{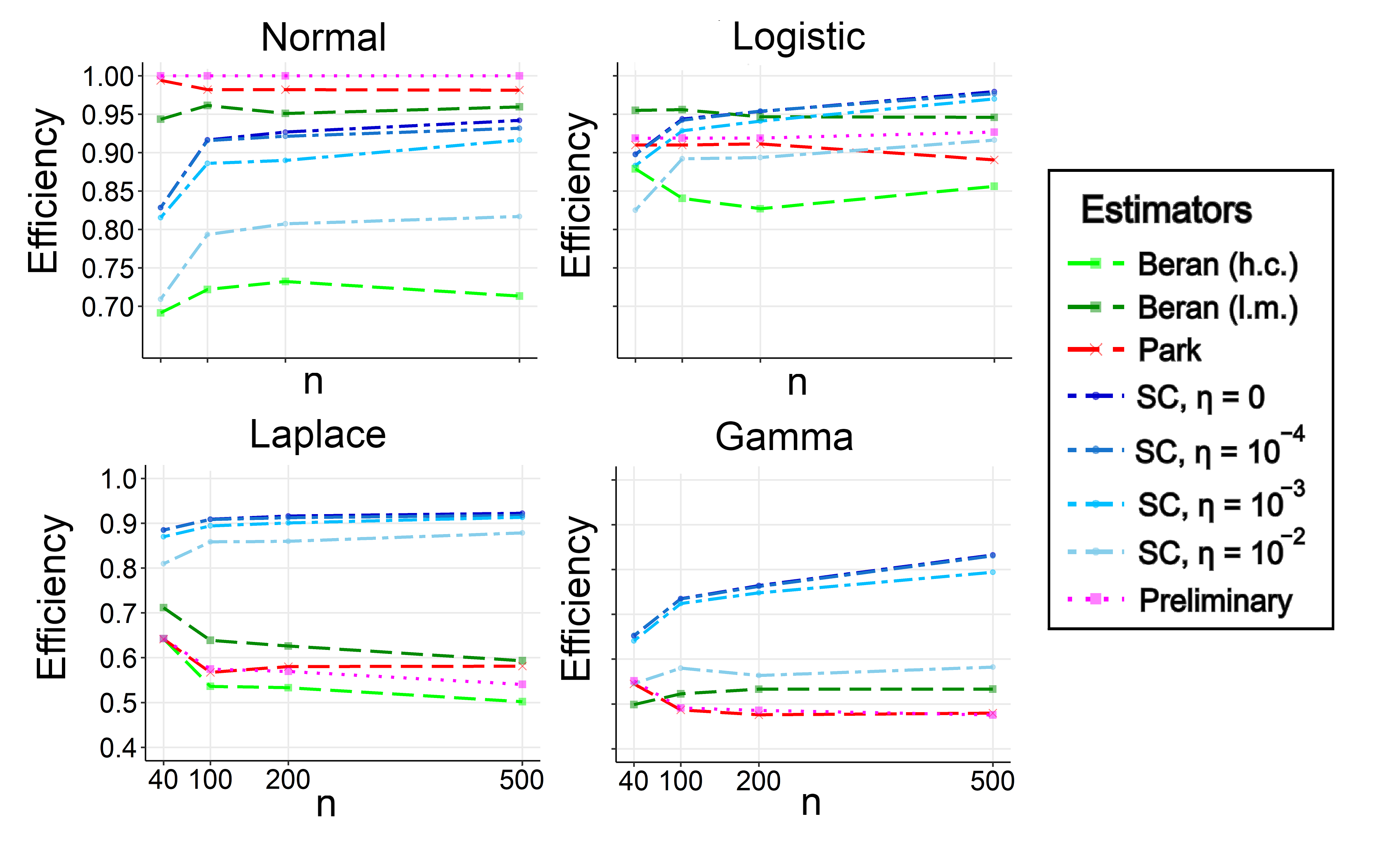}
         \caption{Comparison of the efficiency of different estimators. The  parametric MLE is omitted since its efficiency is one by \eqref{def: efficiency}.}
          \label{fig:efficiency}
    \end{subfigure}
    \begin{subfigure}{\textwidth}
    \centering
         \includegraphics[width=\linewidth, height=2.5in]{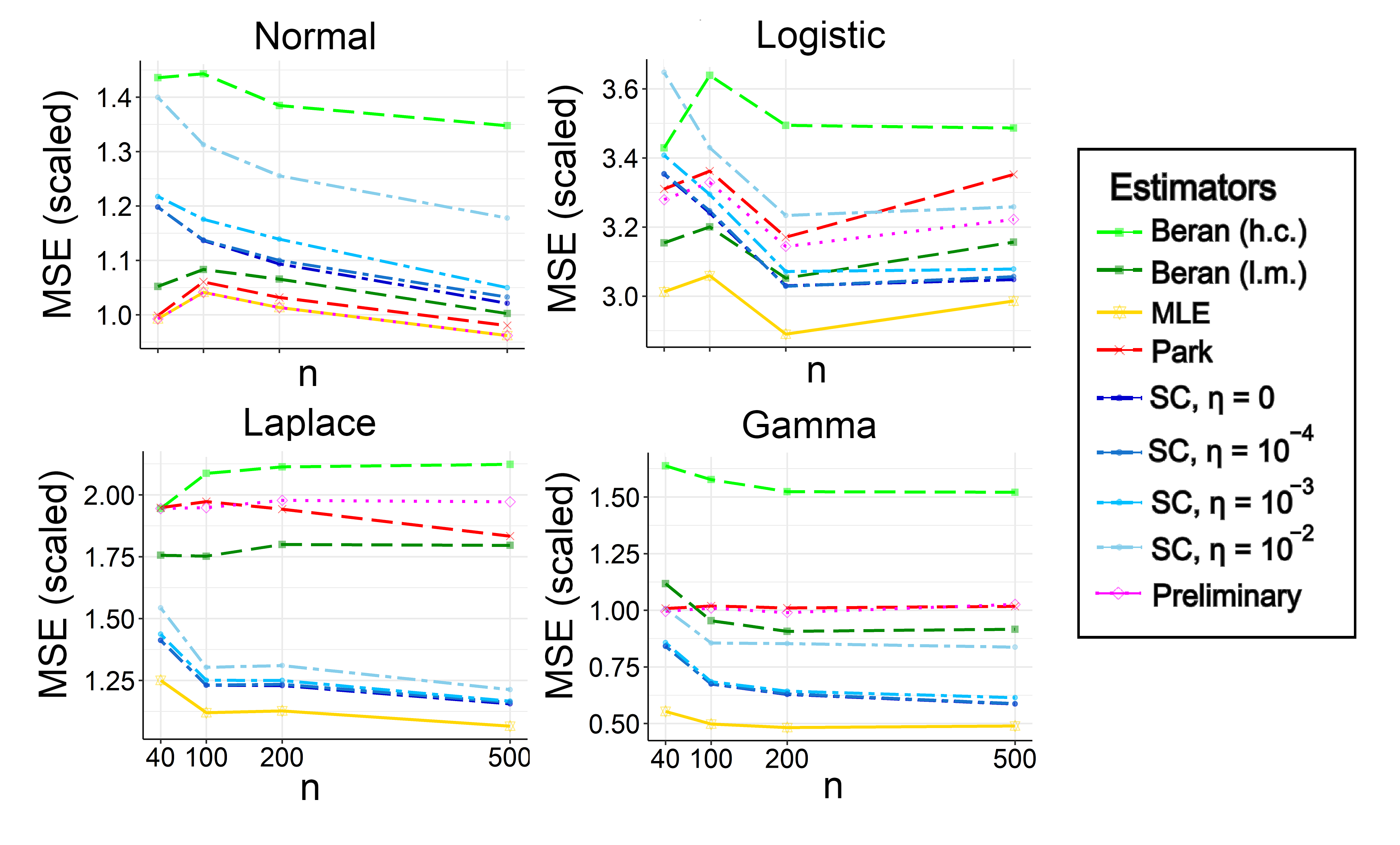}
         \caption{Comparison of the scaled (by $mn/(m+n)$) MSE of different estimators. Since $m=n$, the scaling factor becomes $n/2$.}
         \label{fig:MSE}
    \end{subfigure}
     \caption{ Plot of the efficiency and the MSE of the confidence intervals. Shape and scale of gamma distribution are 4 and 0.5, respectively. Estimators: \textit{Beran (h. c.)} and \textit{Beran (l. m.)} correspond to the (highest) coverage-tuned and (lowest) MSE-tuned \citeauthor{beran}'s estimators, respectively; \textit{Park} denotes  \citeauthor{park}'s estimator; \textit{SC} stands for the shape-constrained estimators where  $\eta$ stands for the truncation level $\etan$; \textit{preliminary} denotes the preliminary estimator $\bDelta=\bY-\bX$. }
   \label{fig:1}
\end{figure}


\begin{figure}[h]
    \centering
    \begin{subfigure}{\textwidth}
    \centering
         \includegraphics[width=0.9\linewidth, height=2.5in]{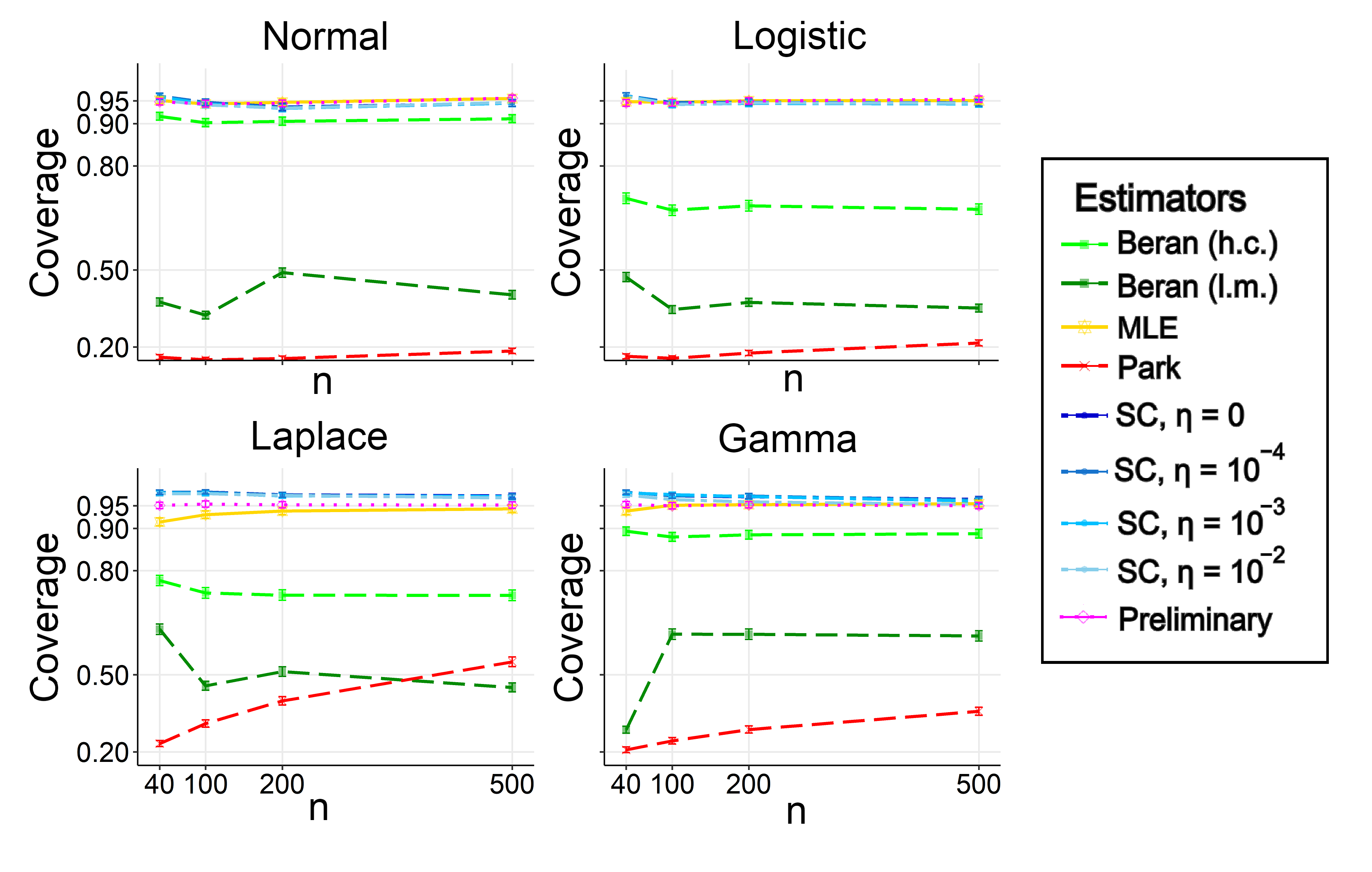}
         \caption{Comparison of the coverage of the 95\% confidence intervals.  The error-bars are given by $\pm$2 standard deviations.}
         \label{fig:coverage}
    \end{subfigure}
     \begin{subfigure}{\textwidth}
    \centering
         \includegraphics[width=0.9\linewidth, height=2.5in]{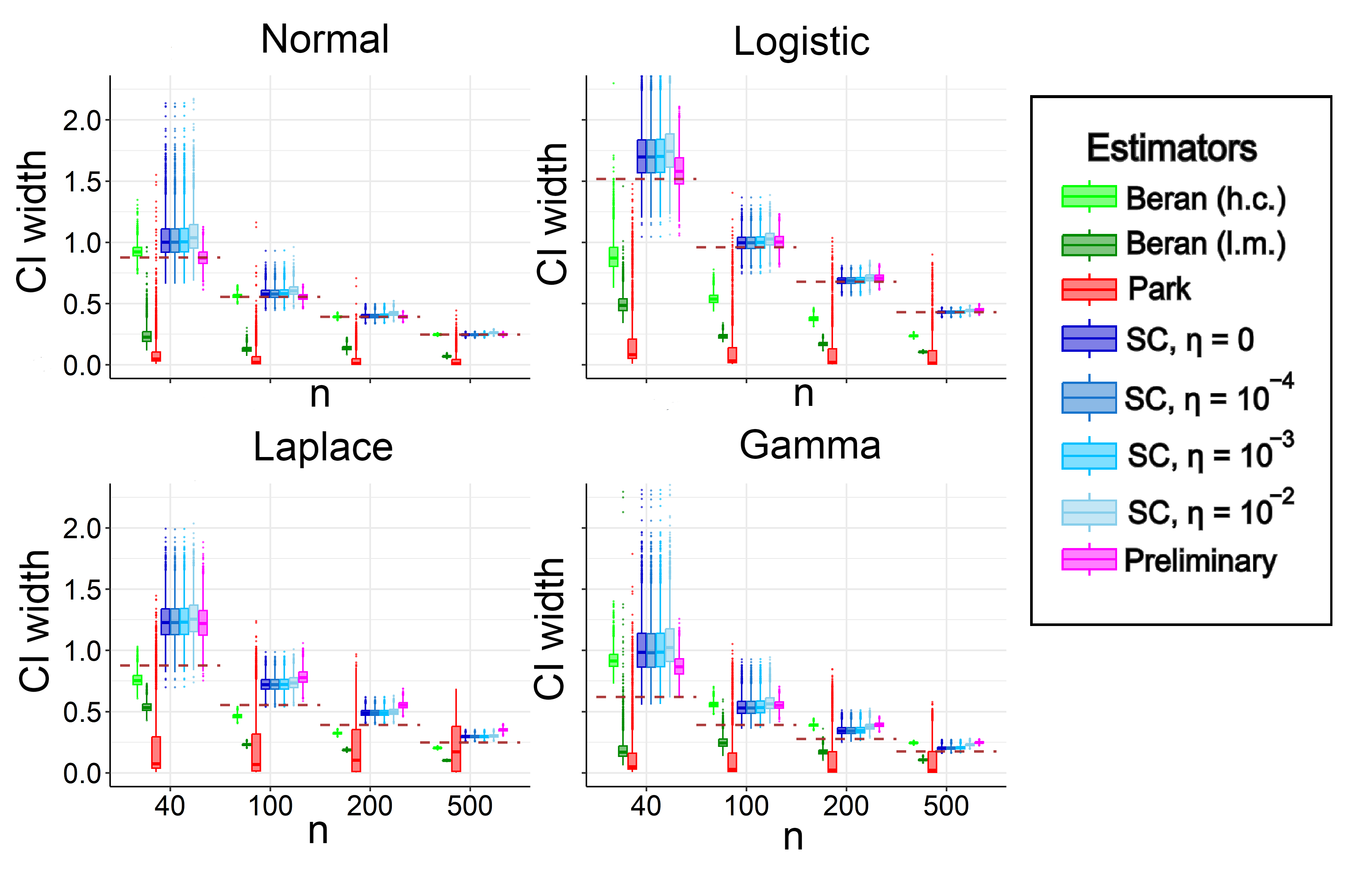}
         \caption{Comparison of the widths of the confidence intervals. The brown line represents the width of the confidence interval for the parametric MLE, as defined by \eqref{def: CI of MLE}.}
          \label{fig:CI_width}
    \end{subfigure}
     \caption{ Plot of the coverage and the width of the confidence intervals. Shape and scale of gamma distribution are 4 and 0.5, respectively. Estimators: \textit{Beran (h. c.)} and \textit{Beran (l. m.)} correspond to the (highest) coverage-tuned and (lowest) MSE-tuned \citeauthor{beran}'s estimators, respectively; \textit{Park} denotes  \citeauthor{park}'s estimator; \textit{SC} stands for the shape-constrained estimators where  $\eta$ stands for the truncation level $\etan$; \textit{preliminary} denotes the preliminary estimator $\bDelta=\bY-\bX$. }
    \label{fig:2}
\end{figure}

A closer inspection of  \citeauthor{park}'s estimator 
   and the MSE-tuned version of \citeauthor{beran}'s estimator brings out a revelation that might partially explain the reason behind their underperformance in the simulation study experiments. We discovered that both of these methods tend to overestimate the Fisher information in our simulations. In the case of \cite{park}, the Fisher information estimator also exhibits high variance. This has two-fold effects. Firstly, the overestimation of the Fisher information leads to excessively narrow confidence intervals. This is due to the inverse proportionality between the width of confidence intervals and Fisher Information estimates. Notably,  in Figure \ref{fig:CI_width}, the two estimators in question exhibit narrower confidence intervals compared to the parametric MLE. Moreover, the width of \citeauthor{park}'s confidence interval is highly variable due to the high variance of its Fisher information estimator. The narrowing of confidence intervals explains the  lack of coverage for these two intervals.  Secondly, since   \citeauthor{park}'s estimator is a one step estimator, from \eqref{def: one step estimator: general}, it can be seen that the overestimation of the Fisher information leads to negligible correction over the preliminary estimators. \citeauthor{beran}'s estimator's formula in \cite{beran} indicates that it also behaves in a similar way when the Fisher information is overestimated.  Consequently, both \citeauthor{park}'s estimator and the MSE-tuned version of \citeauthor{beran}'s estimator become nearly identical to the preliminary estimator when the Fisher information is overestimated.
    Therefore, these estimators exhibit high efficiency when the preliminary estimator is the parametric MLE, e.g., setting 1, but exhibit lower efficiency when the preliminary estimator has low efficiency (e.g., setting 3 and 4); see Figure \ref{fig:efficiency}\footnote{However, the coverage of these two estimators is different than that of the preliminary estimators because we used different variance estimators to construct the confidence intervals for the preliminary estimators.}. Furthermore, in the case of \citeauthor{park}'s estimator, we have verified that in large samples,  minimizing the actual MSE leads to smoothing parameters and efficiency that are quite similar to those obtained through the implemented bootstrap-based tuning method. The above hints that finding the optimal tuning parameters for \citeauthor{beran}'s and \citeauthor{park}'s estimators might involve subtleties beyond minimizing a reliable estimator of the MSE. 
   
   

    To summarize, when the shape constraints are satisfied, the shape-constrained one step estimators exhibit high coverage and efficiency  across all settings. Among them, the untruncated one step estimator has the best overall performance, making it stand out as the most pragmatic estimator for practical purposes. In contrast, the semiparametric comparators, i.e., \citeauthor{park} and \citeauthor{beran}'s estimators,  may exhibit high efficiency in some settings, but their efficiency may be low in other settings. Also, it appears that they generally result in confidence intervals with lower coverage.

\section{Discussion}
In this paper, we illustrate how introducing an additional assumption of log-concavity effectively addresses the challenge of external tuning parameters in the two-sample location shift problem. We establish that the  truncated one step estimators are adaptive when the truncation level decays to zero sufficiently slowly. Simulation experiments conducted across a range of scenarios provide evidence suggesting that the optimal truncation level is zero. Based on this insight, we recommend the untruncated one step estimator as the preferred estimator for this problem. This finding aligns with a broader trend in nonparametric problems, encompassing density estimation, regression, location estimation, and Hellinger distance estimation, where introducing shape restrictions has been shown to alleviate the tuning parameter burden \citep{2009rufi,kuchibhotla,laha2021adaptive,laha2022improved,review,walther2009inference}. Our research contributes to this growing body of research, highlighting the benefits of employing shape constraints for tuning parameter-free inference. We discuss two intriguing avenues for future research below.\vspace{0.1cm}\\
\textit{Extension to higher dimension:} In this case,  $X$ and $Y$ can be $d$-dimensional vectors, where $d\in\mathbb{N}$. The techniques for obtaining the efficient influence function can be generalized to this multidimensional scenario  as long as  $d$ is fixed. However, the algorithm for computing the smoothed log-concave MLE, proposed by \cite{smoothed}, tends to be slow in higher dimensions \citep{xuhigh,review}. Nevertheless,  recent algorithmic advancements by \cite{chen2021new} in computing high-dimensional log-concave  MLE  hold promise for extending log-concavity-based inference to higher dimensions.\vspace{0.1cm}\\
\textit{Misspecified case (non-log-concave $g_0$):} Using the log-concave projection theory developed by \cite{dumbreg}, one may show that  our proposed estimator of $g_0$ would consistently estimate the log-concave projection of $g_0$ onto the class $\mP_0\cap\mathcal{LC}$ when $g_0$ is non-log-concave, provided some weak moment conditions  hold.  In light of \cite{laha2019}'s findings in the symmetric location model,   we anticipate   $\hDelta$  to remain a consistent estimator of $\Delta_0$ even in the misspecified case. However, it might no longer be asymptotically efficient.

\begin{acks}[Acknowledgments]
The authors are grateful to Dr. Jon Wellner and Dr. Rajarshi Mukherjee for their valuable insights.
\end{acks}

\begin{funding}
Nilanjana Laha's research has been partially supported by the NSF-DMS grant DMS-2311098.
\end{funding}
\begin{supplement}
\textbf{Appendix A:} Attached, contains the proof of Theorem \ref{theorem: main theorem specific}.\vspace{1em}\\
\textbf{Code:} R package \texttt{TSL.logconcave} is made available on Github \footnote{https://github.com/nilanjanalaha/TSL.logconcave}.
\end{supplement}
 \bibliographystyle{imsart-nameyear}
 \bibliography{location_estimation}
\FloatBarrier
\newpage
\begin{center}
\begin{huge}
    Appendix
\end{huge}
\end{center}
 \appendix
 Appendix \ref{sec: pf architecture} provides a high-level overview of the proofs, demonstrating that proving Theorem \ref{theorem: main theorem specific} relies on establishing Proposition \ref{prop: pooled estimator satisfies Conditions} and Theorem \ref{theorem: main theorem}. The proofs of Proposition \ref{prop: pooled estimator satisfies Conditions} and Theorem \ref{theorem: main theorem} are presented in Appendices \ref{sec: proof of the conditions} and \ref{app: proof of Theorem 1}, respectively. Appendix \ref{secpf: auxilliary} compiles the proofs of the auxiliary lemmas used in the preceding proofs. Finally, Section \ref{sec: technical facts} presents some technical facts used in the proofs.

\section{Proof architecture}
\label{sec: pf architecture}
 Before going into further details, we will provide an outline of the proof of Theorem \ref{theorem: main theorem specific}. 
 First, we will show that it is enough to prove Theorem \ref{theorem: main theorem specific} when $m/N$ has a limit. Suppose  the assertion of Theorem \ref{theorem: main theorem specific} holds  when $m/N$ has a limit. Consider  $\eta_{m,n}=CN^{-2\varrho/5}$ where $C>0$,  $\varrho\in(0,1/4]$, and  let $P_{m,n}$ be  the measure induced by the distribution  of $\sqrt{mn/N}(\hDelta-\Delta_0)$. Also, let $\PP$ denote the measure induced by a  centred Gaussian distribution with variance $\Io^{-1}$. Note that proving  Theorem \ref{theorem: main theorem specific} is equivalent to showing that $P_{m,n}$ converges weakly to $\PP$ for any  $\varrho\in(0,1/4]$ and $C>0$.  
Because the sequence $m/N$ is  bounded, Bolzano-Weierstrass theorem implies that given  any subsequence of $(m,n)$, we can find a further subsequence $(m_k,n_k)$ so that  $m_k/(m_k+n_k)$ converges to a limit. Replacing $(m,n)$ with this subsequence $(m_k,n_k)$, by our hypothesis, we obtain that $P_{m_k,n_k}$ converges weakly to $\PP$. Therefore, we have shown that given any subsequence of $P_{m,n}$, we can find a further subsequence $P_{m_k,n_k}$ that weakly converges to $\PP$. The latter statement implies that $P_n$ converges weakly to $\PP$  \citep[see  Theorem 2.6 of][for a proof of this result]{billingsley2013}. Hence, we have shown that  it suffices to prove Theorem \ref{theorem: main theorem specific} when $m/N$ has a limit. Therefore, in what follows, without loss of generality, we assume that $m/N$ has a limit in $[0,1]$. We do a notation overload and denote this limit as $\lambda$. Note that previously we used $\lambda$ in Section \ref{sec: the model}.

As mentioned in Section \ref{sec: asymptotic properties}, establishing the rate of Hellinger error decay for our $\hn=\hls$ is an important step of the proof. We establish that $\hls$ satisfies some consistency-type results, which will be used in obtaining the Hellinger rate. In particular, we will show that $\hls$ satisfies two conditions. These conditions, stated in terms of $\hn$, will pertain to the consistency requirement and the rate of Hellinger decay, respectively.

\begin{condition}
\label{condition: hn basic}
Suppose $\hn$ is an estimator of $g_0$. Let $\{\zeta_{m,n}\}_{m,n\geq 1}\subset\RR$  be any stochastic sequence such that $\zeta_{m,n}\to_p 0$ as $\min(m,n)\to\infty$.  Then $\hn$ satisfies the following properties as $m,n\to\infty$:
\begin{itemize}
    \item[A.] $\edint|\hn(x)-g_0(x)|dx\to_p 0.$
    \item[B.] For any compact set $K\subset\iint(\dom(\ps_0))$,  $\hln$ satisfies 
$\sup_{x\in K}|\hln(x+\zeta_{m,n})-\ps_0(x)|\to_p 0$.
\item[C.] For all $x\in\iint(\dom(\ps_0))$ such that $x$ is continuity point of $\ps_0'$, $\hln$ satisfies
\[\hln'(x+\zeta_{m,n})\to_p\ps_0'(x).\]
\end{itemize}
\end{condition}
Condition \ref{condition: hn basic} requires the $L_1$ consistency of $\hn$, and the pointwise consistency of $\hln$ and its derivative. Our next condition involves 
 the rate of decay of the Hellinger error $\H(\hn,g_0)$. 
\begin{condition}
\label{cond: hellinger rate}
Suppose $\hn$ is an estimator of $g_0$. There exists $p\in(0,1/2]$ so that $\H(\hn, g_0)=O_p(N^{-p})$.
\end{condition}
 Proposition \ref{prop: pooled estimator satisfies Conditions} below establishes that $\hls$ satisfies Condition \ref{condition: hn basic} and  Condition \ref{cond: hellinger rate} with $p=1/4$, implying that $\H(\hls,g_0)=O_p(n^{-1/4})$.
\begin{proposition}
\label{prop: pooled estimator satisfies Conditions}
Suppose $\min(m,n)\to\infty$. Then the pooled estimator $\hls$ satisfies Condition \ref{condition: hn basic} and  Condition \ref{cond: hellinger rate} with $p=1/4$ under the setup of Theorem \ref{theorem: main theorem specific}. 
\end{proposition}
  The next step is to prove the asymptotic efficiency of $\bDelta-\Delta_0$. To this end,  we prove a general theorem, Theorem \ref{theorem: main theorem}, 
 which establishes the asymptotic efficiency of $\hDelta$ for any given $\hn$, $\bDelta$, and $\bmu$, provided $\hn$ satisfies Conditions \ref{condition: hn basic} and \ref{cond: hellinger rate}, and $\bDelta$ and $\bmu$ are $\sqrt{N}$-consistent estimators. The $\sqrt N$-consistency assumption on the preliminary estimators is standard in the  two-sample location shift model literature \citep{beran,park}. Since sample averages are $\sqn$-consistent, our preliminary estimators $\bmu=\bX$ and $\bDelta=\bY-\bX$ are $\sqrt N$-consistent.




\begin{theorem}\label{theorem: main theorem}
Suppose the preliminary estimators $\bmu$ and $\bDelta$ are $\sqrt N$-consistent, $\hn$ satisfies Conditions \ref{condition: hn basic} and  \ref{cond: hellinger rate}, and $m/N\to \lambda$ $\in[0, 1]$. Further suppose  $\eta_{m,n}=O(N^{-2\varrho/5})$ where $\varrho\in(0,p]$ and  $p$ is as in Condition \ref{cond: hellinger rate}. Then the one step estimator in \eqref{definition: truncated one step estimator} satisfies
\[\sqrt{\dfrac{mn}{N}}(\hDelta-\Delta_0)\to_d \N(0,\Io^{-1}),\]
where  $\Io$ is the Fisher information.
\end{theorem}
  Note that Theorem \ref{theorem: main theorem specific} follows from  Theorem \ref{theorem: main theorem} and Proposition \ref{prop: pooled estimator satisfies Conditions}. Therefore, it suffices to prove Proposition \ref{prop: pooled estimator satisfies Conditions} and Theorem \ref{theorem: main theorem}. Their  proofs  can be found in  Sections \ref{sec: proof of the conditions} and \ref{app: proof of Theorem 1}, respectively.



\subsection{Notation}
We will require some new notation and terminologies for proving  Theorem \ref{app: proof of Theorem 1} and Proposition \ref{prop: pooled estimator satisfies Conditions}.
We list them below for convenience. 
Recall that $\etan$ denotes the truncation parameter as defined in Section~\ref{sec: truncated one step}.  We defined $\xia$ and $\xib$ to be $\tilde G_{m,n}^{-1}(\etan)$ and $\tilde G_{m,n}^{-1}(1-\etan)$, respectively. $\mathbb{F}_m$ denotes the empirical distribution of $X_1, \ldots, X_m$ and $\mathbb{H}_n$ denotes the empirical distribution function of $Y_1, \ldots, Y_n$. The  empirical process of the $X_i$'s will be denoted by $\mathbb G_m=\sqrt{m}(\Fm-F_0)$.
 The  probability measure corresponding to the joint distribution of $X_1,\ldots,X_m$ and $Y_1,\ldots,Y_n$ will be denoted by $\PP$. For a distribution function $G$, let $J(G)$ denote the set $\{x\ :\ 0<G(x)<1\}$.

 For  any function $h:\RR\mapsto\RR$, $||h||_k$ will denote the $L_k$ norm defined as
\[||h||_{k}=\lb\edint |h(x)|^k dx\rb^{1/k},\quad k\geq 1.\]
For  any signed measure $Q$ on $\RR$ and any $Q$-measurable  function $h$, 
$Q h$ will denote the integral $\int_{\RR} h dQ$ provided $h$ is integrable with respect to $Q$.
 For a  class of $Q$-measurable functions $\mathcal Z$ , we let $\|Q\|_{\mathcal Z}$ denote the supremum $\sup_{h\in\mathcal Z}\abs{Qh}$. For a measure $P$ on $\RR$, we define the $L_{P,k}$ norm of the function $h$  as
\[\|h\|_{P,k}=\lb\edint |h(x)|^k dP(x)\rb^{1/k},\quad k\geq 1.\]
For any class of $P$-measurable functions $\mathcal Z$, we will denote by $\|\mathcal Z\|_{P,k}$ the supremum  $\sup_{h\in\mathcal Z}\|h\|_{\PP,k}$.
 


For two distribution functions $F_1$ and $F_2$ with respective densities $f_1$ and $f_2$, the total variation distance between $F_1$ and $F_2$ is given by
$d_{TV}(F_1,F_2)=\|f_1-f_2\|_1/2$.
Following   \cite{villani2003}, page $75$, we define the Wasserstein distance between two measures  $\nu$ and $\nu'$ on $\RR$ as:
\begin{equation}\label{def: Wasserstein distance}
d_W(\nu,\nu')=\edint|F(x)-G(x)|dx,
\end{equation}
where $F$ and $G$ are the cumulative distribution functions corresponding to  $\nu$ and $\nu'$ respectively. By an abuse of notation, we will refer to the above distance by $d_W(F,G)$ in some instances.
  Suppose $\e>0$. For any function class $\mathcal Z$ and a norm $|\cdot|$, the bracketing entropy $N_{[\ ]}(\epsilon, \mathcal Z,| \cdot|)$ is as in Definition 2.1.6 on page 83 of \cite{wc}. The covering number $N(\epsilon, \mathcal Z,| \cdot|)$ is as in page 83 of \cite{wc}.


The Cartesian product of two sets $A$ and $B$ is denoted by $A\times B$. 
The closure of the set $A$ will be denoted by $\overline{A}$. For any function $h$, the indicator function of the event $h(x)\leq C$ is denoted as $1_{[h(x)\leq C]}$. For any set $A$, $1_A(x)$ is the indicator function of the event $x\in A$. The standard Gaussian density is denoted by $\varphi$.

In what follows, we use the following fact repeatedly, which also implies that a log-concave density is bounded above with finite moments. 
 \begin{fact}[Lemma 1 of \cite{theory}]
 \label{fact: Lemma 1 of theory paper}
 Let $f$ be a univariate log-concave density. Then all moments of $f$ exist. Moreover,  there exist constants $\alpha>0$ and $b\in\RR$ so that $f(x)\leq e^{-\alpha |x|+b} $ for all $x\in\RR$.
 \end{fact}
Most of the additional facts  used in the proofs are collected in Section \ref{sec: technical facts}.

\section{Proof of Proposition \ref{prop: pooled estimator satisfies Conditions}}
\label{sec: proof of the conditions}
Subsequent discussion in this section presupposes that the preliminary estimators are as  outlined in Section \ref{sec: preliminary estimators}, specifically, $\hn = \hat g_{m, n}^{\text{pool,sm}}$, $\bmu = \bX$, and $\bDelta = \bY - \bX$.
 We will prove Proposition \ref{prop: pooled estimator satisfies Conditions} in four steps. In the first step, we will prove that $\hl$ satisfies $\|\hl-g_0\|_1\as 0$. In the second step, we will show that $\hl$ satisfies Condition  \ref{cond: hellinger rate} with $p=1/4$. In the third step, we analyze the Hellinger error decay of $\hls$. In the fourth step, building on the above steps, we will show that $\hls$  satisfies Conditions \ref{condition: hn basic} and \ref{cond: hellinger rate}, thus finishing the proof of Proposition \ref{prop: pooled estimator satisfies Conditions}.
 
\subsection{Step 1: proving $\|\hl-g_0\|_1\as 0$}
\label{sec: prop: step 1}

In order to establish the convergence $\|\hl-g_0\|_1\as 0$, we will employ the projection theory framework developed by \cite{dumbreg}. Let us denote the empirical distribution function of the combined sample comprising the $X_i-\bmu$'s and the $Y_j-\bmu-\bDelta$'s as $\Fp$. Recall that  $\hl$ is the log-concave MLE based on this combined sample. As per the insights of \cite{dumbreg}, $\Gl$, the distribution function of $\hl$, is the projection of $\Fp$ onto the class  of the distribution functions with logconcave densities. This projection is with respect to the Kullback-Leibler (KL) divergence; for an in-depth exposition, please refer to \cite{dumbreg}. This projection operator adheres to a continuity property concerning Wasserstein distance. Specifically, we can draw upon Theorem 2.15 of \cite{dumbreg} to deduce that, for any sequence of distributions $F_n$, the convergence $d_W(F_n, G_0)\as 0$ implies the convergence $d_{TV}(F_n,G_0)\as 0$, provided that $G_0$ is non-degenerate with finite first moment. Here   $d_W$ denotes the Wasserstein distance and $d_{TV}$ denotes the total variation distance. In our case, the non-degeneracy of $G_0$ is immediate because it is a distribution function with continuous density. Also, all moments of $G_0$ is finite by  Fact \ref{fact: Lemma 1 of theory paper}. Consequently, to establish the convergence $|\hl-g_0|_1\as 0$, it is sufficient to show that $d_W(\Fp, G_0)\as 0$.


 To show the convergence $d_W(\Fp, G_0)\as 0$, it is enough to show that \citep[see Theorem 6.9 of][]{villani2009}  $\Fp$ weakly converges to  $G_0$ almost surely and 
 \begin{equation}
     \label{convergence: L1 of Fp and G0}
     \edint |x|d\Fp(x)\as \edint |x|dG_0(x).
 \end{equation}
 We will show the weak convergence first. To that end, we show a stronger result  which implies weak convergence. In particular, we show that $\Fp$ uniformly converges to $G_0$  with probability one. 
 Note that for any $x\in\RR$,
 \begin{align}
 \label{def: pooled empirical distribution}
     \Fp(x)=\frac{m\Fm(x+\bmu)+n\mathbb H_n(x+\bmu+\bDelta)}{m+n},
 \end{align}
 which implies for any $x\in\RR$,
 \[\abs{\Fp(x)-G_0(x)}\leq \abs{\Fm(x+\bmu)-G_0(x)}+\abs{\mathbb H_n(x+\bmu+\bDelta)-G_0(x)}.\]
Hence to show  $\Fp$  converges uniformly to $G_0$,  it suffices to  show that with probability one,  $\Fm(\cdot +\bmu)$ and $\mathbb H_n(\cdot+\bmu+\bDelta)$ uniformly converges 
to $G_0$ with probability one. To prove $\Fm(\cdot+\bmu)$ uniformly converges to $G_0$ almost surely, note that 
\begin{align}
\label{inprop: step 1: weak}
  \MoveEqLeft  \sup_{x\in\RR}\abs{\Fm(x+\bmu)-G_0(x)}\leq   \sup_{x\in\RR}\abs{\Fm(x+\bmu)-F_0(x+\bmu)}\nn\\
    &\ +\sup_{x\in\RR}\abs{F_0(x+\bmu)-G_0(x)},
\end{align}
whose first term converges to zero almost surely by the Glivenko-Cantelli Theorem because $F_0$ is continuous  \citep[cf. Theorem 2.1 of][]{vdv}. For the second term, note that  
\[\sup_{x\in\RR}|F_0(x+\bmu)-F_0(x+\mu_0)|\leq |\bmu-\mu_0|\sup_{x\in\RR}f_0(x)\]
which converges to zero almost surely because (a) $\sup_{x\in\RR}f_0(x)<\infty$ by Fact \ref{fact: Lemma 1 of theory paper} and (b) $|\bmu-\kmu|\as 0$ since $\bmu=\bX$ is a strongly consistent estimator of $\mu_0$. Since   $G_0(x)=F_0(x+\mu_0)$, it follows that $\sup_{x\in\RR}\abs{F_0(x+\bmu)-G_0(x)}\as 0$. Therefore, \eqref{inprop: step 1: weak} implies that $\sup_{x\in\RR}\abs{\Fm(x+\bmu)-G_0(x)}\as 0$.
Similarly, we can prove that with probability one, $\mathbb H_n(\cdot +\bmu+\bDelta)$ uniformly converges to $G_0$ almost surely. Hence the almost sure weak convergence of $\Fp$  to  $G_0$ follows from \eqref{def: pooled empirical distribution}.

To show $d_W(\Fp, G_0)\as 0$, it remains to prove \eqref{convergence: L1 of Fp and G0}. To show \eqref{convergence: L1 of Fp and G0}, we  first note that \eqref{def: pooled empirical distribution} implies
\begin{align}\label{inlemma: Fp wasserstein}
   \MoveEqLeft \edint \abs{x}d(\Fp(x)-G_0(x))dx= \frac{m}{m+n}\edint \abs{x}d(\Fm(x+\bmu)-G_0(x))dx\nn\\
    &\ +\frac{n}{m+n}\edint \abs{x}d(\mathbb H_n(x+\bmu+\overline\Delta_n)-G_0(x))dx.
\end{align}
Since we assume $m/(m+n), n/(m+n)\leq 1$, it is enough to show that 
\begin{gather}
    \edint \abs{x}d(\Fm(x+\bmu)-G_0(x))dx\as 0\quad\\
    \edint \abs{x}d(\mathbb H_n(x+\bmu+\overline\Delta_n)-G_0(x))dx\as 0.
\end{gather}
 We will prove the first convergence only because the second convergence will hold similarly. 
 To that end, observe that
 \begin{align*}
\MoveEqLeft     \abs{\edint |x|d\Fm(x+\bmu)-\edint |x|dG_0(x)}\\
     =&\ \abs{\edint |x-\bmu|d\Fm(x)-\edint |x-\kmu|dF_0(x)}\\
     =&\ \abs{\edint \slb |x-\bmu|-|x-\kmu|\srb d\Fm(x)+\edint |x-\kmu|d(\Fm-F_0)(x)}\\
     \leq &\ \abs{\edint \slb |x-\bmu|-|x-\kmu|\srb d\Fm(x)}+\abs{\edint |x-\kmu|d(\Fm-F_0)(x)}\\
     \leq & \edint \abs{ |x-\bmu|-|x-\kmu|} d\Fm(x)+\abs{\edint |x-\kmu|d(\Fm-F_0)(x)}\\
     \stackrel{(a)}{\leq }&\ \edint \abs{ \bmu-\kmu} d\Fm(x)+\abs{\edint |x-\kmu|d(\Fm-F_0)(x)}\\
     =&\ \abs{\bmu-\kmu}+\abs{\edint |x-\kmu|d(\Fm-F_0)(x)}
 \end{align*}
 where (a) follows from the triangle inequality. 
Since we take $\bmu=\bX$, by the strong law of large numbers, 
 $\bmu\as \kmu$. The strong law also implies that \[\edint |x-\kmu|d(\Fm-F_0)(x)\as 0.\]
 Therefore
 \[  \abs{\edint |x|d\Fm(x+\bmu)-\edint |x|dG_0(x)}\as 0.\]
 We can similarly show that
 \[\edint \abs{x}d\mathbb H_n(x+\bmu+\overline\Delta_n)-\edint \abs{x}dG_0(x)\as 0.\]
 Thus \eqref{inlemma: Fp wasserstein} is proved. Hence, we have shown that $d_W(\Gl,G_0)\as 0$, which completes the proof of $\|\hl-g_0\|_1\as 0$.

 
'

\subsection{Step 2: proving that $\hl$ satisfies Condition \ref{cond: hellinger rate} with $p=1/4$}
Suppose $Z_1,\ldots,Z_N$ is the pooled sample of the $X_i-\mu_0$'s and the $Y_j-\mu_0-\Delta_0$'s, that is $Z_i=X_i-\mu_0$ for $i=1,\ldots,m$ and $Z_i=Y_i-\mu_0-\Delta_0$ for $i=m+1,\ldots,N$. Note that $Z_1,\ldots,Z_N\iid G_0$.
Let us denote the empirical distribution function of the $Z_i$'s by $\Gp$.
Therefore
\[\Gp(x)=\frac{m\Fm(x+\mu_0)+n\Hn(x+\kmu+\kdelta)}{m+n}\quad\text{for all }x\in\RR.\]
Denote by $\hlk$ the log-concave MLE based on the sample $\{Z_1,\ldots,Z_N\}$.  

Using triangle inequality on Hellinger distance, $\H(\hl,g_0)\leq \H(\hl,\hlk)+\H(\hlk,g_0)$.  Since $\hlk$ is the log-concave MLE based on an i.i.d. sample with log-concave density $g_0$, the rate of decay of $\H(\hlk,g_0)$ can be obtained using existing results on log-concave MLEs. In particular, \cite{dossglobal} implies that  $\H(\hlk,g_0)=O_p(N^{-2/5})$. Hence, the proof will follow if we can show that $\H(\hl,\hlk)=O_p(N^{-1/4})$. To this end, we will use Theorem 2 of \cite{barber2020}.

Note that $\hlk$ is the density of the log-concave projection of $\Gp$ in the sense of \cite{dumbreg}, as explained in Section \ref{sec: prop: step 1}.
In contrast, $\hl$ is the density of the log-concave projection of $\Gpool$ defined in \eqref{def: pooled empirical distribution}. Theorem 2 of \cite{barber2020} bounds $ \H(\hl,\hlk)$ in terms of the Wasserstein distance between $\Gp$ and $\Gpool$. This theorem states that if $\Gp$ and $\Gpool$ are non-degenerate distributions with finite first moment, then there exists an absolute constant $C>0$ so that 
\begin{equation*}
    \H(\hl,\hlk)\leq C \left(\frac{d_W(\Gp,\Gpool)}{\epsilon_{\Gp}}\right)^{1/2}
\end{equation*}
where  for any distribution function $F$, $\epsilon_{F}$ is defined as $\e_F=E_F\left[\abs{U-E_F[U]}\right]$, with $U$ being a random variable with  distribution function $F$.

To show $  \H(\hl,\hlk)=O_p(N^{-1/4})$, it suffices to show that $\e_{\Gp}$ is bounded away from zero with probability  one and that $d_W(\Gp,\Gpool)=O_p(N^{-1/2})$. We will first show that $\e_{\Gp}$ is bounded away from zero. To this end, note that by triangle inequality, 
\[\abs{\e_F-E_F[|U|]}=\abs{E_F[\abs{U-E_F[U]}]-E_F[|U|]}\leq |E_F(U)|.\]
Thus $\e_F\geq \E_F[|U|]-|E_F[U]|$. Hence, 
\[\e_{\Gp}\geq \E_{\Gp}[|U|]-|E_{\Gp}[U]|.\]
Note that
\[\E_{\Gp}[|U|]=\edint |x|d\Gp(x)=\frac{\sum_{i=1}^{N}|Z_i|}{N}\as \edint |x|g_0(x)dx\]
by the strong law of large numbers because $Z_1,\ldots,Z_N$ are i.i.d. with common density $g_0$, which has finite first moment because of log-concavity (see Fact \ref{fact: Lemma 1 of theory paper}). The integral $\int |x| g_0(x)dx =0$ only if $g_0$ is identically $0$ almost everywhere. However, in that case $g_0$ can not be an absolutely continuous density function, which is required for the Fisher information to be finite (see the discussion preceeding \eqref{eq: finitieness of FI}). Therefore, $\int |x| g_0(x)dx >0$. On the other hand,
\[E_{\Gp}[U]=\frac{\sum_{i=1}^{N}Z_i}{N}\as \edint xg_0(x)dx,\]
by another application of the strong law of large numbers.  Now $g_0$ is centered because it is a member of $\mP_0$ defined in \eqref{def: mP 0}, which indicates that $\int xg_0(x)dx=0$. Thus, we have shown that 
\[\e_{\Gp}\geq \E_{\Gp}[|U|]-|E_{\Gp}[U]|\as \edint |x|g_0(x)dx>0\]
which establishes that $\e_{\Gp}$ is bounded away from zero with probability one. Thus it only remains to show that $d_W(\Gp,\Gpool)=O_p(N^{-1/2})$.

To show  $d_W(\Gp,\Gpool)=O_p(N^{-1/2})$, we will use an alternative definition of $d_W$ which is due to the Kantorovich-Rubinstein Duality Theorem
 \citep[cf. theorem 2.5 of][]{bobkovbig}. By this theorem the Wasserstein distance between two distribution functions $F_1$ and $F_2$ is
 \[d_W(F_1,F_2)=\sup_{h\in\text{Lip}_1}\bl\edint h(x)d(F_1-F_2) \bl,\]
 where $\text{Lip}_1$ is the class of all real-valued functions with Lipschitz constant one. Therefore,
 \[d_W(\Gp,\Gpool)=\sup_{h\in\text{Lip}_1}\bl\edint h(x)d(\Gp(x)-\Gpool(x)) \bl.\]
 Note that
 \begin{align*}
   \edint h(x)d\Gp(x)=&\ \frac{m\edint h(x)d\Fm(x+\mu_0)+n\edint h(x)d\Hn(x+\mu_0+\Delta_0)}{m+n} \\
   =&\ \frac{m\edint h(x-\kmu)d\Fm(x)+n\edint h(x-\kmu-\kdelta)d\Hn(x)}{m+n}.
 \end{align*}
 Similarly, we can show that 
 \begin{align*}
   \edint h(x)d\Gpool(x)= \frac{m\edint h(x-\bmu)d\Fm(x)+n\edint h(x-\bmu-\bDelta)d\Hn(x)}{m+n}.
 \end{align*}
Therefore,
\begin{align*}
 \MoveEqLeft \bl \edint h(x)d\Gpool(x) -   \edint h(x)d\Gp(x)\bl \leq  \frac{m}{m+n} \bl\edint \slb h(x-\kmu)-h(x-\bmu)\srb d\Fm(x)\bl\\
  &\ + \frac{n}{m+n} \bl\edint \slb h(x-\kmu-\kdelta)-h(x-\bmu-\bDelta)\srb d\Hn(x)\bl.
\end{align*}
 If $h$ is Lipschitz with Lipschitz constant one, then it follows that $|h(x-\kmu)-h(x-\bmu)|\leq |\bmu-\kmu|$. Thus it follows that
 \begin{align*}
 \MoveEqLeft \sup_{h\in\text{Lip}_1}\bl \edint h(x)d\Gpool(x) -   \edint h(x)d\Gp(x)\bl \leq  \frac{m}{m+n}|\bmu-\kmu|\\
 &\ +\frac{n}{m+n}|\bmu+\bDelta-\kmu-\kdelta|
\end{align*}
which is $O_p(N^{-1/2})$ by central limit theorem because $\bmu=\bX$ and $\bmu+\bDelta=\bY$. Hence, the proof of this step follows.

 \subsection{Step 3:  Hellinger error decomposition of $\hls$:}
 \label{lemmas on hellinger error}
 First, we will show that $ \H(\hls, g_0) \lesssim \sqrt{2} \H(\hl, g_0) + \hatlambda$. Since Step 2 gives $\H(\hl, g_0)=O_p(N^{-1/4})$, it will only remain to show  that  $\hatlambda=O_p(N^{-1/4})$. 
 \begin{lemma}
 \label{lemma: prop: hellinger decomposition}
    Under the setup of Theorem \ref{theorem: main theorem specific}, the density estimator $\hls$ defined in \eqref{def: smoothed pooled estimator} satisfies 
 \begin{equation*}
     \H(\hls, g_0) \lesssim \sqrt{2} \H(\hl, g_0) + \hatlambda.
 \end{equation*}
 \end{lemma}
 \begin{proof}
 Fact \ref{fact: dTV and hellinger} implies $ \H^2(\hls, g_0)  \le  d_{TV}(\hls, g_0)$, which is $1/2$ times
 \begin{equation*}
 \begin{split}
   \lVert \hls - g_0 \rVert_1 
    & \leq \hatlambda^{-1}\edint \abs{\edint \slb\hl(x-t)-g_0(x-t)\srb\varphi(t/\hatlambda)dt} dx\\
    & \quad + \hatlambda^{-1}\edint \abs{ \edint \slb g_0(x)-g_0(x-t)\srb \varphi(t/\hatlambda)dt} dx\\
    & \leq \|\hl-g_0\|_1+\hatlambda^{-1}\edint \abs{\edint \varphi(t/\hatlambda)\dint_{x-t}^x g_0'(z)dzdt }dx,
 \end{split}
 \end{equation*}
 whose first term is bounded by $ \sqrt{2}\mathcal{H}(\hl, g_0)$ by Fact \ref{fact: dTV and hellinger}, and the second term can be bounded as 
 \begin{align*}
     \MoveEqLeft  \hatlambda^{-1}\edint \abs{\edint \varphi(t/\hatlambda)\dint_{x-t}^x g_0'(z)dzdt }dx\\
   \leq &\ \hatlambda^{-1}\edint\lb \dint_{-\infty}^0 \dint_{x}^{x-t}  \varphi(t/\hatlambda)\abs{g_0'(z)}dzdt +\dint_{0}^\infty \dint_{x-t}^{x}  \varphi(t/\hatlambda)\abs{g_0'(z)}dzdt\rb dx\\
   = &\ \hatlambda^{-1}\edint\lb \dint_{0}^\infty \dint_{x}^{x+t}  \varphi(-t/\hatlambda)\abs{g_0'(z)}dzdt +\dint_{0}^\infty \dint_{x-t}^{x}  \varphi(t/\hatlambda)\abs{g_0'(z)}dzdt\rb dx\\
   \leq &\ \hatlambda^{-1}\edint \dint_{-\infty}^\infty \dint_{x-|t|}^{x+|t|}  \varphi(t/\hatlambda)\abs{g_0'(z)}dzdtdx\\
   =&\ \hatlambda^{-1}\edint \varphi(t/\hatlambda)\dint_{-\infty}^\infty \dint_{x-|t|}^{x+|t|}  \abs{g_0'(z)}dzdxdt\\
   =&\ 2\edint |t| \hatlambda^{-1}\varphi(t/\hatlambda)\edint |g_0'(z)|dz dt\\
   =&\ 4\hatlambda \E[|\ZZ|]\cdot\edint |g_0'(z)|dz,
 \end{align*}
 where $\ZZ\sim \N(0, 1)$. The proof follows noting  $\E[|\ZZ|]<\infty$ and  $  \dint |g_0'(z)| dz<\infty$ because 
 \begin{equation*}
  \dint |g_0'(z)| dz=     \edint |\psi_0'(z)| g_0(z) dz\le \left( \edint \psi_0'(z)^2g_0(z)dz \right)^{1/2} \left( \edint g_0(z)dz \right)^{1/2}= \sqrt{\mathcal{I}_{g_0}},
  \end{equation*}
 which is finite since $g_0\in\mP_0$.

 \end{proof}
 


Next we will find the rate of decay of the smoothing parameter $\hatlambda$ defined in \eqref{def: smoothing parameter}. 

\begin{lemma}
\label{lemma: prop: rate of hatlambda}
  Under the assumptions of Theorem \ref{theorem: main theorem specific}, the $\hatlambda$ defined in \eqref{def: smoothing parameter} satisfies $\hatlambda=O_p(N^{-1/4})$.  
\end{lemma}
 
 \begin{proof}[Proof of Lemma \ref{lemma: prop: rate of hatlambda}]

 
 The definition of $\hatlambda$ in \eqref{def: smoothing parameter} implies $\hatlambda^2=\widehat s_n^2-\widehat\sigma_n^2$. Here 
 \[\widehat\sigma_n^2=\edint z^2 \ghat(z)dz - \left( \edint z\ghat(z)dz \right)^2\]
 and $\widehat s_n^2$ is the sample variance of the combined sample of the $X_i-\bmu$'s and the $Y_j-\bDelta-\bmu$'s. 
 Recalling from Section \ref{sec: prop: step 1}  that the empirical distribution function of this combined sample was denoted to be $\Fp$, we obtain that 
 \[\widehat s_n^2=\edint z^2 d\Gpool(z) - \left( \edint z d\Gpool(z) \right)^2.\]
  Since the combined sample has sample-mean zero, it follows that $\int zd\Gpool(z)=0$.
Also, we constructed $\ghat$ so that $\int z\ghat(z)dz=0$.  Hence,
 \begin{equation}
 \label{def: bound: beta n}
 \begin{split}
     \hatlambda^2 &= \edint z^2\ghat(z)dz - \edint z^2 d\Gpool(z)\\
     &= \edint z^2 d\left( \Ghat(z)-\Gpool(z) \right)\\
     &\le \underbrace{\left\lvert \edint z^2 d\left( \Ghat(z)-\Gtrue(z) \right)\right\rvert}_{T_1} + \underbrace{\left\lvert \edint z^2 d\left( \Gpool(z)-\Gtrue(z) \right)\right\rvert}_{T_2}.
 \end{split}
 \end{equation}
 Using Cauchy-Schwartz inequality in the second step, we obtain that
 \begin{equation*}
 \begin{split}
   T_1=  & \left\lvert \edint z^2 d\left( \Ghat(z)-\Gtrue(z) \right)\right\rvert\\
     =& \left| \edint z^2 \left(\sqrt{\ghat}(z)-\sqrt{g_0}(z)  \right)\left( \sqrt{\ghat}(z)+\sqrt{g_0}(z)  \right)dz \right|\\
     \le& \left\{ \edint z^4 \left( \sqrt{\ghat}(z)+\sqrt{g_0}(z)  \right)^2 dz \right\}^{1/2} \H(\ghat, g_0)\\
     \le & \H(\ghat, g_0)\cdot \left\{ \edint 2z^4 \left( \ghat(z) + g_0(z) \right)dz \right\}^{1/2}\\
     =& \H(\ghat, g_0)\cdot \left\{ \edint 2z^4\ghat(z)dz + \edint 2z^4 g_0(z)dz \right\}^{1/2}\\
     \le & \H(\ghat, g_0)\cdot \left\{ \edint 2z^4\left| \ghat(z)-g_0(z) \right|dz + \edint 4z^4 g_0(z)dz \right\}^{1/2}.
 \end{split}
 \end{equation*}
 Since $g_0$ is a log-concave density, it has finite moments by Fact \ref{fact: Lemma 1 of theory paper}, which implies $\edint z^4 g_0(z)dz<\infty$. On the other hand, 
 \begin{equation*}
 \begin{split}
     \edint z^4\left| \ghat(z)-g_0(z) \right| dz\le \edint e^{2|z|}\left| \ghat(z)-g_0(z) \right| dz.
 \end{split}
 \end{equation*}
Now note that  Step 1 implies $\ghat$ converges to $g_0$ in total variation norm almost surely. Since convergence in total variation norm ensures weak convergence,  the distribution corresponding to $\ghat$ converges to $g_0$ weakly almost surely. Since $\ghat$'s are logconcave, the latter implies  \citep[Proposition 2][]{theory}
\[\edint e^{2|z|} \left| \ghat(z)-g_0(z) \right|dz\as 0.\]
 Therefore $\edint z^4\left| \ghat(z)-g_0(z) \right| \as 0$, and hence
 \begin{equation*}
     \left\lvert \edint z^2 d\left( \Ghat(z)-\Gtrue(z) \right)\right\rvert = O_p\left(  \H(\ghat, g_0)\right)=O_p(N^{-1/4}).
 \end{equation*}
 Therefore, we have shown that $T_1=O_p(N^{-1/4})$.
 Now we focus on the second term $T_2$ in \eqref{def: bound: beta n}, which equals
 \begin{align*}
  \MoveEqLeft \edint z^2 d\left( \Gpool(z)-\Gtrue(z) \right)=  \frac{m}{m+n}\underbrace{\edint z^2 d\left( \Fm(z+\bmu)-G_0(z) \right)}_{T_{21}} \\
     &\  +\frac{n}{m+n} \underbrace{\edint z^2 d\left( \Hn(z+\bmu+\bDelta)-G_0(z) \right)}_{T_{22}}.
\end{align*}
It suffices to show that $T_{21}=O_p(m^{-1/2})$ and $T_{22}=O_p(m^{-1/2})$. We will show that $T_{21}=O_p(N^{-1/4})$ because the proof of $T_{22}$ will follow similarly. Straightforward algebra shows that
\begin{align*}
    \MoveEqLeft \edint z^2 d\left( \Fm(z+\bmu)-G_0(z) \right)\\
    =&\  \edint z^2 d \Fm(z+\bmu)- \edint z^2 d F_0(z+\kmu) \\
    =&\  \edint (z-\bmu)^2 d \Fm(z)- \edint (z-\kmu)^2 d F_0(z)\\
    =&\  \edint (z-\kmu)^2 d(\Fm(z)- F_0(z))+\edint \{(z-\bmu)^2-(z-\kmu)^2\}d\Fm(z),
\end{align*}

whose first term is $O_p(m^{-1/2})$ because of central limit theorem and the second term equals
\[(\kmu-\bmu) \edint (2z-\bmu-\kmu)d\Fm(z)=-(\bmu-\mu)^2,\]
which is also of the order $O_p(m^{-1})$ because $\bmu=\bX$ is a $\sqrt m$-consistent estimator of $\kmu$ by central limit theorem. For the application of the first central limit theorem, $F_0$ is required to have a finite fourth moment. That is guaranteed for our $F_0$ because its all moments are finite by Fact \ref{fact: Lemma 1 of theory paper}. Therefore, we have shown that $T_{21}=O_p(m^{-1/2})$. Similarly, we can show that $T_{22}=O_p(m^{-1/2})$. Therefore, we have shown that $T_2=O_p(m^{-1/2})$. Since $T_1=O_p(N^{-1/4})$, the proof follows from \eqref{def: bound: beta n}.

\end{proof}

\subsection{Step 4: proving the conditions for $\hls$}
Lemmas \ref{lemma: prop: hellinger decomposition} and \ref{lemma: prop: rate of hatlambda} imply that $\H(\hls,g_0)=O_p(N^{-1/4})$, which directly proves Condition \ref{cond: hellinger rate}. Since this implies that  the Hellinger distance between these densities is $o_p(1)$, Fact \ref{fact: dTV and hellinger}  implies that the total variation distance between the corresponding distributions  approaches zero in probability. In other words, $\|\hls-g_0\|_1\to_p 0$. Next, we will show that $\|\hls(\cdot+\zeta_{m,n})-g_0\|_1\to_p 0$. To see this, note that 
\begin{align*}
  \|\hls(\cdot+\zeta_{m,n})-g_0\|_1\leq  \|\hls(\cdot+\zeta_{m,n})-g_0(\cdot+\zeta_{m,n})\|_1+\|g_0(\cdot+\zeta_{m,n})-g_0\|_1
\end{align*}
by triangle inequality. However, 
\[\|\hls(\cdot+\zeta_{m,n})-g_0(\cdot+\zeta_{m,n})\|_1=\|\hls-g_0\|_1=O_p(N^{-1/4}).\]
Also, another application of Fact \ref{fact: dTV and hellinger}  implies 
\[\|g_0(\cdot+\zeta_{m,n})-g_0\|_1\leq \sqrt{2}\H(g_0(\cdot+\zeta_{m,n}),g_0)\]
which is of the order $O_p(\zeta_{m,n})$ by Fact \ref{fact: helli fknot}. Since $\zeta_{m,n}=o_p(1)$, $\|g_0(\cdot+\zeta_{m,n})-g_0\|_1$ is of the order $o_p(1)$. Thus it follows that  $ \|\hls(\cdot+\zeta_{m,n})-g_0\|_1\to_p 0$. 

 Fact \ref {fact: convergence in probability to convergence almost surely} implies that given any subsequence of $\{m,n\}$,  we can find a further subsequence along which $\|\hls(\cdot+\zeta_{m,n})-g_0\|_1$ and $\zeta_{m,n}$ converges to zero almost surely. If we can show that along this subsequence $\|\hls(\cdot+\zeta_{m,n})-g_0\|_\infty\to 0$ with probability one, then  Fact \ref{fact: Shorack} would imply that 
 $\|\hls(\cdot+\zeta_{m,n})-g_0\|_\infty\to_p 0$ and part (A) of Condition \ref{condition: hn basic} will be proved. The same  holds for proving parts (B) and (C) of Condition \ref{condition: hn basic}. Therefore, we will assume that $\|\hls(\cdot+\zeta_{m,n})-g_0\|_1\as 0$ and $\zeta_{m,n}\as 0$. 

 The almost sure $L_1$ converges of $\hls(\cdot+\zeta_{m,n})$ to $g_0$ implies almost sure weak convergence of the corresponding distribution function to $G_0$. 
Since $\hls$ is log-concave,the above-mentioned weak convergence   implies  $\|\hls(\cdot+\zeta_{m,n})-g_0\|_\infty\as 0$ by Proposition 2(c) of \cite{cule2010}. Thus part (A) of Condition \ref{condition: hn basic} follows.  As a consequence, for any $K\subset \dom(\phi_0)$, 
\[\sup_{x\in K}|\log(\hls(\cdot+\zeta_{m,n}))-\psi_0|\as 0,\]
which implies part (B) of Condition \ref{condition: hn basic}. Part (C) of Condition \ref{condition: hn basic} follows from Part (B) by Theorem 25.7 of \cite{rockafellar}. Hence, we have proved that $\hls$ satisfies Condition \ref{condition: hn basic}.

\section{ Proof of Theorem \ref{theorem: main theorem}}
 \label{app: proof of Theorem 1}


We will prove the theorem only for the case when $\etan$ equals $CN^{-2p/5}$. This means we will show that  for any $p\in(0,1)$, if $\H(\hn, g_0)=O_p(N^{-p})$, then $\sqrt{\frac{mn}{N}}(\hDelta-\Delta_0)\to_d \N(0,\I^{-1})$ when $\etan=CN^{-2p/5}$.  Suppose $\etan=CN^{-2\varrho/5}$ with $\varrho\in(0,p]$. Note that $\H(\hn, g_0)=O_p(N^{-\varrho})$ trivially holds since
$\H(\hn, g_0)=O_p(N^{-p})$. Therefore, for $\etan=CN^{-2\varrho/5}$, the proof of  $\sqrt{\frac{mn}{N}}(\hDelta-\kdelta)\to_d \N(0,\I^{-1})$ will follow immediately from what we had just proved. Thus, it suffices to  take  $\etan=CN^{-2p/5}$ for the rest of the proof.

Equation \ref{definition: truncated one step estimator} implies that
\begin{align}
\label{intheorem: main: A B split}
 \hDelta - \bDelta=&\ \underbrace{\dint_{\xia+\bmu}^{\xib+\bmu} \dfrac{\hln'(x-\bmu)}{\hi(\etan)}d\Fm(x)\nonumber}_A\nn\\
&\ -\underbrace{\dint_{\xia+\bmu+\bDelta}^{\xib+\bmu+\bDelta} \dfrac{\hln'(y-\bmu-\bDelta)}{\hi(\etan)}d\Hn(y)}_B.
\end{align}
We will consider the decomposition of $A$ because the decomposition  of $B$ will follow similarly. To this end, we note that 
\begin{align}
\label{intheorem: main: A split}
 \MoveEqLeft \sqrt{m}A \cdot \hi(\etan)= \sqrt{m}\dint_{\xia+\bmu}^{\xib+\bmu} \hln'(x-\bmu)d\Fm(x)\nn\\
=& \underbrace{\sqrt{m}\dint_{\xia+\bmu}^{\xib+\bmu}(\hln'(x-\bmu)-\ps_0'(x-\mu_0))d(\Fm-F_0)(x)}_{T_{1\sN}}\nn\\
&\ +
\underbrace{\sqrt{m}\dint_{\xia+\bmu}^{\xib+\bmu}\hln'(x-\bmu)dF_0(x)}_{T_{2\sN}}\nn\\
 &\ +\underbrace{\sqrt{m}\dint_{\xia+\bmu}^{\xib+\bmu}\ps_0'(x-\mu_0)d(\Fm-F_0)(x)}_{T_{3\sN}}.
\end{align}


The proof of Theorem \ref{app: proof of Theorem 1} is divided into four  main steps. In the first step, we show that the term  $T_{1N}$ is $o_p(N^{-1/2})$ using  Donsker's theorem.  In the second step, we shall show that $T_{2N}$ has the following expression: 
 \[T_{2N}=\sqrt{m}\dint_{\xia}^{\xib}\hln'(z)g_0(z) dz+\sqrt{m}(\bmu-\kmu)\hin(\etan)+o_p(1).\]
 This term accounts for the bias due to the use of $\bmu$ instead of $\mu_0$ while computing the scores. The third step is to prove the asymptotic normality of $T_{3N}$. In the fourth step, we argue that the similar decomposition can be done on $B$ and then we will combine the  terms from $A$ and $B$ to get the asymptotic distribution of $\hDelta$.

  \subsubsection*{\textbf{First step: asymptotic negligibility of $\sqrt{m} T_{1N}$:}}
 

Recalling that we denoted the empirical process $\sqrt{m}(\Fm-F_0)$ by $\mathbb{G}_m$, we obtain 
 \begin{align}\label{intheorem: T1n: representation}
  T_{1N}= \sqrt{m}\dint_{\bmu+\xia}^{\bmu+\xib}\slb \hln'(x-\bmu)-\ph_0'(x)\srb d(\Fm-F_0)(x),
 \end{align}
 where $\ph_0$ denoted the concave function $\ph_0(x)=\ps_0(x-\mu_0)$. Let us denote  $\mathcal T_{m,n}= [\bmu+\xia,\bmu+\xib]$. Also let us define the function $h_{m,n}$ by 
 \begin{equation}\label{inlemma: def: main: hn}
     h_{m,n}(x)=(\hln'(x-\bmu)-\ph_0'(x))1_{\mathcal T_{m,n}}(x),\quad x\in\RR.
 \end{equation}
Then the expression of $T_{1N}$ in \eqref{intheorem: T1n: representation} rewrites as 
\[T_{1N}=\edint h_{m,n}(x)d\mathbb{G}_m(x).\]

  Since we want to apply the Donsker theorem, we need to find an appropriate class containing the function $h_{m,n}$. To this end, first note that $\hln'(x-\bmu)$ is bounded on $\mathcal T_{m,n}$. To see this, note that 
  Lemma~\ref{lemma: tilde psi prime bound} implies
 \begin{equation}\label{inlemma: t1: bounding hln prime}
 \sup_{x\in\mathcal T_{m,n}}|\hln'(x-\bmu)|=O_p(N^{p/5})
 \end{equation}
since  $\etan=O( N^{-2p/5})$. Next, we note that since $\hln$ is concave, $\hln'$ is a monotone function. However,  $x\mapsto \hln'(x) 1_{\mathcal T_{m,n}}(x)$ is not a monotone function. 

 Nonetheless, it is possible to extend the function $x\mapsto \hln'(x) 1_{\mathcal T_{m,n}}(x)$ to  the entire real  line $\RR$ in a manner that preserves the monotonicity of the resulting function, all while maintaining the same upper bound. This extension can be achieved by defining a new function $\widehat u_{m,n}$ which equals $\hln'$ on $\mathcal T_{m,n}$ and takes the values $\hln'(\bmu+\xia)$ and  $\hln'(\bmu+\xib)$ on the intervals $(-\infty,\bmu+\xia]$ and $[\bmu+\xib,\infty)$, respectively. We can also substitute $\hln'$ with $\widehat u_{m,n}$ in the definition of $h_{m,n}$, leading to
  \[h_{m,n}(x)=(\widehat u_{m,n}(x-\bmu)-{\phi}_0'(x))1_{\mathcal T_{m,n}}(x).\]
  

 For any $C>0$, we let
 \begin{equation}\label{intheorem: def: U n}
    \mathcal{U}_{N}(c)=\bigg\{u:\RR\mapsto[-c,c]\ \bl\ \ u\text{ is non-increasing}\bigg\}.
\end{equation}
 Now  define the class $\mathcal{M}_N(C)$ by
  \begin{align*}
 \mathcal{M}_N(C)=  \bigg\{h:\RR\mapsto  \RR\ \bl \ &\ h(x)=(u(x)-{\phi}_0'(x))1_{[r_1,r_2]}(x),\ u\in\mathcal{U}_{N}(M_N) \text{ where }\\
 &\ M_N=CN^{p/5},
 \|h\|_{P_0,2}\leq CN^{-2p/5}(\log N)^{3},\ \ \|h\|_\infty\leq M_N,\\
 &\ [r_1,r_2]\subset [\Delta_0-C\log N,\Delta_0+C\log N]\cap\iint(\dom(\phi_0)) \bigg \}.
 \end{align*}
Next, we show that $h_{m,n}\in\mathcal \M_N(C)$ with high probability for all sufficiently large $N$.
To this end, we note that 
\[\sup_{x\in\mathcal T_{m,n}}|\phi'_0(x)|=\sup_{x\in [\bmu-\mu_0+\xia,\bmu-\mu_0+\xib]}|\psi'_0(x)|.\]
Also since $\etan=O(N^{-2p/5})$, Lemma~\ref{lemma: bound: psi knot prime}  implies that 
 \begin{equation}\label{intheorem: boundedness of psp}
 \sup_{x\in \mathcal T_{m,n}}|\phi_0'(x)|=O_p(\log N),
 \end{equation}
 which, in conjunction with  \eqref{inlemma: t1: bounding hln prime} and \eqref{intheorem: boundedness of psp} imply that $\|h_{m,n}\|_{\infty}=O_p(N^{p/5})$. Since $\|\widehat u_{m,n}\|_{\infty}=O_p(N^{p/5})$, $\widehat u_{m,n}(\mathord{\cdot}-\bmu)\in \mathcal{U}_{N}(CN^{p/5})$ with high probability when  $C$ is  sufficiently large.
On the other hand, Lemma~\ref{Lemma: L2 norm of hn} yields that  $\|h_{m,n}\|_{P_0,2}=O_p(N^{-2p/5}(\log N)^{3})$.
Finally, 
\[
\lim_{n\to\infty}  \PP\slb  \mathcal T_{m,n}\subset [\mu_0-C\log n,\mu_0+C\log n]\cap \iint(\dom(\phi_0))\srb= 1
\]
by  Lemma~\ref{lemma: An inclusion}.
Therefore, it follows that given $t>0$, we can choose $C>0$ so large such that $\PP(h_{m,n}\in\mathcal{M}_N(C))>1-t$. Let us denote  $M_N=CN^{p/5}$.



 Our next task is to find the bracketing entropy of $\M_N(C)$ when $C$ is a large positive number. For any probability measure $Q$ and any $\e>0$, the bracketing entropy of $\mathcal{U}_{N}(M_N)$ can be bounded as
 \begin{equation}\label{inlemma: finite entropy increasing}
\log N_{[\ ]}(\e,\mathcal{U}_{N}(M_N),L_2(Q))\leq C'M_N\e^{-1}
   \end{equation}
   using 
 Theorem~2.7.5 of \cite{wc} (pp. $159$). Here  $C'>0$ is an absolute constant.   On the other hand,  Theorem~2.7.5 of \cite{wc} entails that the class  $\mathcal{F}_I$ of all indicator functions of  the form $1_{[z_1,z_2]}$ with $z_1\leq z_2$, $z_1,z_2\in\RR$,  satisfies
   \begin{equation}\label{inlemma: finite entropy indicator functions}
 N_{[\  ]}(\e,\mathcal{F}_I,L_2(Q))\leq C' 2\e^{-1}, 
   \end{equation}
which, in conjuction with \eqref{inlemma: finite entropy increasing},  leads to 
\[\log N_{[\ ]}(\e,\mathcal{M}_N(C),L_2(P_0))\lesssim {M_N}{\e}^{-1}.\]

For $x<1$, the bracketing integral takes the form 
\begin{align*}
    \mathcal{J}_{[\ ]}(x,\mathcal{M}_N(C),L_2(P_0))=&\ \dint_{0}^{x}\sqrt{1+\log N_{[\ ]}(\e,\mathcal{M}_N(C),L_2(P_0))}d\e\\
    \lesssim &\  2M_N\dint_{0}^{x/M_N}\e^{-1/2}d\e \lesssim \sqrt{x M_N}.
\end{align*}
Also, 
 \[\|\mathcal{M}_N(C)\|_{P_0,2}=\sup_{h\in\mathcal{M}_N(C)}\|h\|_{P_0,2}=CN^{-2p/5}(\log N)^{3}.\]
Letting $K_N=CN^{-2p/5}(\log N)^{3}$ and using 
 Fact~\ref{fact: empirical process of m-p's}, 
 we obtain that $\|\mathbb{G}_m\|_{\mathcal{M}_N(C)}=\sup_{h\in \mathcal{M}_N(C)}|\mathbb{G}_mh|$ satisfies 
\begin{align*}
   \E\|\mathbb{G}_m\|_{\mathcal{M}_N(C)}\lesssim &\   \mathcal{J}_{[\ ]}(K_N,\mathcal{M}_N(C),L_2(P_0))\slb 1+\frac{  \mathcal{J}_{[\ ]}(K_N,\mathcal{M}_N(C),L_2(P_0))}{K_N^2\sqrt{N}}M_N\srb\\
   \lesssim &\  \sqrt{K_N M_N}+K_N^{-1}M_N^2N^{-1/2}.
   \end{align*}
Since $K_N=CN^{-2p/5}(\log N)^3$ and $M_N=CN^{p/5}$, 
\begin{align*}
\sqrt{K_N M_N} =&\ \sqrt{C^2 N^{p/5} N^{-2p/5} (\log N)^{3}} = CN^{-p/10} (\log N)^{3/2} = o(1),\\
\text{and}\quad\frac{M_N^2}{K_N\sqn} = &\ \frac{C^2N^{2p/5}}{CN^{-2p/5} (\log N)^{3}\sqrt{N}}
= \frac{CN^{4p/5} N^{-1/2}}{(\log N)^3} 
= \frac{C N^{(8p-5)/10}}{(\log N)^{3}} 
= o(1),
\end{align*}
where the last step follows since  $p\leq 1/2$ by Condition \ref{cond: hellinger rate}. 


Let us fix $t'>0$  and $\xi>0$. Suppose $C$ is so large such that $\PP( h_{m,n}\notin \mathcal{M}_N(C))<\xi/2$. Then we obtain that 
\begin{align*}
  \MoveEqLeft  \PP\lb \edint h_{m,n}(x)d\mathbb{G}_m(x)>t' \rb\\
  \leq &\  \PP\lb \edint h_{m,n}(x)d\mathbb{G}_m(x)>t' , h_{m,n}\in \mathcal{M}_N(C)\rb+\PP\lb h_{m,n}\notin \mathcal{M}_N(C)\rb\\
  \stackrel{(a)}{\leq} &\  E\lbt\sup_{h\in\mathcal{M}_N(C)}\bl\edint h(x)d\mathbb{G}_m(x)\bl\rbt/t'+\xi/2
\end{align*}
by Markov's inequality. Therefore 
\[ \PP\lb \edint h_{m,n}(x)d\mathbb{G}_m(x)>t' \rb=O((\log N)^{3/2}N^{-p/10})/t'+\xi/2,\]
which is less than $\xi$ for sufficiently large $N$. 

Since $t'$ and $\xi$ are arbitrary, it follows that $\int h_{m,n} d\mathbb{G}_m$ is $o_p(1)$, which completes the proof of $ T_{1N}=o_p(1)$.



\subsubsection*{\textbf{Second step: asymptotic limit of $T_{2N}/(\mu_0-\bmu)$:}} Letting  $\td = \mu_0 - \bmu$, $T_{2N}$ writes as 
\begin{equation*}
\begin{split}
\MoveEqLeft \sqrt{m}\dint_{\xia+\bmu}^{\xib+\bmu}\hln'(x-\bmu)dF_0(x)= \sqrt{m}\dint_{\xia}^{\xib}\hln'(z)g_0(z) dz\\
    &\ + \sqrt{m} \hin(\etan)\td\underbrace{\dint_{\xia}^{\xib}\dfrac{\hln'(z)}{\td\hin(\etan)}\lb g_0(z-\td)-g_0(z)\rb dz}_{T'_{2N}}.
\end{split}
\end{equation*}
 Letting $\mathcal A_{m,n}=[\xia,\xib-\td]$, we write 
 $T_{2N}'$  as
\begin{align}\label{intheorem:t2n:representation}
\MoveEqLeft \dint_{\xia}^{\xib}\hln'(z)\dfrac{\lb g_0(z-\td)-g_0(z)\rb}{\td\hin(\etan)} dz\ = \dint_{\xia}^{\xib}\hln'(z)\dfrac{\dint_{z}^{z-\td}g_0'(t)dt}{\td\hin(\etan)} dz\nn\\
= &\ -\dint_{\RR}1_{\mathcal A_{m,n}}(t)g_0'(t)\dfrac{\dint_{t}^{t+\td}\hln'(z)dz}{\td\hin(\etan)}(t)dt
\end{align}
 by Fubini's Theorem. Here Fubini's Theorem applies because $g_0$ is absolutely continuous, which follows since $\Io<\infty$ (see Section \ref{sec:intro}). From Lemma B.1 of \cite{laha2021adaptive} it follows that  the integral on the right hand side of \eqref{intheorem:t2n:representation} converges in probability to $-1$. Although the setup used in  Lemma B.1 of \cite{laha2021adaptive} mentions the symmetry of $g_0$, the  proof Lemma B.1 of \cite{laha2021adaptive} just requires $g_0$ to be log-concave with finite Fisher information $\mathcal I_{g_0}$. Therefore,  
\[T_{2N}=\sqrt{m}\dint_{\xia}^{\xib}\hln'(z)g_0(z) dz- \sqrt{m} \hin(\etan)\td(1+o_p(1)).\]
Recall that $\sqrt{m}\td=\sqrt{m}(\kmu-\bmu)$ is of the order $O_p(1)$ by our assumption on $\bmu$. Since $\hin(\etan)\to_p\Io$ by Lemma \ref{lemma: consistency of FI: 2}, we obtain $\hin(\etan)=O_p(1)$. Therefore, 
 \begin{align*}
   \sqrt{m}\td \hin(\etan)(1+o_p(1))=\sqrt{m}\td\hin(\etan)+o_p(1),
 \end{align*}
 Thus we have shown that the bias term $T_{2N}$ equals
 \[T_{2N}=\sqrt{m}\dint_{\xia}^{\xib}\hln'(z)g_0(z) dz+\sqrt{m}(\bmu-\kmu)\hin(\etan)+o_p(1)\]

\subsubsection*{\textbf{Third step: showing the asymptotic normality of $T_{3N}$:}}
Letting $C_{m,n}=(-\infty,\bmu+\xia]\cup[\bmu+\xib,\infty]$, we note that  $T_{3N}$ writes as 
\begin{align}\label{intheorem: T3n split new}
T_{3N}=&\ \sqrt{m}\int_{\bmu+\xia}^{\bmu+\xib}\ps_0'(x-\mu_0)d(\mathbb{F}_m-F_0)(x)\nonumber\nn\\
=&\  \int_{-\infty}^{\infty}\ps_0'(x-\mu_0)d\mathbb{G}_m(x) -\dint_{C_{m,n}} \ps_0'(x-\mu_0)d\mathbb{G}_m(x).
\end{align}
An application of the the central limit Theorem yields
\[\edint\ps_0'(x-\mu_0)d\mathbb{G}_m(x)\to_d \N(0,\I).\]
Here the central limit theorem applies because the second moment of $\psi_0'(X-\mu_0)$ exists when $X$ has the distribution function $F_0$. 

Hence, the proof will be complete if we can show that 
\begin{align}
\label{intheorem: main: step 2}
  \edint 1_{C_{m,n}}(x)\ps_0'(x-\mu_0)^2d\mathbb G_m(x)=o_p(1).
\end{align}
To this end, we observe that the function $1_{C_{m,n}^c}=1_{[\widehat\xi_1,\widehat\xi_2]}$ belongs to the class of all indicator functions of the form $1_{[z_1,z_2]}$ where $z_1\leq z_2$, $z_1,z_2\in\RR$. The latter class is Donsker by  \eqref{inlemma: finite entropy indicator functions}. Then  Theorem 2.1 of \cite{epindex} implies \eqref{intheorem: main: step 2} holds 
if 
\[\edint 1_{C_{m,n}}(x)\ps_0'(x-\mu_0)^2f_0(x)dx\to_p 0.\]

To show the above, we will exploit the fact that  $\I<\infty$. Fact~\ref{fact: condition for integrability}  implies that given any $\e>0$, we can find a $\sigma>0$ so small such that 
 $\int_{\mathcal B}\ps_0'^2(x-\mu_0)f_0(x)dx<\e$ for any $P_0$-measurable set $\mathcal{B}\subset\RR$ provided $P_0(\mathcal B)<\sigma$. 
 Thus it is enough to  show that $\int_{\mathcal C_{m,n}}f_0(x)dx=o_p(1)$. 

 Suppose $G_0^{-1}(0)>-\infty$ and $G_0^{-1}(1)<\infty$. 
 Since $\bmu\to_p\mu_0$, $\xia\to_p G_0^{-1}(0)$ and $ \xib\to_p G_0^{-1}(1)$ by Lemma~\ref{lemma: xi: xi goes to 1}, and $F_0$ is continuous, it follows that
\begin{align*}
   \dint_{C_{m,n}}f_0(x)dx=&\ 1-F_0(\bmu+\xib)+F_0(\bmu+\xia) \\
   \to_p &\  1-F_0(\mu_0+G_0^{-1}(1))+F_0(\mu_0+G_0^{-1}(0))
\end{align*}
by continuous mapping theorem. Since $\mu_0+G_0^{-1}(1)=F_0^{-1}(1)$
  and $\mu_0+G_0^{-1}(0)=F_0^{-1}(0)$, it is clear that  $\int_{\mathcal C_{m,n}}f_0(x)dx=o_p(1)$ in this case. If $G_0^{-1}(0)=-\infty$ or $G_0^{-1}(1)=\infty$, then continuous mapping theorem may not be applicable. However, in these cases, $\xia\to -\infty$ or $\xib\to\infty$, respectively,  with probability approaching one. Using the above, we can show that $\int_{\mathcal C_{m,n}}f_0(x)dx=o_p(1)$ still holds in these cases. Therefore, we have shown that 
 \[T_{3N}=\edint\ps_0'(x-\mu_0)d\mathbb{G}_m(x)+o_p(1),\]
 where $\dint\ps_0'(x-\mu_0)d\mathbb{G}_m(x)\to_d \N(0,\Io)$ by the central limit theorem. 


\subsection*{Fourth step: combining $A$ and $B$}
Combining the last three steps with \eqref{intheorem: main: A split}, we obtain that 
\begin{align*}
\MoveEqLeft \sqrt{\dfrac{mn}{N}}A \hin(\etan)= T_{1N}+T_{2N}+T_{3N}\\
=&\ \sqrt{\dfrac{mn}{N}}(\bmu-\mu_0)\hin(\etan)+\sqrt{\dfrac{mn}{N}}\dint_{\xia}^{\xib}\hln'(x)g_0(x)dx\\
&\ +\sqrt{\frac{n}{N}}\edint\ps_0'(x-\mu_0)d\mathbb{G}_m(x)+o_p(1).
\end{align*}
Since $\sqrt{m}\int\ps_0'(x-\mu_0)d\mathbb{G}_m(x)\to_d \N(0,\Io)$, the random variable $\sqrt{m}\int\ps_0'(x-\mu_0)d\mathbb{G}_m(x)$ is $O_p(1)$. 
Also noting $\sqrt{n/N}= \sqrt{1-\lambda}+o(1)$, we deduce that 
\[\sqrt{\frac{mn}{N}}\edint\ps_0'(x-\mu_0)d\mathbb{G}_m(x)=\sqrt{(1-\lambda)m}\edint\ps_0'(x-\mu_0)d\mathbb{G}_m(x)+o_p(1)\]
Recalling $\mathbb G_m=\sqrt{m}(\Fm-F_0)$, we thus obtain that
\begin{align*}
 \sqrt{\dfrac{mn}{N}}A \hin(\etan)= &\ \sqrt{\dfrac{mn}{N}}(\bmu-\mu_0)\hin(\etan)+\sqrt{\dfrac{mn}{N}}\dint_{\xia}^{\xib}\hln'(x)g_0(x)dx\\
&\ +\sqrt{(1-\lambda)m}\edint\ps_0'(x-\mu_0)d(\Fm(x)-F_0(x))+o_p(1)
\end{align*}
By symmetry, the term $B$ defined in \eqref{intheorem: main: A B split} satisfies 
\begin{align*}
\MoveEqLeft  \sqrt{\dfrac{mn}{N}}B \hin(\etan)=  \sqrt{\dfrac{mn}{N}}(\bmu+\bDelta-\mu_0-\kdelta)\hin(\etan)\\
 &\ +\sqrt{\dfrac{mn}{N}}\dint_{\xia}^{\xib}\hln'(x)g_0(x)dx +\sqrt{\lambda n}\edint\ps_0'(x-\mu_0-\kdelta)d(\Hn(x)-H_0(x))+o_p(1).
\end{align*} 
Note that
\[\sqrt{n}\edint\ps_0'(x-\mu_0-\kdelta)d(\Hn(x)-H_0(x))\to_d \N(0,\Io).\]
Thus
\begin{align*}
\MoveEqLeft \sqrt{\frac{mn}{N}}\hin(\etan)(A-B)
= \sqrt{(1-\lambda)m}\edint\ps_0'(x-\mu_0)d(\Fm(x)-F_0(x))\\
&\ - \sqrt{\lambda n}\edint\ps_0'(x-\mu_0-\kdelta)d(\Hn(x)-H_0(x)) -\sqrt{\dfrac{mn}{N}}(\bDelta-\Delta_0)\hin(\etan) +o_p(1).
\end{align*}
Let us denote
\begin{align*}
  \ZZ_{m,n}= &\ \sqrt{(1-\lambda)m}\edint\ps_0'(x-\mu_0)d(\Fm(x)-F_0(x))\\
   &\ - \sqrt{\lambda n}\edint\ps_0'(x-\mu_0-\kdelta)d(\Hn(x)-H_0(x)). 
\end{align*}
Since the $X$-sample and $Y$-sample are independent, and $\lambda\in[0,1]$, it follows that 
$
 \ZZ_{m,n} \to_d \N(0,\Io). 
$
Thus we have obtained that
\[\sqrt{\frac{mn}{N}}(A-B)=\frac{\ZZ_{m,n}}{\hin(\etan)}-\sqrt{\dfrac{mn}{N}}(\bDelta-\Delta_0) +o_p(1)\]
From \eqref{intheorem: main: A B split} we obtain that 
\[(\sqrt{mn/N})(\hDelta-\bDelta)=(\sqrt{mn/N})(A-B),\]
leading to
\[\sqrt{\frac{mn}{N}}(\hDelta-\bDelta)=\frac{\mathbb Z_{m,n}}{\hin(\etan)}-\sqrt{\frac{mn}{N}}(\bDelta-\kdelta)+o_p(1).\]
A change of sides yield
\[\sqrt{\frac{mn}{N}}(\hDelta-\kdelta)=\frac{\mathbb Z_{m,n}}{\hin(\etan)}+o_p(1).\]
Since  Lemma \ref{lemma: consistency of FI: 2} implies $\hin(\etan)\to_d\Io$, we obtain  $\ZZ_{m,n}/\hin(\etan)\to_d \N(0,\Io^{-1})$ by Slutskey's Lemma. Therefore, the proof follows.

\hfill $\Box$
 \

\section{Auxilliary lemmas for proving Theorem \ref{theorem: main theorem}}
\label{secpf: auxilliary} 

 


 In this section, we will prove some auxiliary lemmas required for proving Theorem \ref{theorem: main theorem}.
 We will assume that $g_0$ satisfies Assumption \ref{assump: L} in all the lemmas below. 
 
Some of the lemmas in this section will rely on  quantile functions. We will repeatedly use some common results regarding the quantile function, which is the genralized of a distribution function. We list these results below for the sake of convenience. First, we state a Fact collecting some common properties of the quantile function. 
\begin{fact}
\label{fact: inverse common}
    Suppose $F$ is a distribution function on $\RR$, $x\in\RR$, and $t\in(0,1)$. Then the following assertions hold. 
    \begin{enumerate}[label=(\roman*)]
        \item $F(F^{-1}(t))=t$ if $F$ is continuous.
        \item $F^{-1}(F(x))\leq x$ if and only if $x\geq F^{-1}(0)$.
        \item $F^{-1}(F(x))= x$ if $F$ is strictly increasing in a neighborhood of $x$.
        \item $F^{-1}$ is strictly increasing on $(0,1)$ if and only if $F$ is continuous. 
    \end{enumerate}
\end{fact}
For a proof of Fact \ref{fact: inverse common}, see Lemma A.3 of \cite{bobkovbig}.  The next fact pertains to the diffrentiability of the quantile function $F^{-1}$.
   
\begin{fact}[Proposition A.18 of \cite{bobkovbig}]
 \label{fact: bobkov big}
 Consider a density $f$ supported on an (potentially unbounded) interval and let $F$ be the corresponding distribution function. Then $F^{-1}$ is strictly increasing, and  $F^{-1}(q_2)-F^{-1}(q_1)=\int_{q_1}^{q_2}dt/f(F^{-1}(t))$ for all $0<q_1<q_2<1$.
 \end{fact}

   Suppose  $f$ is a log-concave density with distribution function $F$.    Since $F$ is continuous,  $F^{-1}$ is  strictly increasing on $(0,1)$  and $F(F^{-1}(t))=t$ for any $t\in(0,1)$ by Fact \ref{fact: inverse common}.   Note that  $\supp(f)=\dom(\log f)$.  Since $\log f$ is concave, its domain is an interval, which implies $\supp(f)$ is also an interval \citep[cf. pp. 74 of][]{hiriart2004}. Then $f>0$ on $J(F)=\{x\in\RR: 0<F(x)<1\}$ \citep[cf. Theorem 1(iv) of][]{dumbgen2017}.  Therefore, $J(F)=\iint(\supp(f))=\iint(\dom(\psi_0))$, and $F$ is strictly increasing on $J(F)$. Hence, for any $x\in J(F)$, $F^{-1}(F(x))=x$ by Fact  \ref{fact: inverse common} (iii).
 Since the density $f$ is supported on an interval, by
Fact \ref{fact: bobkov big},   $F^{-1}$ is strictly increasing and   differentiable with derivative $1/f(F^{-1}(t))$ on $(0,1)$.
   
 Some of the upcoming lemmas  will use lemmas from Appendix B.2 of  \cite{laha2021adaptive}.  \cite{laha2021adaptive} considers observations from a symmetric log-concave density  and defines  $\xin$ as $\tilde{G}_n^{-1}(1 - \eta_n)$, where $\tilde{G}_n$ is their analogue of the estimated distribution function $\tilde G_{m,n}$ and $\eta_n$ denotes the truncation level, i.e., it is an analogue of $\etan$. Therefore, when we mention that a lemma follows from \cite{laha2021adaptive}, the reader needs to replace $\xin$ with $\xib$ and $-\xin$ with $\xia$ in the original proof.  In \cite{laha2021adaptive}'s setup, $g_0$ satisfies the same assumptions as required by our setup. The only additional assumption on $g_0$ made in \cite{laha2021adaptive} is that of symmetry.  However,  the proofs of the lemmas cited in this section do not use the symmetry of $g_0$  and, hence, are applicable to our setup.

 \subsubsection{\textbf{Lemmas on $\xia, \xib$:}}
 \label{sec: lemmas on xi n}

Recall that we defined $\xia$ and $\xib$ to be $\tilde G_{m,n}^{-1}(\etan)$ and $\tilde G_{m,n}^{-1}(1-\etan)$, respectively, in Section \ref{sec: method}. The choice of $\etan$ and $\hn$ corresponding to $\tilde G_{m,n}$ should be clear from the context. 
 \begin{lemma}
 \label{lemma: xi: xi is in dop psi knot}
Suppose  $\hn\in\mathcal{LC}$ satisfies  Conditions~\ref{condition: hn basic} and \ref{cond: hellinger rate}. Then
\[\PP\slb [\xia,\xib]\subset \mathrm{int}(\dom(\psi_0))\srb\to 1,\]
provided $\etan=C N^{-2p/5}$, where $C>0$ and $p$ is as in Condition~\ref{cond: hellinger rate}. 

 \end{lemma}
 
 \begin{proof}[Proof of Lemma~\ref{lemma: xi: xi is in dop psi knot}]
 
 We will only prove that $\PP(\xia\in \iint(dom(\psi_0)))\to 1 $ because the proof of $\PP(\xib\in \iint(dom(\psi_0)))\to 1 $ will follow similarly. We will first show that $\xia\in\iint(J(G_0))$ with high probability for large $m$ and $n$. To this end, observe that 
  \[|G_0(\xia)-\tilde G_{m,n}(\xia)|\leq d_{TV}(G_0,\tilde G_{m,n})\stackrel{(a)}{\leq} \sqrt{2}\H(\hn,g_0),\]
  where  step (a) follows from Fact~\ref{fact: dTV and hellinger}.
 Note that for any $q\in(0,1)$ and  any continuous distribution function $F$  over $\RR$,
$F(F^{-1}(q))= q$. Hence, $\tilde G_{m,n}(\xia)=\tilde G_{m,n}(\tilde G_{m,n}^{-1}(\etan))=\etan$. Therefore, $G_0(\xia)\geq \tilde G_{m,n}(\xia)-\sqrt{2}\mathcal{H}(\hn, g_0)= \etan - \sqrt{2}\mathcal{H}(\hn, g_0)$. Moreover, noting  $\mathcal{H}(\hn, g_0) = O_p(N^{-p})$ from Condition~\ref{cond: hellinger rate} and $\etan=CN ^{-2p/5}\gg N^{-p}$, we derive $\PP(G_0(\xia)\geq \etan/2)\to 1$. On the other hand, 
\[G_0(\xia)\leq \tilde G_{m,n}(\xia)+\sqrt{2}\H(\hn,g_0)=\etan+\sqrt{2}\H(\hn,g_0),\]
which is less than $1/2$ with  probability approaching one. Therefore, $\PP(\xia\leq G_0^{-1}(1/2))\to 1$. Therefore, $\xia\in\iint(J(G_0))$ with probability approaching one.

 We already mentioned that  $\iint(J(G_0))=\iint(\dom(\psi_0))$ for log-concave $g_0$. Therefore the proof of  $\PP(\xia\in \iint(dom(\psi_0)))\to 1 $ is complete. The other part will follow similarly by establishing $G_0(\xib)\leq \tilde G_{m,n} (\xib) + O_p(N^{-p})$. 

 \end{proof}
 
 \begin{lemma}
 \label{lemma: xi: xi goes to 1}
 Under the setup of Lemma~\ref{lemma: xi: xi is in dop psi knot}, $\xia\to_p G_0^{-1}(0)$ and $\xib\to_p G_0^{-1}(1)$ as $\etan\to 0$.
 \end{lemma}
 \begin{proof}[Proof of Lemma~\ref{lemma: xi: xi goes to 1}]

We will show the proof for $\xib$ because the proof for $\xia$ will follow similarly. 
  Suppose, if possible, 
 $\xib\to_p G_0^{-1}(1)$ does not hold. We will consider two cases: (a) $G_0^{-1}(1)<\infty$ and (b) $G_0^{-1}(1)=\infty$.

 \textbf{Case (a):}
 Since  $\xib\to_p G_0^{-1}(1)$ does not hold, we can find an $\epsilon'>0$ and a subsequence 
 $\{m_k,n_k\}\subset \{m,n\}$ 
 so that 
 $\liminf_{k\to\infty} \PP(|\widehat{\xi}_{2,k}-G_0^{-1}(1)|>\epsilon)>0$ for all $\epsilon\leq \epsilon'$ where the quantile $\widehat{\xi}_{2,k}=\xib(m_k,n_k)$  corresponds to $(m_k,n_k)$. 
 Since Lemma~\ref{lemma: xi: xi is in dop psi knot} implies  $\PP(\widehat{\xi}_{2,k}\leq G_0^{-1}(1))\to 1$, it follows that 
 \[\liminf_{k\to\infty} \PP(\widehat{\xi}_{2,k}<G_0^{-1}(1)-\epsilon)>0.\]
 Now we show that we can choose $\e$ so small such that  $G_0^{-1}(t)=G_0^{-1}(1)-\epsilon$ for some $t\in[1/2,1)$.

  Let $t=G_0(G_0^{-1}(1)-\e)$.  First, we will show that $t\in(1/2,1)$ for sufficiently small $\epsilon>0$. 
 Since $g_0$ is log-concave, it is positive on $J(G_0)$,  and $G_0$ is strictly increasing on $J(G_0)$. Also, since $G_0$ is continuous, $G_0^{-1}$ is strictly increasing on $(0,1)$ by Fact \ref{fact: inverse common}. 
Therefore, $G_0^{-1}(1/2)<G_0^{-1}(1)$.   For sufficiently small $\epsilon>0$, $G_0^{-1}(1)-\e\in J(G_0)$. Moreover, we  can choose $\e$ so small such  that $G_0^{-1}(1)-\e\in(G_0^{-1}(1/2),G_0^{-1}(1))$,
  which implies $t=G_0(G_0^{-1}(1)-\e)>G_0(G_0^{-1}(1/2))\geq 1/2$. Here the first inequality follows from the strict monotonocity of $G_0$ on $J(G_0)$. 
The strict monotonocity of the continuous distribution function  $G_0$ also implies that   $t=G_0(G_0^{-1}(1)-\e)<G_0(G_0^{-1}(1))\leq 1$, where the last equality follows from Fact \ref{fact: inverse common} (i). Therefore, $t\in(1/2,1)$. 

Finally,
  \[G_0^{-1}(t)=G_0^{-1}\left(G_0(G_0^{-1}(1)-\e)\right)=G_0^{-1}(1)-\e\]
    by part (iii) of Fact \ref{fact: inverse common} since $G_0$ is strictly increasing on $J(G_0)$.
  Therefore, we have
 proved that  there exists $t\in(1/2,1)$ so that  $G_0^{-1}(t)=G_0^{-1}(1)-\epsilon$, and
 \begin{equation}
     \label{inlemma: xi in dom}
     \liminf_{N\to\infty} \PP(\widehat{\xi}_{2,k}< G_0^{-1}(t))>0.
 \end{equation}


 %

 However, because $\etan\to 0$,  $1-\etan\geq (1+t)/2$ for sufficiently large $n$, which yields
  $\widehat{\xi}_{2,k}\geq\tilde G_{m,n}^{-1}((1+t)/2)$.   Condition \ref{condition: hn basic} implies $\tilde G_{m,n}$  converges in probability to $G_0$ in total variation norm, which implies  $\tilde G_{m,n}$  uniformly converges to $G_0$ in probability. Therefore, by Fact~\ref{fact: consistency of the quantiles}, the quantile functions of $\tilde G_{m,n}$  converge pointwise in probability, leading to  $\tilde G_{m,n}^{-1}((1+t)/2)\to_ pG_0^{-1}((1+t)/2)$. However, $G_0^{-1}((1+t)/2)>G_0^{-1}(t)$, where the strict inequality follows because 
 $G_0^{-1}$ is strictly increasing on $(0,1)$. Thus  $\PP(\widehat{\xi}_{2,k}\geq G_0^{-1}(t))\to 1$, contradicting \eqref{inlemma: xi in dom}.  Therefore the proof follows by contradiction.

\textbf{Case (b):} Since  $\xib\to_p \infty$ does not hold,  there exists $M>0$ and a subsequence $(m_k,n_k)$ so that $\liminf_{k\to\infty} \PP(\widehat{\xi}_{2,k}<M)>0$. Let $\eta_k\equiv\eta_{m_k,n_k}$ denote the corresponding subsequence of $\etan$. Note that $\widehat{\xi}_{2,k}<M$ implies $\tilde G_{m,n}(\widehat{\xi}_{2,k})\leq \tilde G_{m,n}(M)$.  Since  $\tilde G_{m,n}$ has a density, it is continuous. Therefore  by an application of Fact \ref{fact: inverse common} (i), $\tilde G_{m,n}(\widehat{\xi}_{2,k})=1-\eta_{k}$. Hence, $\liminf_{k\to\infty} \PP(1-\eta_k<\tilde G_{m,n}(M))>0$. By definition of $\eta_k$, $1-\eta_k\to 1$ as $k\to\infty$. However,  $\tilde G_{m,n}(M)\to_p G_0(M)$ by Condition~\ref{condition: hn basic}, where $G_0(M)<1$ since  $G_0^{-1}(1)=\infty$. 
Thus  $\liminf_{k\to\infty} \PP(1-\eta_k<\tilde G_{m,n}(M))$ can not be positive. We have come to a contradiction, which concludes our proof. 
 
 \end{proof}
 
  \begin{lemma}\label{lemma: bound: xi n}
 Suppose $\hn\in\mathcal{LC}$ satisfies Condition~\ref{condition: hn basic} and Condition~\ref{cond: hellinger rate}.  Take $\etan=CN^{-2p/5}$ with $p$ as specified in Condition~\ref{cond: hellinger rate} and $C>0$. Then
 \[\xib\leq\frac{-\log 2+2p(\log N)/5}{\omega_{m, n}}+\tilde G_{m,n}^{-1}(1/2),\]
 \[\xia\geq-\frac{-\log 2+2p(\log N)/5}{\omega_{m, n}}+\tilde G_{m,n}^{-1}(1/2),\]
 where $\omega_{m, n}$ is as in Fact~\ref{fact: f grtr than F: base}. In particular,
 $|\xia|, |\xib|=O_p(\log N)$.
 \end{lemma}
 
 \begin{proof}[Proof of Lemma~\ref{lemma: bound: xi n}]
 We will consider the case of $\xib$ first. Recall that since $\hn\in\mathcal{LC}$, we have $\hn(\tilde G_{m,n}^{-1}(z))>0$ if $z\in(0,1)$. Define  $\alpha_{m, n}=\tilde G_{m,n}^{-1}(1/2)$ and $\alpha_0=G_0^{-1}(1/2)$.  Fact \ref{fact: bobkov big} implies that 
 \begin{align*}
     \xib=& \left\{\tilde{G}_{m, n}^{-1}(1-\etan) -\tilde{G}_{m, n}^{-1}(1/2) \right\}+\tilde{G}_{m, n}^{-1}(1/2)\\
     =&\dint_{1/2}^{1-\etan}\frac{dz}{\hn(\tilde G_{m,n}^{-1}(z))}+\alpha_{m, n}\\
     \leq& \dint_{1/2}^{1-\etan}\frac{dz}{\omega_{m, n}[1-\tilde G_{m,n}(\tilde G_{m,n}^{-1}(z))]}+\alpha_{m, n},
 \end{align*}
 where the last step follows from Fact~\ref{fact: f grtr than F: base} and $\omega_{m, n}$ is as in Fact~\ref{fact: f grtr than F: base}.
Fact \ref{fact: inverse common} (i) implies that $\tilde G_{m,n}(\tilde G_{m,n}^{-1}(z))=z$ for $z\in (0, 1)$ since $\tilde G_{m,n}$ is continuous. Therefore \[\xib\leq \frac{\log(1/2)-\log(\etan)}{\omega_{m, n}} + \alpha_{m, n}=\frac{-\log 2+2p(\log N)/5}{\omega_{m, n}}+\alpha_{m, n}.\]
Finally Fact~\ref{fact: f grtr than F: base} implies $\omega_{m, n}\to_p \omega_0$, Fact~\ref{fact: consistency of the quantiles} implies $\alpha_{m, n}\to_p \alpha_0$ and therefore $\xib = O_p(\log N)$. 

In a similar manner, we can prove
\[\xia\geq-\frac{-\log 2+2p(\log N)/5}{\omega_{m, n}}+\tilde G_{m,n}^{-1}(1/2),\]
which would imply $-\xia=O_p(\log N)$.  Since $\xia\leq \xib$, it follows that $|\xia|$, $|\xib|=O_p(\log N)$.

 \end{proof}

  \begin{lemma}\label{lemma: xi: tilde xi n}
 Suppose the assumptions of Lemma~\ref{lemma: bound: xi n} holds. Define \[\tilde \xi_1 = \tilde G_{m,n}^{-1}(\etan/2)\;,\; \tilde \xi_2=\tilde G_{m,n}^{-1}(1-\etan/2).\]
 Consider a sequence of non-negative random variables, denoted as $v_{m, n}$, satisfying $\PP(v_{m, n}< \etan/(2\|g_0\|_\infty))\to 1$. Then 
  \begin{align}\label{inlemma: yn set inclusion}
   \PP([\xia-v_{m, n},\xib+v_{m, n}]\subset [\tilde{\xi}_1,\tilde{\xi}_2])\to 1.  
 \end{align}
 \end{lemma}
 
 \begin{proof}[Proof of Lemma~\ref{lemma: xi: tilde xi n}]
 The proof follows from the proof of Lemma B.5 in \cite{laha2021adaptive} when $\xin$ is replaced with $\xib$, $-\xin$ is replaced with $\xia$ and $g_0(0)$ is replaced with $\lVert g_0\rVert_{\infty}$ in the original proof.
 
 \end{proof}

 \begin{lemma}\label{lemma: An inclusion}
  Under the assumptions of Lemma~\ref{lemma: bound: xi n},
   \[\PP\slb[\xia-|v_{m, n}|,\xib+|v_{m, n}|]\subset \mathrm{int}(\dom(\psi_0))\srb\to 1,\]
   provided $v_{m, n}=o_p(\etan)$.
 \end{lemma}
 
 \begin{proof}[Proof of Lemma~\ref{lemma: An inclusion}]

 Define $\tilde \xi_1=\tilde G_{m,n}^{-1}(\frac{\etan}{2}), \tilde \xi_2=\tilde G_{m,n}^{-1}(1-\frac{\etan}{2})$. By Lemma~\ref{lemma: xi: xi is in dop psi knot}, \[\PP\left([\tilde \xi_1, \tilde \xi_2]\subset \iint(\dom(\psi_0))\right)\to 1.\] 
 Since $v_{m, n}=o_p(\etan)$, the proof follows by invoking Lemma~\ref{lemma: xi: tilde xi n}.
 \end{proof}

 \subsubsection{\textbf{lemmas on  $\hn$ and $g_0$:}}
 \label{sec: lemmas on hn and g knot}
  As usual, we let  $\xia=\tilde G_{m, n}^{-1}(\etan), \xib=\tilde G_{m, n}^{-1}(1-\etan)$ for the lemmas in the current section. 
 \begin{lemma}
 \label{lemma: hn: lower bound  on hn}
 Suppose $\hn\in\mathcal {LC}$ is a density satisfying Condition~\ref{condition: hn basic} and and the corresponding distribution function is $\tilde G_{m, n}$. Then
 \begin{itemize}[topsep=0pt]
     \item [A.] $\sup_{x\in[\xia,\xib]}\hn(x)=O_p(1)$ and $\sup_{x\in[\xia,\xib]}\hn(x)^{-1}=O_p(\etan^{-1})$,
     \item[B.] If $\hln=\log\hn$, then $\sup_{x\in[\xia, \xib]}\hln(x)=O_p(1)$. In particular, when $\etan=CN^{-2p/5}$ with $p\in(0,1)$ and $C>0$, we have  $\sup_{x\in[\xia, \xib]}(-\hln(x))$ is $O_p(\log N)$.
 \end{itemize}
 \end{lemma}
 
 \begin{proof}[Proof of Lemma~\ref{lemma: hn: lower bound  on hn}]
 This lemma directly follows from Lemma B.7 of \cite{laha2021adaptive}, replacing $\xin$ with $\xib$ and $-\xin$ with $\xia$.
 \end{proof}

 \subsubsection{\textbf{Lemmas on $g_0$:}}
 \begin{lemma}
 \label{lemma: g:  bound on psi}
 Let $g_0\in\mathcal{LC}$ be a density satisfying Assumption~\ref{assump: L} and $\kappa$ is as in Assumption~\ref{assump: L}. Then there exist constants $C_1\equiv C_1(g_0)$ and $\kappa_1\equiv \kappa_1(\kappa, g_0)$ such that
 \[|\psi_0(x)|\leq C_1+\kappa_1 x^2\]
 for any $x\in\iint(\dom(\psi_0))$.
 
 Further, under the assumptions of Lemma~\ref{lemma: bound: xi n}, for $\etan=CN^{-p}$ with $p\in(0,1)$ and $C>0$ and $\xia=\tilde G_{m,n}^{-1}(\etan), \xib=\tilde G_{m,n}^{-1}(1-\etan)$, we have
 \[\sup_{x\in[\xia, \xib]}|\psi_0(x)|=O_p((\log N)^2).\]
 \end{lemma}
 
 \begin{proof}[Proof of Lemma~\ref{lemma: g:  bound on psi}]
 Since $g_0$ is log-concave, it is bounded above (see Fact \ref{fact: Lemma 1 of theory paper}). Therefore, there exists a constant $C$ depending on $g_0$ such that $\psi_0(x)\leq C$.
 Therefore, we just need to bound $-\psi_0(x)$ for $x\in\iint(\dom(\psi_0))$. Note that since $\psi_0$ is concave, $-\psi_0$ is convex. Assumption \ref{assump: L} implies that the directional derivatives of this convex function are Lipschitzian with constant $\kappa$. Fix $c\in \iint(\dom(\psi_0))$. Let $-\psi_0'(c)$ be the right derivative of $-\psi_0'$ at $c$, which exists due to the convexity of $-\psi_0$. Note that $-\psi_0'(c)$ is also a subgradient of  $-\psi_0$ at $c$ where the subgradient of a convex function is as defined in page 167 of \cite{hiriart2004}.  Therefore, using some algebraic manipulation, it can be shown that \citep[cf. pp. 241][]{hiriart2004}
 \[-\psi_0(x)\leq -\psi_0(c)-\psi_0'(c)(x-c)+\frac{\kappa}{2}(c-x)^2.\]
 Notice that $\psi_0(c)>-\infty$ because $c\in\dom(\psi_0)$.
The definition of $\mathcal C$ implies $\psi_0\in\mathcal C$ is proper. Therefore $\psi_0'(c)$ is finite for any $c\in\iint(\dom(\psi_0))$ by Theorem 23.4, pp. 271, of \cite{rockafellar}. 
 Observe that
 \[-\psi_0'(c)(x-c)\leq |\psi_0'(c)|\max\{|x-c|^2,1\}.\]
 Therefore, it follows that
 \[-\psi_0(x)\leq -\psi_0(c)+|\psi_0'(c)|+ (\kappa/2+ |\psi_0'(c)|)|x-c|^2.\]
 Hence, $|\psi_0(x)|\leq C_1+\kappa_1 x^2$ follows if we take 
\[C_1=\max\{C, -\psi_0(c)+|\psi_0'(c)|\}\quad\text{and}\quad\kappa_1=\kappa/2+ |\psi_0'(c)|.\]
 

 Now we note that $\PP([\xia, \xib] \subset \dom(\psi_0))\to 1$ by Lemma~\ref{lemma: xi: xi is in dop psi knot} and $|\xia|, |\xib|$ are $O_p(\log N)$ by Lemma~\ref{lemma: bound: xi n}. Hence, the rest of the proof follows.
 \end{proof}

 \begin{lemma}\label{lemma: bound: psi knot prime}
 Consider the setup of Lemma \ref{lemma: g:  bound on psi}.
Suppose $\etan$ is as chosen in Lemma~\ref{lemma: g: bound on psi} and $v_{m, n}$ is a sequence of non-negative random variables satisfying $v_{m, n}=o_p(\etan)$. Then the following holds:
 \[\sup_{x\in[\xia-v_{m, n},\xib+v_{m, n}]}|\psi_0'(x)|=O_p(\log(N)).\]
 \end{lemma}
 \begin{proof}[Proof of Lemma~\ref{lemma: bound: psi knot prime}]
 
 First note that $\psi_0$ is concave and therefore the maxima of $|\psi_0'|$ over any interval is attained at either endpoints. Let $c\in\iint(\dom(\psi_0))$.
  Using Lemma~\ref{lemma: An inclusion}, one gets that the probability of $[\xia-v_{m, n},\xib+v_{m, n}]\subset\dom(\psi_0)$ approaches one.  Hence,  Assumption~\ref{assump: L} implies that
 \begin{align}\label{inlemma: lemma bound psi knot prime}
 \begin{split}
     |\psi_0'(\xia-v_{m, n})|\leq |\psi_0'(c)|+\kappa(|\xia|+v_{m, n}+|c|),\\
     |\psi_0'(\xib+v_{m, n})|\leq |\psi_0'(c)|+\kappa(|\xib|+v_{m, n}+|c|)
 \end{split}
 \end{align}
 with probability approaching one. Since $\psi_0\in\mathcal C$ is proper and $c\in\iint(\dom(\psi_0))$, $|\psi_0'(c)|<\infty$ by Theorem 23.4 of \cite{rockafellar}.
 From Lemma~\ref{lemma: bound: xi n}, we get $|\xia|, |\xib| = O_p(\log N)$ and by our assumption on $v_{m,n}$, $|v_{m, n}| = o_p(1)$. Hence $|\psi_0'(\xia-v_{m, n})|, |\psi_0'(\xib+v_{m, n})|$ are $O_p(\log N)$ and the proof follows.
 \end{proof}

 \subsubsection{\textbf{Lemmas on $\hln$:}}
 \begin{lemma}\label{lemma: Laha_Nilanjana 31}
Let $\hn$ be a density satisfying satisfying Conditions~\ref{condition: hn basic} and \ref{cond: hellinger rate} and $p\in(0, 1)$ be as specified in Condition~\ref{cond: hellinger rate}. Let $a_{m, n}, c_{m, n}$ be sequences such that
\begin{equation}\label{inlemma: statement: os: nilanjana 31}
    \sup_{x\in[a_{m,n}, c_{m,n}]} (|\psi_0(x)|+|\hln(x)|)=O_p((\log N)^2).
\end{equation}
Then
\begin{align*}
    \dint_{a_{m,n}}^{c_{m,n}} (\hln(x)-\psi_0(x))^2 g_0(x)dx=&O_p((\log N)^cN^{-2p}),\\
     \dint_{a_{m,n}}^{c_{m,n}} (\hln(x)-\psi_0(x))^2 \hn(x)dx=&O_p((\log N)^c N^{-2p}),
\end{align*}
for some absolute constant $c>0$.

In particular, if  $\hn\in\mathcal{LC}$ and $a_{m,n}=\xia(\tilde G_{m,n})=\tilde G_{m,n}^{-1}(\etan)$, $c_{m,n}=\xib(\tilde G_{m,n})=\tilde G_{m,n}^{-1}(1-\etan)$ with $\etan=CN^{-2p/5}$ for some $C>0$,  then \eqref{inlemma: statement: os: nilanjana 31} holds.
 \end{lemma}

 \begin{proof}[Proof of Lemma~\ref{lemma: Laha_Nilanjana 31}]
 The first part of the proof follows from the proof of Lemma B.12 in \cite{laha2021adaptive} by replacing $-a_n$ and $a_n$  in that proof by $a_{m,n}$ and $c_{m,n}$, respectively.

To prove the the second part of the lemma, we show  that if  $\hn\in\mathcal{LC}$ and $a_{m,n}=\tilde G_{m,n}^{-1}(\etan), c_{m,n}=\tilde G_{m,n}^{-1}(1-\etan)$, then  \eqref{inlemma: statement: os: nilanjana 31} holds. 
In particular,  Lemma~\ref{lemma: g:  bound on psi} shows  that this $a_{m,n}, c_{m,n}$ satisfies
 \begin{align}\label{inlemma: os: Nilanjana 31}
     \sup_{x\in[a_{m,n}, c_{m,n}]} |\psi_0(x)|=O_p((\log N)^2).
 \end{align}
On the other hand, 
$\sup_{x\in [a_{m,n},c_{m,n}]}|\hln(x)| =O_p(\log N)$ by 
Lemma~\ref{lemma: hn: lower bound  on hn}. Hence the proof follows. 

 \end{proof}

  \subsubsection{\textbf{Lemmas on  $\hln'$:}}
  \begin{lemma}\label{FI Lemma: Lemma 1}
Suppose $\hn\in\mathcal{LC}$ satisfies Conditions~\ref{condition: hn basic} and \ref{cond: hellinger rate}. Let $\rho_{m, n} = \frac{\etan}{\log N}$ where $\etan = CN^{-2p/5}$ and $p\in(0, 1)$ is as in Condition~\ref{cond: hellinger rate}. Let $a_{m, n}$ and $c_{m, n}$ be random sequences satisfying  \eqref{inlemma: statement: os: nilanjana 31} so that $|a_{m, n}|, |c_{m, n}|= O_p(\log N)$. Further suppose $a_{m,n}$ and $c_{m,n}$ satisfy
\begin{equation}\label{inlemma: statement: os: suport inclusion}
 \PP\left( [a_{m,n}-\rho_{m,n},c_{m,n}+\rho_{m,n}]\subset \iint(\dom(\psi_0))\cap\iint(\dom(\hln))\right )\to 1,
 \end{equation}
 \begin{flalign}\label{inlemma: statement: os: FI: an xin}
 \text{and}\quad \PP(\tilde G_{m,n}(a_{m,n})>\etan/4, 1-\tilde G_{m,n}(c_{m,n})>\etan/4)\to 1. 
\end{flalign}

 Then, for any density $\nu_{m,n}$ with $\|\nu_{m,n}\|_\infty=O_p(1)$, there exists a constant $c>0$ so that 
 \[ \dint_{a_{m,n}}^{c_{m,n}}\slb \hln'(z)-\psi_0'(z)\srb^2dz=O_p( (\log N)^c N^{-4p/5}),\]
  \[\dint_{a_{m,n}}^{c_{m,n}} \slb \hln'(z)-\psi'_0(z)\srb^2 \nu_{m,n}(z)dz=O_p( (\log N)^c N^{-4p/5})).\]
  In particular, if
  $\hn\in\mathcal{LC}$, $a_{m,n}=\xia(\tilde G_{m,n})=\tilde G_{m,n}^{-1}(\etan)$, and $c_{m,n}=\xib(\tilde G_{m,n})=\tilde G_{m,n}^{-1}(1-\etan)$, then the lemma holds.
 \end{lemma}
 
 \begin{proof}[Proof of Lemma~\ref{FI Lemma: Lemma 1}]

The proof of the first part follows from  the proof of Lemma B.13 in \cite{laha2021adaptive} by replacing $-a_n$ and $a_n$ in the original proof by $a_{m,n}$ and $c_{m,n}$, respectively, and hence skipped.

For the second part of Lemma \ref{FI Lemma: Lemma 1}, 
 consider  $\hn\in\mathcal{LC}$, $a_{m,n}=\xia$, and  $c_{m,n}=\xib$. They satisfy \eqref{inlemma: statement: os: nilanjana 31} by Lemma~\ref{lemma: Laha_Nilanjana 31}. Also Lemma~\ref{lemma: bound: xi n} implies that $|\xia|,\ |\xib|=O_p(\log N)$ Noting $\rho_{m,n}=o(\etan)$, \eqref{inlemma: statement: os: suport inclusion} follows from  Lemma~\ref{lemma: xi: tilde xi n}
and Lemma~\ref{lemma: An inclusion}.
 Since \eqref{inlemma: statement: os: FI: an xin} trivially holds,  second part of Lemma~\ref{FI Lemma: Lemma 1} also follows.
 \end{proof}

 \begin{lemma}\label{lemma: consistency of FI: 2}
 Let $\hn$ be a density satisfying Condition~\ref{condition: hn basic} and Condition~\ref{cond: hellinger rate}. Let $a_{m,n}$  and $c_{m,n}$ be random sequences  satisfying 
 \begin{align}
\label{inlemma: statement: FI: an}
    |a_{m,n}|=O_p(\log N),\quad 
\PP\slb a_{m,n}\in\iint(\dom(\psi_0))\srb\to 1,\quad\text{and}\quad
G_0( a_{m,n})\to_p 0,\\
 |c_{m,n}|=O_p(\log N),\quad 
\PP\slb c_{m,n}\in\iint(\dom(\psi_0))\srb\to 1,\quad\text{and}\quad
 G_0(c_{m,n})\to_p 1.
 \end{align}
 Suppose there exists an absolute constant $c>0$ so that $\hn, a_{m,n}, c_{m,n}$ satisfy
 \begin{align}\label{inlemma: statement: FI}
     \dint_{a_{m,n}}^{c_{m,n}}(\hln'(z)-\psi_0'(z))^2\hn(z)dz=O_p((\log N)^c N^{-4p/5}),
 \end{align}
where $p$ is as in Condition~\ref{cond: hellinger rate}.
Then 
 $\int_{a_{m,n}}^{c_{m,n}}\hln'(z)^2\hn(z)dz\to_p \Io$. In particular, if $\hn\in\mathcal{LC}$, $a_{m,n}=\xia(\tilde G_{m,n})=\tilde G_{m,n}^{-1}(\etan)$, and $c_{m,n}=\xib(\tilde G_{m,n})=\tilde G_{m,n}^{-1}(1-\etan)$ where $\etan=C N^{-2p/5}$ for some $C>0$, then the statement of the lemma holds.
\end{lemma} 
 
 \begin{proof}[Proof of Lemma~\ref{lemma: consistency of FI: 2}]
First notice that
 \begin{align*}
  \MoveEqLeft    \dint_{a_{m,n}}^{c_{m,n}}\slb\hln'(z)^2\hn(z)-\psi_0'(z)^2g_0(z)\srb dz\\
  = &\ \dint_{a_{m,n}}^{c_{m,n}}\slb(\hln'(z)-\psi_0'(z))^2\hn(z)+2\psi_0'(z)\hln'(z)\hn(z)-\psi_0'(z)^2\hn(z)-\psi_0'(z)^2g_0(z)\srb dz\\
  =&\ \underbrace{ \dint_{a_{m,n}}^{c_{m,n}}(\hln'(z)-\psi_0'(z))^2\hn(z)dz}_{A_1}+ 2\underbrace{\dint_{a_{m,n}}^{c_{m,n}}\psi_0'(z)(\hln'(z)-\psi_0'(z))\hn(z)dz}_{A_2}\\
  &\ +\underbrace{\dint_{a_{m,n}}^{c_{m,n}}\psi_0'(z)^2(\hn(z)-g_0(z))dz}_{A_3}.
 \end{align*}
 Using \eqref{inlemma: statement: FI} from the assumptions, we have
 $A_1=O_p((\log N)^c N^{-4p/5})$ and therefore $A_1$ is $o_p(1)$.

Now we show that $A_2$ is $o_p(1)$. Note that,
\begin{align*}
|A_2|
   \leq &\ \max\{|\psi_0'(a_{m,n})|,|\psi_0'(c_{m,n})|\}\dint_{a_{m,n}}^{c_{m,n}}|\hln'(z)-\psi_0'(z)|\hn(z)dz \\
  \leq &\ \max\{|\psi_0'(a_{m,n})|,|\psi_0'(c_{m,n})|\}\lb\dint_{a_{m,n}}^{c_{m,n}}(\hln'(z)-\psi_0'(z))^2\hn(z)dz \rb^{1/2},
\end{align*}
where the first step is due to the fact that $|\psi_0'|$ attains its supremum over an interval in either of the end-points,  and the last step is due to the Cauchy-Schwarz inequality. Therefore, $|A_2|\leq (|\psi_0'(a_{m,n})|+|\psi_0'(a_{m,n})|)\sqrt{A_1}$.

Recall that $\PP(a_{m, n}\in \iint(\dom(\psi_0)))\to 1$ under the setup of the current lemma. Fix  $x_0\in \iint(\dom(\psi_0))$. Therefore, By Assumption~\ref{assump: L}, with probability converging to one, 
\[|\psi_0'(a_{m,n})|\leq |\psi_0'(x_0)|+\kappa |a_{m,n}-x_0| \leq |\psi_0'(x_0)|+\kappa |a_{m,n}| + \kappa|x_0|,\]
which is $O_p(\log N)$ because $|\ps_0'(x_0)|<\infty$ since $x_0\in \iint(dom(\psi_0))$ \citep[cf. Theorem 23.4 of][]{rockafellar} and $|a_{m,n}|=O_p(\log N)$ under the setup of the current lemma. Hence, we have showed that $|\psi_0'(a_{m,n})| = O_p(\log N)$. Similarly one can show that $|\psi_0'(c_{m,n})| = O_p(\log N)$. Therefore, 
 \[|A_2|=O_p(\log N) \sqrt{A_1}=O_p((\log N)^{c/2+1}N^{-2p/5})=o_p(1).\]
 
To bound $|A_3|$, we again use the facts that $|\psi_0'|$ attains maxima at either endpoints on an interval and $|\psi_0'(a_{m,n})|, |\psi_0'(c_{m,n})|=O_p(\log N)$ to obtain 
\begin{align*}
    |A_3| &\le \max\{|\psi_0'(a_{m,n})|,|\psi_0'(c_{m,n})|\}^2\dint_{a_{m, n}}^{c_{m, n}}\left| \hn(z) - g_0(z)\right|dz\\
    & \leq  O_p((\log N)^2) d_{TV}(\hn, g_0)\\
    & \stackrel{(a)}{\leq} 
     O_p((\log N)^2)\H(\hn,g_0)\stackrel{(b)}{=} O_p((\log N)^2 N^{-p})=o_p(1),
 \end{align*}
 where (a) is by Fact~\ref{fact: dTV and hellinger} and (b) is due to Condition~\ref{cond: hellinger rate}.
Combining $A_1, A_2, A_3$, we obtain that 
\begin{align}
\label{inlemma: FI: differences}
    & \dint_{a_{m,n}}^{c_{m,n}}\slb\hln'(z)^2\hn(z)-\psi_0'(z)^2g_0(z)\srb dz=o_p(1).
\end{align}
Therefore, it remains to show that
\begin{align}
    \label{inlemma: FI: penultymate}
    \dint_{a_{m,n}}^{c_{m,n}}\ps_0'(z)^2g_0(z)dz\to_p \Io,
\end{align}
which is equivalent to
\[\dint_{-\infty}^{a_{m,n}}\ps_0'(z)^2g_0(z)dz\to_p 0\quad\text{and}\quad \dint_{c_{m.n}}^{\infty}\ps_0'(z)^2g_0(z)dz\to_p 0.\]
To this end, it suffices to show  the first convergence because the second convergence will follow similarly. 
Since $\Io<\infty$,  Fact~\ref{fact: condition for integrability}  implies that given any $\e>0$, we can find a $\sigma>0$ so small such that 
 $\int_{\mathcal B}\ps_0'^2(x)g_0(x)dx<\e$ for any $P_0$-measurable set $\mathcal{B}\subset\RR$ provided $P_0(\mathcal B)<\sigma$. 
 Thus it is enough to  show that $\int_{-\infty}^{a_{m,n}}g_0(z)dz=G_0(a_{m,n})=o_p(1)$. However, by our assumption on $a_{m,n}$, $G_0(a_{m,n})\to_p 0$. Hence, it follows that $\int_{-\infty}^{a_{m,n}}\ps_0'(z)^2g_0(z)dz$ is $o_p(1)$. 
 Therefore, we have established \eqref{inlemma: FI: penultymate}, which, combined with \eqref{inlemma: FI: differences}, finishes the proof of the first part of the current lemma.


 The proof of the second part of Lemma~\ref{lemma: consistency of FI: 2} follows by observing that $|\xia|$ and $|\xib|$ are of the order $O_p(\log n)$ according to Lemma~\ref{lemma: bound: xi n}. Additionally, with probability converging to one, we have $\xia$ and $\xib\in\iint(\dom(\psi_0))$ as established in Lemma~\ref{lemma: xi: xi is in dop psi knot}. Furthermore, Lemma~\ref{lemma: xi: xi goes to 1} implies  that $\xia$ tends to $G_0^{-1}(0)$ and $\xib$ tends to $G_0^{-1}(1)$ in probability. Lastly,  $\hn$ satisfies  \eqref{inlemma: statement: FI} by Lemma~\ref{FI Lemma: Lemma 1}. Hence, the proof follows from the first part of the lemma. 
 \end{proof}

\begin{lemma}\label{lemma: tilde psi prime bound}
    Consider the setup of Theorem~\ref{theorem: main theorem}. Suppose $\etan = CN^{-2p/5}$ where $C>0$ is a constant and $p$ is as in  Condition~\ref{cond: hellinger rate}. Let $v_{m, n}$ be a potentially stochastic sequence of non-negative variables satisfying $\PP(v_{m, n}\leq \etan/(2\|g_0\|_\infty))\to 1$. Then the following assertion  holds:
    \[ \sup_{x\in[\xia-v_{m, n},\xib+v_{m, n}]}|\hln'(x)|=O_p(\etan^{-1/2})=O_p(N^{p/5}).\]
\end{lemma}
\begin{proof}[Proof of Lemma~\ref{lemma: tilde psi prime bound}]
Consider $q\in (0, 1/4)$. We have already mentioned that since $\hn$ is log-concave, it is positive on $J(\tilde G_{m,n})$. Therefore, using Fact \ref{fact: bobkov big}, we obtain that 
\[
\dint_{\tilde{G}^{-1}_{m,n}(q/2) }^{\tilde{G}^{-1}_{m,n}(q)} \hln'(x)^2\hn(x)dx=\dint_{q/2}^q\hln'(\tilde{G}^{-1}_{m,n}(z))^2dz.
\]
 Hence, we have
\begin{align*}
    \frac{q}{2}\inf_{\eta\in [\frac{q}{2}, q]}\hln'(\tilde{G}^{-1}_{m,n}(\eta))^2 \leq& \dint_{\tilde{G}^{-1}_{m,n}(q/2) }^{\tilde{G}^{-1}_{m,n}(q)} \hln'(x)^2\hn(x)dx\\
     \stackrel{(a)}{\leq} &\dint_{\tilde{G}^{-1}_{m,n}(q/2) }^{\tilde{G}^{-1}_{m,n}(1-q/2)} \hln'(x)^2\hn(x)dx,
\end{align*}
where (a) holds since $q<1-q/2$. Therefore,
\begin{align*}
    \inf_{\eta\in [\frac{q}{2}, q]} \left|\hln'(\tilde{G}^{-1}_{m,n}(\eta))\right|^2 \leq \frac{2}{q}\dint_{\tilde{G}^{-1}_{m,n}(q/2) }^{\tilde{G}^{-1}_{m,n}(1-q/2)} \hln'(x)^2\hn(x)dx.
\end{align*}
Similarly
\begin{align*}
    \inf_{\eta\in [q, \frac{3q}{2}]} \left|\hln'(\tilde{G}^{-1}_{m,n}(\eta))\right|^2 \leq \frac{2}{q}\dint_{\tilde{G}^{-1}_{m,n}(q/2) }^{\tilde{G}^{-1}_{m,n}(1-q/2)} \hln'(x)^2\hn(x)dx.
\end{align*}
Now take $q = \frac{\etan}{2}$. Then
\[\frac{2}{q}\dint_{\tilde{G}^{-1}_{m,n}(q/2) }^{\tilde{G}^{-1}_{m,n}(1-q/2)} \hln'(x)^2\hn(x)dx = O_p(\etan^{-1}),\]
since the integral converges in probability to $\Io$ by Lemma~\ref{lemma: consistency of FI: 2}. Therefore 
\[\inf_{\eta\in [\frac{q}{2}, q]} \left|\hln'(\tilde{G}^{-1}_{m,n}(\eta))\right|, \inf_{\eta\in [q, \frac{3q}{2}]} \left|\hln'(\tilde{G}^{-1}_{m,n}(\eta))\right| = O_p(\etan^{-1/2}) = O_p(N^{p/5}).\]
Recall that $\xia = \tilde G_{m,n}^{-1}(\etan), \xib = \tilde G_{m,n}^{-1}(1-\etan)$. Define $\tilde \xi_1 = \tilde G_{m,n}^{-1}(\etan/2), \tilde \xi_2 = \tilde G_{m,n}^{-1}(1-\etan/2)$. Also note that since $\hln'$ is non-increasing, for any interval $[a, b]$ such that $\hln'(a)\ge \hln'(b)\ge 0$ (or $0\ge\hln'(a)\ge \hln'(b)$), $\inf_{[a, b]}|\hln'(x)|$ is attained at $b$ (or $a$). Now we consider two scenarios:
\begin{enumerate}
    \item \textbf{$\hln'\left(\tilde \xi_1\right)\geq 0$:} Because $\xia=\tilde G_{m,n}^{-1}(\etan/2)$, it follows that  $\hln'\left(\tilde G_{m,n}^{-1}(\etan/4)\right) \geq \hln'\left(\tilde \xi_1\right)\geq 0$ and
    \[ \inf_{\eta\in [\frac{q}{2}, q]} \left|\hln'(\tilde{G}^{-1}_{m,n}(\eta))\right| = \left|\hln'(\tilde \xi_1)\right|,\]
    and hence $\left|\hln'(\tilde \xi_1)\right|$ is $O_p(N^{p/5})$.
    \item \textbf{$\hln'\left(\tilde \xi_1\right) < 0$:} Then $0 >\hln'\left(\tilde \xi_1\right) \geq \hln'\left(G_{m,n}^{-1}(3q/2)\right) $ and hence
    \[ \inf_{\eta\in [q, \frac{3q}{2}]} \left|\hln'(\tilde{G}^{-1}_{m,n}(\eta))\right| = \left|\hln'\left(\tilde{\xi_1}\right)\right| =O_p(N^{p/5}).\]
\end{enumerate}
Combining the cases, we get $\left|\hln'\left(\tilde{\xi_1}\right)\right| = O_p(N^{p/5})$. Similarly one can also find $\left|\hln'\left(\tilde{\xi_2}\right)\right| = O_p(N^{p/5})$. Since $\hln'$ is non-increasing, its supremum or infrimum over an interval is attained at one of the end-points. Therefore $\sup_{x\in [\tilde\xi_1, \tilde \xi_2]}|\hln'(x)| = \max\left\{ \left|\hln'\left( \tilde\xi_1 \right)\right|, \left|\hln'\left( \tilde\xi_2 \right)\right|\right\} = O_p(N^{p/5})$.
The rest of the proof follows from Lemma~\ref{lemma: xi: tilde xi n}.
\end{proof}

 \begin{lemma}\label{Lemma: L2 norm of hn}
 Under the set up of Theorem~\ref{theorem: main theorem}, $h_{m, n}$ defined in \eqref{inlemma: def: main: hn} satisfies
 \[\|h_{m,n}\|_{P_0, 2}^2=O_p(N^{-4p/5}(\log N)^{3}).\]
\end{lemma}

\begin{proof}[Proof of Lemma~\ref{Lemma: L2 norm of hn}]
This proof follows from the proof  of Lemma B.16 of \cite{laha2021adaptive}, and hence skipped.
\end{proof}

 \section{Technical facts}
\label{sec: technical facts}

Below we  list some facts which have been used repeatedly in our proofs. We start with some facts on the Hellinger distance.

\begin{fact}\label{fact: dTV and hellinger}
Let $F$ and $G$ be two distribution functions with densities $f$ and $g$, respectively. Then $d_{TV}(F,G)\leq \sqrt{2}\H(f,g)$ and $\H^2(f,g)\leq d_{TV}(F,G)$.
\end{fact}

\begin{fact}[Fact 15 of \cite{laha2021adaptive}]\label{fact: helli fknot}
 For a log-concave density $g_0$ with $\Io<\infty$, $\H(g_0(\mathord{\cdot}+y),g_0)=O(|y|)$.
 \end{fact}

The next facts pertain to the convergence of random variables. 
\begin{fact}[Theorem 5.7 (ii) of \cite{shorack2000}]\label{fact: convergence in probability to convergence almost surely}
Consider a random sequence $\{X_n\}_{n\geq 1}$ and a random variable $X$. If $X_n\to_p X$, then there exists a  subsequence $n_k$ for which $X_{n_k}\as X$. 
\end{fact}
\begin{fact}[Theorem 5.7 (vii) of \cite{shorack2000}]
\label{fact: Shorack}
Consider a sequence of random variables $\{X_n\}_{n\geq 1}$ and a random variable $X$. Then $X_n\to_p X$ if and only if  for every subsequence $\{n_k\}_{k\geq 1}$, there is a further subsequence $\{n_{r}\}_{r\geq 1}$ such that $X_{n_r}\as X$.
\end{fact}

\begin{fact}\label{fact: consistency of the quantiles}
 Let $(F_n)_{n\geq 1}$ be a sequence of distribution functions and $F$ be distribution function such that $\|F_n-F\|_{\infty}\to 0$. Let $F$ have a density, denoted by $f$. Then for $t\in\iint(\supp(f))$, we have $|F_n^{-1}(t)-F^{-1}(t)|\to 0$.
 \end{fact}
 
 \begin{proof}
The proof follows from Lemma A.5 of \cite{bobkovbig}.
 \end{proof}

  

  The next fact is an empirical process result.
  \begin{fact}
 \label{fact: empirical process of m-p's}
 Assume that $\mathcal F$ is a class of measurable functions $h$ with $\int h^2dP_0<\epsilon^2$ where $\|h\|_\infty\leq M$ for some positive constant $M$. Then 
 \[\E\|\mathbb{G}_n\|_{\mathcal F}\lesssim J_{[\  ]}(\epsilon,\mathcal F, L_2(P_0))\lb 1+\frac{M J_{[\  ]}(\epsilon,\mathcal F, L_2(P_0))}{\epsilon^2\sqn}\rb,\]
 where
 \[J_{[\  ]}(\epsilon,\mathcal F, L_2(P_0))=\dint_0^\epsilon\sqrt{ 1+\log N_{[\  ]}(\epsilon',\mathcal F,L_2(P_0))}d\epsilon'.\]
 \end{fact}
 
 \begin{proof}
 The proof follows from Lemma 3.4.2 of \cite{vdv}.
 \end{proof}

 
 

 
 
 The following is a property of integrable functions.
 \begin{fact}[Exercise 16.18, pp. 223 of \cite{billingsley}]
 \label{fact: condition for integrability}
 Let $\PP$ be a finite measure on $\RR$. Suppose $\int_{\RR} |h| dP<\infty$ for a measurable function $h$. Then, for any $\epsilon>0$, there exists $\sigma>0$ such that $\int_{\mathcal B}|h|dP<\epsilon$ for any $\PP$-measurable set $\mathcal B$ with $\PP(\mathcal B)<\sigma$.
 \end{fact}

Finally, we list some useful facts regarding log-concave densities.

 
 

  \begin{fact}[Fact 3 of \cite{laha2021adaptive}]\label{fact: f grtr than F: base}
 Let $\hn$ be a log-concave density that satisfies Condition~\ref{condition: hn basic}.   Then there exists a constant $\omega_0>0$ depending only on $g_0$ and a random sequence $\omega_n\geq 0$ with $\omega_n\to_p \omega_0>0$ such that, 
 \[\hn(x)\geq \omega_n\min(\tilde{G}_n(x), 1-\tilde{G}_n(x)), \]
 for any $x\in\RR$.
 \end{fact}

\end{document}